\documentclass[11pt,a4paper,reqno]{amsart}%
\usepackage{amsthm,amsmath,amsfonts,amssymb,xcolor,amsxtra,appendix,bookmark,dsfont,bm,mathrsfs}

\usepackage[latin1]{inputenc}
\usepackage{color} 
\usepackage{amssymb}\usepackage{graphicx}
\usepackage{tikz}
\usepackage{hyperref}
\theoremstyle{plain}
\newtheorem{theorem}{Theorem}
\newtheorem{lemma}[theorem]{Lemma}

\newtheorem{proposition}[theorem]{Proposition}

\newtheorem{definition}{Definition}

\theoremstyle{remark}
\newtheorem{remark}[theorem]{Remark}

\allowdisplaybreaks

\title[Entropy Structures and Long-Time Relaxation for 3-Wave Kinetic Equations]{Entropy Structures and Long-Time Relaxation for 3-Wave Kinetic Equations}

\author[G. Staffilani]{Gigliola Staffilani
}
\address{Department of Mathematics, Massachusetts Institute of Technology, Cambridge, MA 02139, USA}
\email{gigliola@math.mit.edu} 
\thanks{G.S. is  funded in part by  the NSF grants DMS-2052651, DMS-2306378 and the Simons Foundation through the Simons Collaboration on Wave Turbulence.}

\author[M.-B. Tran]{Minh-Binh Tran*}
\address{Department of Mathematics, Texas A\&M University, College Station, TX 77843, USA}
\email{minhbinh@tamu.edu} 
\thanks{ M.-B. T. is  funded in part by      NSF CAREER  DMS-2303146, and NSF Grants DMS-2204795, DMS-2305523,  DMS-2306379.}

\begin{document}
\date{\today}

\begin{abstract} 
We establish a new class of entropy structures for \(3\)-wave kinetic equations
with a broad family of interaction weights. Unlike the classical entropies
arising from detailed balance, these estimates are generated by a one-sided
algebraic balance condition encoded in the interaction weights.
To the best of our knowledge, this family of entropy estimates has not
previously appeared in the physics literature on wave turbulence. These
estimates form the central  key
ingredient in the construction of global weak \(L^1_{\mathrm{loc}}\) solutions.
We also prove a long-time rigidity result, showing that the solutions obtained
by this entropy compactness method relax locally to the zero equilibrium as
\(t\to\infty\).
\end{abstract}

\maketitle

 \tableofcontents

\section{Introduction}\label{intro} 

Over the past several decades, wave turbulence theory has become a fundamental
framework for describing the long-time statistical behavior of weakly nonlinear
dispersive wave systems. Its range of applications is extensive, including
inertial waves in rotating fluids, Alfv\'en wave turbulence in the solar wind,
wave propagation in magnetized plasmas, and wave phenomena arising in plasma
and fusion physics, among many others. The origins of the theory go back to the
pioneering work of Peierls \cite{Peierls:1993:BRK}, followed by foundational
contributions of Benney and Saffman \cite{benney1966nonlinear}, Zakharov and
Falkovich \cite{zakharov1967weak}, Benney and Newell \cite{benney1969random},
and Hasselmann \cite{hasselmann1962non,hasselmann1974spectral}. These works
led to the systematic formulation of wave kinetic equations, most notably the
classical 3-wave and 4-wave kinetic equations, which describe the
resonant redistribution of energy among weakly interacting wave modes. More
recently, rigorous mathematical derivations of wave kinetic equations  have been obtained in a sequence of breakthrough
works by Deng and Hani
\cite{deng2019derivation,deng2021propagation,deng2023long,deng2023,
	deng2021full}. For a comprehensive discussion of the physical theory of wave
turbulence and its applications, we refer the reader to
\cite{Nazarenko:2011:WT,PomeauBinh,zakharov2012kolmogorov}.

In this work, we consider the following isotropic  3-wave kinetic equation
\begin{equation}\label{3wave}
	\begin{aligned}
		\partial_t f(t,\omega) \; &=\; \mathcal{Q}[f](t,\omega),
		\qquad (t,\omega)\in(0,\infty)\times[0,\infty),\\
		f(0,\omega) \; &=\; f_0(\omega),
		\qquad \omega\in[0,\infty),
	\end{aligned}
\end{equation}
where the collision operator $\mathcal{Q}$ is defined by
\begin{equation}\label{3waveOp}
	\begin{aligned}
		\mathcal{Q}[f](t,\omega)
		:= {} &
		\iint_{[0,\infty)^2}
		\mathrm d\omega_1\, \mathrm d\omega_2\,
		\delta(\omega-\omega_1-\omega_2)\,
		U(\omega_1)U(\omega_2)
		\Big(
		f_1 f_2
		- f f_1
		- f f_2
		\Big) \\
		&\;
		- 2
		\iint_{[0,\infty)^2}
		\mathrm d\omega_1\, \mathrm d\omega_2\,
		\delta(\omega_2-\omega-\omega_1)\,
		U(\omega_1)U(\omega_2)
		\Big(
		f_1 f
		- f_1 f_2
		- f f_2
		\Big).
	\end{aligned}
\end{equation}

In the above expression, we use the shorthand notation
\[
f = f(\omega), \qquad
f_1 = f(\omega_1), \qquad
f_2 = f(\omega_2).
\]

3-wave kinetic equations arise in many physical contexts and have been
studied from a variety of analytical viewpoints. Representative examples
include phonon interactions in anharmonic crystal lattices
\cite{CraciunBinh,EscobedoBinh,GambaSmithBinh,tran2020reaction}, capillary wave
dynamics \cite{nguyen2017quantum},  stratified oceanic flows
\cite{GambaSmithBinh,kim2025wave}, and kinetic descriptions of Bose--Einstein
condensates
\cite{AlonsoGambaBinh,cortes2020system,EPV,escobedo2023linearized1,
	escobedo2023linearized,ToanBinh,nguyen2017quantum,Binh1}. Numerical studies of 3-wave kinetic equations have also  appeared recently in \cite{banks2025new,das2024numerical,walton2023numerical,walton2024numerical,walton2022deep}.

In \cite{zakharov2012kolmogorov}, equation~\eqref{3wave} is studied with the explicit acoustic interaction weight
\begin{equation}\label{acoustic}
	U(x)=C_U |x|^{d-1}.
\end{equation}
Here \(C_U>0\) is an explicit dimensional constant, normalized to \(1\) for notational simplicity, and \(d=2\) or \(d=3\) denotes the spatial dimension. The corresponding model \eqref{3wave}--\eqref{acoustic} is the principal canonical 3-wave kinetic model treated in detail in \cite{zakharov2012kolmogorov}; see, in particular, Equations~(3.2.3), (3.4.10), (4.3.16), (5.1.8), and (5.1.26) therein. The monograph \cite{zakharov2012kolmogorov} analyzes several fundamental aspects of
\eqref{3wave}--\eqref{acoustic} as a canonical physical example of a 3-wave kinetic
model, including the derivation and structure of stationary Kolmogorov--Zakharov
spectra in Chapter~3, stability issues in Chapter~4, and physical applications in
Chapter~5. In fact, \eqref{3wave}--\eqref{acoustic} is the only 3-wave
model treated in such depth in \cite{zakharov2012kolmogorov}.

The main novelty of the present paper is the identification of new entropy structures for the 3-wave kinetic equation~\eqref{3wave}. To the best of our knowledge, these structures have not previously appeared in the physics literature on wave turbulence. They apply to a broad class of interaction weights \(U\), including the acoustic weights associated with the canonical model \eqref{3wave}--\eqref{acoustic} treated in \cite{zakharov2012kolmogorov}. Thus, the present work reveals  previously unnoticed entropy structures for this class of kinetic equations. In particular, it provides a rare instance in which a new entropy mechanism is uncovered through mathematical analysis, rather than being inherited from the established physics literature (see Remark \ref{Remark:entrpy} for an intuitive explanation of the entropy).

These new entropy structures constitute the central a priori estimates of the
paper. They provide the compactness, stability, and tightness needed to
construct global \(L^1_{\mathrm{loc}}\) weak solutions to \eqref{3wave}. 
This is a significant
feature: for nonlinear kinetic equations, the construction of global-in-time
weak solutions is often highly delicate, and the available theories frequently
lead only to very weak notions of solution, such as renormalized solutions, or
to Radon measure-valued solutions.

Our construction proceeds through a carefully designed approximation scheme. We
first introduce  truncated equations whose collision operators retain enough of
the symmetry of the original 3-wave interaction to preserve the new entropy
estimates uniformly at the approximate level. We then pass to the limit in
spaces of Radon measures. The limit obtained in this way is, a
priori, only a nonnegative Radon measure. The entropy estimates  are then used in
a crucial way to identify its structure: the limit can be decomposed into an
explicit atomic part and a regular part belonging to \(L^1_{\mathrm{loc}}\) (see Proposition \ref{Prop:Existence:U=x:Hphi}, Proposition \ref{prop:global-Ux2-Hphi}, Remark \ref{Remark:EntropyRole}
 and Remark \ref{Remark:Atom} for detailed discussions on the compactness argument and the atomic part).
This decomposition is a surprising observation and one of the key points of the
analysis. Indeed, although the limit is obtained only as a Radon
measure, the entropy bounds rule out any additional singular component beyond
the explicit atom, leaving a genuinely \(L^1_{\mathrm{loc}}\) density after the
atomic part is removed. After removing the atomic component, the regular part is
shown to solve the original, non-truncated equation globally in time in the
\(L^1_{\mathrm{loc}}\) weak sense.

We further show that the solutions constructed in this paper exhibit a striking
long-time behavior. Namely,
\[
f(t,\cdot)\longrightarrow 0
\qquad\text{in } L^1_{\mathrm{loc}}((0,\infty))
\qquad\text{as } t\to\infty .
\]
This conclusion is rather unexpected. Indeed, one might naturally anticipate convergence, at least along
subsequences or after suitable renormalization, toward a nontrivial stationary
spectrum, such as one of the Kolmogorov--Zakharov spectra associated with the
model; see Remark~\ref{Remark:KZ} for a discussion of this point.

Thus, the entropy structure uncovered here plays a dual role. It provides the
central compactness mechanism needed for the construction of global weak
solutions, and it also imposes a strong rigidity on the asymptotic dynamics.
Within the class of global \(L^1_{\mathrm{loc}}\) weak solutions generated by
finite-entropy initial data, no nonzero local-in-frequency equilibrium can arise
as a long-time limit. The only possible asymptotic profile in the local
\(L^1\)-topology is the zero equilibrium. This should not be interpreted as the
absence of energy transfer; rather, it reflects an energy-cascade mechanism in
which the finite-frequency density vanishes because energy is transported toward
\(\omega=\infty\). We refer the reader to Remark~\ref{Remark:Cascade} for a more
detailed discussion of this interpretation.

We do not address uniqueness of \(L^1_{\mathrm{loc}}\) weak solutions in the
present work. For weak solutions with only local \(L^1\) control, uniqueness is
typically beyond the reach of the presently available theory for nonlinear
kinetic equations at low regularity. This limitation is not specific to the
present \(3\)-wave model. Closely related difficulties occur, for example, for
the homogeneous Boltzmann equation, the Nordheim equation for bosons, and the
\(4\)-wave kinetic equation, where uniqueness questions often remain open or
require substantially stronger solution classes; see, for instance,
\cite{EscobedoVelazquez:2015:FTB,EscobedoVelazquez:2015:OTT,Lu2014_RegularityCondensation}.
Accordingly, the focus of the present paper is on the entropy structure, the
compactness mechanism, the construction of global \(L^1_{\mathrm{loc}}\) weak
solutions, and their long-time relaxation, while uniqueness is left outside the
scope of this work.

We finally recall several analytical developments related to the 4-wave
kinetic equation and to closely connected wave kinetic models.
An important theme in the analysis of the 4-wave kinetic equation is the
behavior of Kolmogorov--Zakharov spectra. These spectra are stationary solutions
of the equation and play a central role in the description of energy cascades
in wave turbulence. Their stability properties, together with cascade dynamics
and near-equilibrium stability for the four-wave kinetic equation, have been
studied in \cite{collot2024stability} and
\cite{menegaki20222,escobedo2024instability}.
Further analytical results include local well-posedness and ill-posedness
theory for 4-wave kinetic equations, as well as the study of the associated
hierarchy in polynomially weighted \(L^\infty\) spaces for inhomogeneous
4-wave kinetic equations; see
\cite{ampatzoglou2025ill,ampatzoglou2025optimal,GermainIonescuTran,
	ampatzoglou2025inhomogeneous,staffilani2024energy}. Convergence rates for discrete approximations
of solutions to the four-wave kinetic equation were investigated in
\cite{dolce2024convergence}. Condensation phenomena have been studied in \cite{EscobedoVelazquez:2015:OTT,staffilani2024condensation}. Local well-posedness has also been established for
the MMT model, a one-dimensional 4-wave kinetic equation, in
\cite{germain2023local}. Uniqueness of solutions to the spectral hierarchy in kinetic wave turbulence theory has also been studied in \cite{rosenzweig2022uniqueness}. 
Finally, the 4-wave kinetic collision operator is closely related to the
Nordheim collision operator for dilute Bose gases. A substantial mathematical
literature on the Nordheim equation has developed around well-posedness,
regularity, condensation phenomena, and long-time dynamics; see
\cite{escobedo2003homogeneous,escobedo2007fundamental,EscobedoVelazquez:2015:FTB,Lu2014_RegularityCondensation,Lu2016_LongTimeBEC,Lu2018_LongTimeStrongBE,chen2025derivation}.

%
%
\subsection*{Acknowledgements}
The authors thank Matthew Rosenzweig for fruitful discussions on this topic.

\section{Settings and main result}

In this work, we consider two physically relevant classes of interaction
weights \(U\):
\begin{itemize}
	\item[{\bf (Case A)}] The acoustic weight
	\begin{equation}\label{CasA}
	U(x)=x,
	\end{equation}
	which corresponds to the two-dimensional case \(d=2\) in \eqref{acoustic}.
	
	\item[{\bf (Case B)}] The regularized power-law weight
	\begin{equation}\label{CasB}
	U(x)=x^\rho(1+x)^{-\beta},
	\qquad
	\rho-\beta<2,
	\qquad
	\rho\ge 1,
	\qquad
	\beta\le \rho-1.
	\end{equation}
\end{itemize}

\begin{remark}\label{r1}
We note that the case \(U(x)=x^2\), corresponding to the three-dimensional
acoustic case \(d=3\) in \eqref{acoustic}, also satisfies the entropy structures
proved in Lemma~\ref{Lemma:EntropyTest} (see also Lemma~\ref{lemmeentro}). In fact, Lemma~\ref{Lemma:EntropyTest} establishes that the entropy structures
hold not only for the cases \(d=2,3\), but for all dimensions \(d\ge 2\). However,
this case lies beyond the \(L^1_{\mathrm{loc}}\) theory developed in the
present paper and has already been studied in \cite{soffer2020energy}. In that work, it
was shown that the dynamics may produce an instantaneous transfer of energy to
infinity. It may happen that for every positive time \(t>0\), a portion of the energy has escaped
the finite-frequency region. In particular, the solution cannot remain an
\(L^1_{\mathrm{loc}}\) density even locally in time for certain initial data. Thus, for the three-dimensional acoustic case \(d=3\), the \(L^1_{\mathrm{loc}}\) framework pursued here is not the natural one for the dynamics studied in \cite{soffer2020energy} and measure-valued solutions are needed to capture the possible instantaneous transfer of energy to infinity. In
addition, the long-time behavior is completely degenerate from the
finite-frequency point of view: all of the energy is lost as \(t\to\infty\).

Thus, although the 3-dimensional acoustic case \(d=3\) also enjoys the
entropy structures identified in the present work, the equation \eqref{3wave}
does not admit a global \(L^1_{\mathrm{loc}}\) weak solution. The present paper
therefore focuses on the remaining acoustic case \(d=2\), for which we prove
the existence of global \(L^1_{\mathrm{loc}}\) solutions
 and analyze their long-time behavior.  The interplay between the entropy structures discovered in the present work and the
energy cascade mechanism studied in \cite{soffer2020energy} for the case \(d=3\)
will be the subject of a forthcoming work. We emphasize that the latter mechanism
is entirely different from the entropy mechanism developed in this paper.	
	\end{remark}

\begin{remark}\label{r2}
The kernel
\[
U(x)=x^\rho(1+x)^{-\beta}
\]
may be regarded as a phenomenological generalization of the homogeneous power-law interactions arising in isotropic \(3\)-wave models. Indeed,
\[
U(x)\sim x^\rho \quad \text{as } x\to 0,
\qquad
U(x)\sim x^{\rho-\beta} \quad \text{as } x\to \infty,
\]
so that it retains the usual low-frequency scaling while producing a weaker effective coupling at high frequencies. This makes it a useful model for describing a crossover between an acoustic-type regime at moderate scales and a softened ultraviolet regime \cite{connaughton2009numerical,connaughton2010dynamical}.  
	
\end{remark}

For a convex entropy density \(e:[0,\infty)\to[0,\infty)\), we define

\begin{equation}\label{Entropy}
	\mathcal{E}_e[f]
	:= \int_{[0,\infty)} \, \mathrm d\omega \frac{e[Uf](\omega)}{U(\omega)}.
\end{equation}

Our main result states as follows.

\begin{theorem} 
	\label{thm}
Assume that
\[
U(\omega)=\omega
\qquad\text{or}\qquad
U(\omega)=\omega^\rho(1+\omega)^{-\beta},
\]
where, in the second case,
\[
\rho-\beta<2,
\qquad
\rho\ge1,
\qquad
\beta\le \rho-1.
\]
Let \(f_0\in L^1((0,\infty))\) be nonnegative and compactly supported in \((0,\infty)\). Assume moreover that there exists a convex function
\[
e:[0,\infty)\to[0,\infty)
\]
such that
\[
e\in C^\infty[0,\infty),
\qquad
e(0)=0,
\qquad
e'(0)=0,
\qquad
\lim_{r\to\infty}\frac{e(r)}{r}=+\infty,
\]
and such that
\[
	\mathcal{E}_e[f_0] \ = \ \int_{[0,\infty)} \mathrm d\omega\, \frac{e(U(\omega)f_0(\omega))}{U(\omega)}<\infty.
\]
Then we have the following inequality: \begin{equation*}
	\frac{\mathrm d}{\mathrm dt}\,\mathcal{E}_e[f]
	=
	\int_{[0,\infty)}	\mathrm d\omega \partial_t\!\left(\frac{e[Uf](\omega)}{U(\omega)}\right)\,
	\le 0 .
\end{equation*}
There exists a nonnegative function
\[
f\in C\bigl([0,\infty);w\text{-}L^1_{\mathrm{loc}}((0,\infty))\bigr)
\]
which is a global weak solution of \eqref{3wave} in the following sense. Define
\[
g(t,\omega):=U(\omega)f(t,\omega),
\qquad
g_0(\omega):=U(\omega)f_0(\omega).
\]
Then:
\begin{itemize}
	\item if \(U(\omega)=\omega\),  for every \(\varphi\in C_c^2((0,\infty))\) and every \(t\ge0\),
	\begin{equation}\label{thm:weak1}
		\int_{(0,\infty)} \mathrm d\omega\, g(t,\omega)\varphi(\omega)
		=
		\int_{(0,\infty)} \mathrm d\omega\, g_0(\omega)\varphi(\omega)
		+
		2\int_0^t \mathrm ds
		\int\!\!\!\int_{\{x>y>0\}} \mathrm dx\,\mathrm dy\,
		g(s,x)g(s,y)\,H_\varphi(x,y),
	\end{equation}
	where
	\[
	H_\varphi(x,y)
	:=
	(x+y)\varphi(x+y)
	+
	(x-y)\varphi(x-y)
	-
	2x\,\varphi(x)
	-
	2y\,\varphi(y).
	\]
	
	\item if \(U(\omega)=\omega^\rho(1+\omega)^{-\beta}\),  for every \(\varphi\in C_c^2((0,\infty))\) and every \(t\ge0\),
	\begin{equation}\label{thm:weak2}
		\int_0^\infty \mathrm d\omega\, g(t,\omega)\varphi(\omega)
		=
		\int_0^\infty \mathrm d\omega\, g_0(\omega)\varphi(\omega)
		+
		2\int_0^t \mathrm ds
		\int\!\!\!\int_{x>y>0} \mathrm dx\,\mathrm dy\,
		g(s,x)g(s,y)\,\mathcal H_\varphi(x,y),
	\end{equation}
	where
	\[
	\mathcal H_\varphi(x,y)
	=
	(x+y)^\rho(1+x+y)^{-\beta}\varphi(x+y)
	+
	(x-y)^\rho(1+x-y)^{-\beta}\varphi(x-y)
	\]
	\[
	\qquad
	-
	\Bigl((x-y)^\rho(1+x-y)^{-\beta}+(x+y)^\rho(1+x+y)^{-\beta}\Bigr)\varphi(x)
	\]
	\[
	\qquad
	+
	\Bigl((x-y)^\rho(1+x-y)^{-\beta}-(x+y)^\rho(1+x+y)^{-\beta}\Bigr)\varphi(y).
	\]
\end{itemize}

Finally, 
\[
\lim_{t\to\infty}f(t,\cdot)=0
\qquad\text{in }L^1_{\mathrm{loc}}((0,\infty)).
\]
\end{theorem}

We also recall the definition of local Radon measures, which will be used throughout  the paper. 
\begin{definition}[Local Radon measures]
	Let \(\Omega \subset \mathbb{R}\) be an open set. A measure \(\mu\) on \(\Omega\) is called a \emph{local Radon measure} if
	\begin{enumerate}
		\item \(\mu\) is a Borel measure on \(\Omega\),
		\item \(\mu\) is finite on every compact subset of \(\Omega\), that is,
		\[
		\mu(K) < \infty \quad \text{for all compact } K \subset \Omega,
		\]
		\item \(\mu\) is inner regular, namely, for every Borel set \(A \subset \Omega\),
		\[
		\mu(A) = \sup \{ \mu(K) : K \subset A,\; K \text{ compact} \}.
		\]
	\end{enumerate}
	The space of local Radon measures on \(\Omega\) is denoted by
	\[
	\mathcal{M}_{\mathrm{loc}}(\Omega),
	\]
	and the cone of nonnegative local Radon measures on \(\Omega\) is denoted by
	\[
	\mathcal{M}_{+,\mathrm{loc}}(\Omega).
	\]
	If \(\mu\) is a nonnegative Radon measure on \(\Omega\) with finite total mass,
	\[
	\mu(\Omega)<\infty,
	\]
	then we write
	\[
	\mu\in \mathcal{M}_{+}(\Omega).
	\]
\end{definition}

\begin{remark}\label{Remark:Cascade}
The construction of weak solutions is based on a cut-off procedure. More precisely,
one first constructs the unique solution \(f_N\) of the cut-off equation
\eqref{Cutoff2-new}, and then passes to the limit as \(N\to\infty\). For the cut-off
problem, the energy is conserved. It is therefore useful to trace the fate, under
this limiting process, of the conserved energy
\[
\int_0^\infty \mathrm d\omega\, \omega\,U(\omega) f_N(t,\omega)
=
\int_0^\infty \mathrm d\omega\, \omega\,U(\omega) f_0(\omega),
\]
which is guaranteed at the cut-off level by Lemma~\ref{Lemma:EnergyCutoff}.

By weak \(L^1_{\mathrm{loc}}\) lower semicontinuity, applied first on compact
subsets of \((0,\infty)\) and then combined with monotone convergence, we obtain
\[
\int_0^\infty \mathrm d\omega\, \omega\,U(\omega) f(t,\omega)
\le
\int_0^\infty \mathrm d\omega\, \omega\,U(\omega) f_0(\omega)
\qquad \text{for all } t\ge 0.
\]
In general, however, equality need not persist after the passage to the limit.
The possible loss of energy can, a priori, arise from two distinct mechanisms:
\begin{enumerate}
	\item[(i)] the formation of an atom \(\alpha(t)\delta_{\omega=0}\) in the
	full measure limit. This mechanism, however, is energetically inert, since
	\(\omega U(\omega)\big|_{\omega=0}=0\), and therefore such an atom cannot
	carry any energy;
	
	\item[(ii)] the escape of mass, and hence of energy, to \(\omega=\infty\),
	which is not detected by the \(L^1_{\mathrm{loc}}\)-topology.
\end{enumerate}

Thus, any energy lost in the limiting process can only be lost through the second
mechanism, namely through escape to infinite frequency. Combining this observation
with the relaxation result
\[
\lim_{t\to\infty} f(t,\cdot)=0
\qquad\text{in } L^1_{\mathrm{loc}}((0,\infty)),
\]
we conclude that the finite-frequency energy is eventually transported to
\(\omega=\infty\). In other words, although the inequality
\[
\int_0^\infty \mathrm d\omega\, \omega U(\omega) f(t,\omega)
<
\int_0^\infty \mathrm d\omega\, \omega U(\omega) f_0(\omega)
\]
may already occur for some positive times \(t>0\), the remaining energy at finite
frequencies also disappears in the long-time limit. This describes precisely the
energy-cascade scenario predicted by wave turbulence theory: energy is    transferred toward arbitrarily
large frequencies. We refer also to
\cite{soffer2019energy,staffilani2024energy} for related results of this type.
\end{remark}

\begin{remark}\label{Remark:EntropyRole}
A noteworthy point in the compactness argument is that the approximating sequence discussed in Remark \ref{Remark:Cascade}
\[
	f_N(t,\omega)\,\mathrm d\omega
\]
is obtained, at first, only as a sequence of nonnegative Radon measures on
\([0,\infty)\). The a priori bounds available at the cut-off level yield tightness
and boundedness in the space of measures, but they do not by themselves exclude
the formation of singular components in the limit. Thus, after extraction, one
only knows a priori that
\[
	\lim_{N\to\infty}f_N(t,\omega)\,\mathrm d\omega = \mu_t
	\qquad\text{narrowly in }\mathcal M_+([0,\infty)).
\]
The entropy estimates in Section \ref{Sec:Entro}
 play a decisive role in improving this information. They
imply local uniform integrability away from the origin and therefore rule out
singular concentration on every compact subinterval of \((0,\infty)\). Consequently,
the only possible singular part of the limiting measure is concentrated at the
degenerate point \(\omega=0\). Hence the full measure limit has the form (see Proposition \ref{Prop:Existence:U=x:Hphi} and Proposition \ref{prop:global-Ux2-Hphi})
\begin{equation}\label{mu}
	\mu_t=\alpha(t)\delta_{\omega=0}+f(t,\omega)\,\mathrm d\omega,
	\qquad \alpha(t)\ge0,
\end{equation}
with
\[
	f(t,\cdot)\in L^1_{\mathrm{loc}}((0,\infty)).
\]

This identification is one of the key consequences of the entropy structures in Section \ref{Sec:Entro}.
Indeed, the weak compactness supplied by measure theory alone would only produce
a Radon-measure-valued limit. The entropy estimates upgrade this measure-valued
limit to a genuine locally integrable density away from \(\omega=0\). In this
sense, the entropy structures do not merely provide an a priori bound; they
determine  the precise regularity and singular structure of the limit.
The possible atom at \(\omega=0\) is harmless for the weak formulation used in
the paper, since \(U(0)=0\) and the equation is written in terms of the weighted
density \(g=Uf\). Thus this atomic component is invisible to the dynamics of
\(g\), while the regular part \(f(t,\omega)\) is the actual \(L^1_{\mathrm{loc}}\)
weak solution constructed in the limit.
\end{remark}

\begin{remark}\label{Remark:Atom}
Several comments are in order concerning the constant \(\alpha(t)\ge0\) appearing
in the decomposition \eqref{mu} (see Remark \ref{Remark:EntropyRole}, Proposition \ref{Prop:Existence:U=x:Hphi} and Proposition \ref{prop:global-Ux2-Hphi})
of a limit point in the narrow topology of the measures
\[
	\bigl\{f_N(t,\omega)\,\mathrm d\omega\bigr\}_{N\ge N_0}
	\qquad\text{in }\mathcal M([0,\infty)).
\]
The measure \(\mu_t\) is an auxiliary object arising from compactness in the
space of Radon measures on the closed half-line \([0,\infty)\). A priori, such a
limit could contain singular components. The entropy estimates exclude singular
concentration on every compact subset of \((0,\infty)\), and therefore any
remaining singular part can only be supported at the degenerate point
\(\omega=0\). This gives the above decomposition, with
\[
	f^{\mathrm{reg}}(t,\cdot)\in L^1_{\mathrm{loc}}((0,\infty)).
\]

The important point is that the actual weak solution constructed in this paper
is the regular part \(f^{\mathrm{reg}}\), not the full measure \(\mu_t\). Indeed,
the weak formulations \eqref{thm:weak1} and \eqref{thm:weak2} are written in
terms of the weighted density
\[
	g(t,\omega)=U(\omega)f(t,\omega).
\]
Since \(U(0)=0\) in both Case~A and Case~B, an atom located at \(\omega=0\) is
annihilated by the factor \(U(\omega)\). Consequently, such an atomic component
contributes neither to the left-hand side nor to the collision terms in the
weak formulation. Its time evolution is therefore invisible to the equation
written in terms of \(g\). In particular, the weak formulation imposes no
condition on the coefficient \(\alpha(t)\).

Thus the possible nonuniqueness of \(\alpha(t)\) should be understood as an
artifact of embedding the approximating sequence into the larger space
\(\mathcal M([0,\infty))\). The approximation scheme identifies the locally
integrable density on \((0,\infty)\), which is the object entering the weak
formulation. It does not canonically determine how much mass, if any, is stored
at the single degenerate point \(\omega=0\), because that mass is invisible to
the weighted equation.

The function \(f^{\mathrm{reg}}\) obtained from the compactness argument is a
bona fide \(L^1_{\mathrm{loc}}\) weak solution of \eqref{3wave}, in the sense
specified in Theorem~\ref{thm}.  
\end{remark}

\begin{remark}\label{Remark:KZ}
The relaxation statement
\[
	\lim_{t\to\infty} f(t,\cdot)=0
	\qquad\text{in } L^1_{\mathrm{loc}}((0,\infty))
\]
should not be interpreted as contradicting the existence or physical relevance of
Kolmogorov--Zakharov spectra. Rather, it is a rigidity statement within the
particular finite-entropy \(L^1_{\mathrm{loc}}\) framework used in this paper.

Indeed, the Kolmogorov--Zakharov stationary spectra associated with
\eqref{3wave} are formal scale-invariant power laws of the form
\[
	f_{\mathrm{KZ}}(\omega)\sim \omega^{-\kappa},
\]
where the exponent \(\kappa\) depends on the interaction weight \(U\); see
\cite{Nazarenko:2011:WT,zakharov2012kolmogorov}.  For the physically relevant
values of \(\kappa\), they typically fail either to be locally integrable near
\(\omega=0\), or to satisfy the entropy condition
\[
	\mathcal E_e[f]
	=
	\int_{[0,\infty)}\,\mathrm d\omega
	\frac{e(U(\omega)f(\omega))}{U(\omega)}
	<\infty
\]
imposed in Theorem~\ref{thm}. Thus the KZ spectra lie outside the solution class
selected by the entropy method.

The conclusion \(f(t,\cdot)\to0\) in \(L^1_{\mathrm{loc}}\) therefore means the
following: among the global weak solutions generated by finite-entropy initial
data through the present compactness procedure, no nonzero stationary profile can
persist on bounded frequency intervals as \(t\to\infty\). In this sense, the
entropy structure enforces a strong local-in-frequency rigidity. It rules out
nontrivial finite-frequency asymptotic states within the finite-entropy class,
but it does not rule out KZ spectra as formal stationary cascades outside that
class.

This interpretation is consistent with the usual cascade picture. The convergence
to zero is only local in frequency and does not assert convergence in a global
energy topology. In view of Remark~\ref{Remark:Cascade}, the disappearance
of the local density should be understood as the transfer of energy toward
arbitrarily large frequencies, rather than as relaxation to a thermodynamic
equilibrium with conserved finite-frequency mass. Thus the result complements,
rather than contradicts, the Kolmogorov--Zakharov description: KZ spectra describe
formal stationary flux states, while   Theorem \ref{thm} describes the rigidity of
finite-entropy \(L^1_{\mathrm{loc}}\) weak solutions constructed in the present
framework.
\end{remark}

\begin{remark}
	The assumption that $f_0$ is compactly supported in $(0,\infty)$ in 
	Theorem~\ref{thm} is purely technical and can be relaxed. It is used in 
	Proposition~\ref{Prop:Existence:U=x:Hphi} to guarantee the existence of $N_0$ such that 
	$f_{N,0} = f_0$ for all $N \geq N_0$, thereby simplifying the identification 
	of the initial datum after passage to the limit. More general initial data, 
	integrable against the natural weights and satisfying the finite-entropy 
	condition of Theorem~\ref{thm}, can be accommodated by approximating 
	$f_0$ by a sequence of compactly supported data and exploiting the entropy 
	and moment estimates uniformly along this approximation. We have chosen not 
	to pursue this generalization here, as it would lengthen the
	paper without affecting the conceptual content of the main results.
\end{remark}

 \section{Entropy 
 	structures}\label{Sec:Entro}

The following Lemma is fundamental for the paper. It reveals a new entropy, which has not been observed in the physics literature. 
\begin{lemma}
\label{Lemma:EntropyTest}
Let \(f\) be a suitable non-negative regular solution to \eqref{3wave}, and let
\(e:[0,\infty)\to[0,\infty)\) be a  function satisfying \[
e\in C^1([0,\infty)), \qquad e \text{ is convex}, \qquad e(0)=0, \qquad e\ge 0.
\] 
 Assume that 
 $U(x)=U(|x|)\ge 0$ and
\begin{equation}\label{UCondition}
	U(x+y)\ge U(y) + U(x-y),
	\qquad
	U(x)\ge U(y),
	\qquad
	\text{for all } (x,y)\in[0,\infty)^2 \text{ with } x\ge y.
\end{equation}

Moreover, we assume that \(U\) only vanishes on a set of zero measure.

Then we have the following inequality:\begin{equation}\label{Lemma:EntropyTest:1}
	\frac{\mathrm d}{\mathrm dt}\,\mathcal{E}_e[f]
	=
	\int_{[0,\infty)}	\mathrm d\omega \partial_t\!\left(\frac{e[Uf](\omega)}{U(\omega)}\right)\,
	\le 0 .
\end{equation}

Moreover, we have
 that for every \(t>0\),
	\begin{equation}\label{eq:dissipation-assumption}
	2	\int_0^t \mathrm{d}s
		\int_0^\infty \mathrm{d}\omega
		\int_0^\omega \mathrm{d}\omega_1\,
		U(\omega_1)U(\omega-\omega_1)\,f(s,\omega)\,f(s,\omega_1)\,
		e'[Uf](\omega)
		\le
		\mathcal{E}_e[f_0].
	\end{equation}

\end{lemma}

\begin{proof}
Using $e'[Uf]$ as a test function, we obtain
\begin{equation}\label{Lemma:EntropyTest:E1}
	\begin{aligned}
	&	\int_{[0,\infty)}\, \mathrm d\omega \mathcal{Q}[f](\omega)\, e'[Uf](\omega)\\
		= {} & 
		\iiint_{[0,\infty)^3}\,
		\mathrm d\omega\, \mathrm d\omega_1\, \mathrm d\omega_2
		\delta(\omega-\omega_1-\omega_2)
	U(
	\omega_1)U(
	\omega_2)	\Big(
		f_1 f_2
		- f f_1
		- f f_2
		\Big)
		e'[Uf]  \\
		& - 2
		\iiint_{[0,\infty)^3}\,
		\mathrm d\omega\, \mathrm d\omega_1\, \mathrm d\omega_2
		\delta(\omega_2-\omega-\omega_1)
		U(
	\omega_1)U(
	\omega_2)	\Big(
		f_1 f
		- f_1 f_2
		- f f_2
		\Big)
		e'[Uf]  \\
		= {} &
		\int_{[0,\infty)} \mathrm d\omega
		\int_{[0,\omega]} \mathrm d\omega_1
		U(
	\omega_1)U(
	\omega-\omega_1)	\Big(
		f(\omega_1) f(\omega-\omega_1)
		- f(\omega) f(\omega_1)
		- f(\omega) f(\omega-\omega_1)
		\Big)
		e'[Uf](\omega) \\
		& -
	2	\int_{[0,\infty)} \mathrm d\omega
		\int_{[0,\infty)} \mathrm d\omega_1
		U(
	\omega_1)U(
	\omega_1+\omega)	\Big(
		f(\omega) f(\omega_1)
		- f(\omega_1) f(\omega+\omega_1)
		- f(\omega) f(\omega+\omega_1)
		\Big)
		e'[Uf](\omega).
	\end{aligned}
\end{equation}

We then define
\begin{equation}\label{Lemma:EntropyTest:E1a}
	\begin{aligned}
		A_1 := {} &
		\int_{[0,\infty)} \mathrm d\omega
		\int_{[0,\omega]} \mathrm d\omega_1\,
		U(\omega_1)U(\omega-\omega_1)
		\Big(
		f(\omega_1) f(\omega-\omega_1)
		- f(\omega) f(\omega_1)
		- f(\omega) f(\omega-\omega_1)
		\Big)
		\,e'[Uf](\omega) \\
		&\;-\;
		2\int_{[0,\infty)} \mathrm d\omega
		\int_{[0,\infty)} \mathrm d\omega_1\,
		U(\omega_1)U(\omega_1+\omega)
		\,f(\omega) f(\omega_1)\,
		\mathbf{1}_{\omega_1\le \omega}\,e'[Uf](\omega), \\[0.5em]
		A_2 := {} &
		-2\int_{[0,\infty)} \mathrm d\omega
		\int_{[0,\infty)} \mathrm d\omega_1\,
		U(\omega_1)U(\omega_1+\omega)
		\Big(
		f(\omega) f(\omega_1)\mathbf{1}_{\omega_1\ge \omega} \\
		&\hspace{6.3em}
		- f(\omega_1) f(\omega+\omega_1)
		- f(\omega) f(\omega+\omega_1)
		\Big)
		\,e'[Uf](\omega).
	\end{aligned}
\end{equation}

Next, we estimate $A_1$. By the symmetry between $\omega_1$ and $\omega-\omega_1$, we can write
\begin{equation*}
	\begin{aligned}
		A_1 = {} &
		\int_{[0,\infty)} \mathrm d\omega
		\int_{[0,\omega]} \mathrm d\omega_1
		\Big(
	U(\omega_1)U(\omega-\omega_1)	f(\omega_1)\, f(\omega-\omega_1)
		- 2U(\omega_1)U(\omega-\omega_1) f(\omega)\, f(\omega_1)
		\Big)\,
		e'[Uf](\omega)\\
		&\;-\;
		2\int_{[0,\infty)} \mathrm d\omega
		\int_{[0,\infty)} \mathrm d\omega_1\,
		U(\omega_1)U(\omega+\omega_1)
		\,f(\omega) f(\omega_1)\,
		\mathbf{1}_{\omega_1\le \omega}\,
		e'[Uf](\omega), \\
		\le {} &
		2 \int_{[0,\infty)} \mathrm d\omega
		\int_{[0,\omega]} \mathrm d\omega_1U(\omega_1)U(\omega-\omega_1)
		f(\omega_1)\, f(\omega-\omega_1)\,
		\mathbf{1}_{\{\omega_1 \ge \omega-\omega_1\}}\,
		e'[Uf](\omega) \\
		& -
		2 \int_{[0,\infty)} \mathrm d\omega
		\int_{[0,\omega]} \mathrm d\omega_1
	U(\omega)U(\omega_1)	f(\omega)\, f(\omega_1)\,
		e'[Uf](\omega) \ - \ A_3,
	\end{aligned}
\end{equation*}
where we have used the fact that $U(\omega+\omega_1)\ge U(\omega)$
and we have set
\begin{equation*}
	\begin{aligned}
		A_3 : = {} &
		2\int_{[0,\infty)} \mathrm d\omega
		\int_{[0,\omega]} \mathrm d\omega_1
		U(\omega_1)U(\omega-\omega_1) f(\omega)\, f(\omega_1)
		\,
		e'[Uf](\omega).
	\end{aligned}
\end{equation*}

We then estimate $A_1$ as
\begin{equation}\label{Lemma:EntropyTest:E2}
	\begin{aligned}
		A_1 \le {} &
		2\int_{[0,\infty)} \mathrm d\omega
		\int_{[0,\omega]} \mathrm d\omega_1
U(\omega-\omega_1)	f(\omega-\omega_1)	\big(
		U(\omega_1)f(\omega_1) - U(\omega)f(\omega)
		\big)
				e'[Uf](\omega)\,
				\mathbf{1}_{\{\omega_1 \ge \omega-\omega_1\}} \\
		& +
	2	\int_{[0,\infty)} \mathrm d\omega
		\int_{[0,\omega]} \mathrm d\omega_1
	U(\omega-\omega_1)	f(\omega-\omega_1)\, U(\omega)f(\omega)
		e'[Uf](\omega) \,
		\mathbf{1}_{\{\omega_1 \ge \omega-\omega_1\}}\\
		& -
		2 \int_{[0,\infty)} \mathrm d\omega
		\int_{[0,\omega]} \mathrm d\omega_1
	U(\omega)	f(\omega)\,U(\omega_1) f(\omega_1)
		e'[Uf](\omega) \ - \ A_3.
	\end{aligned}
\end{equation}
Noting that $e$ is a convex function, we have
\[
\big(
U(\omega_1)f(\omega_1) - U(\omega)f(\omega)
\big)
e'[Uf](\omega)
\le
e[Uf](\omega_1) - e[Uf](\omega).
\]
Combining this inequality with \eqref{Lemma:EntropyTest:E2}, we obtain
\begin{equation*} 
	\begin{aligned}
		A_1 \le {} &
	2	\int_{[0,\infty)} \mathrm d\omega
		\int_{[0,\omega]} \mathrm d\omega_1
	U(\omega-\omega_1)	f(\omega-\omega_1)
		\big(
		e[Uf](\omega_1) - e[Uf](\omega)
		\big) \,
		\mathbf{1}_{\{\omega_1 \ge \omega-\omega_1\}}\\
		& +
	2	\int_{[0,\infty)} \mathrm d\omega
		\int_{[0,\omega]} \mathrm d\omega_1
	U(\omega-\omega_1)	f(\omega-\omega_1)\, U(\omega)f(\omega)
		e'[Uf](\omega) \,
		\mathbf{1}_{\{\omega_1 \ge \omega-\omega_1\}}\\
		& -
		2 \int_{[0,\infty)} \mathrm d\omega
		\int_{[0,\omega]} \mathrm d\omega_1
	U(\omega)	f(\omega)\,U(\omega_1) f(\omega_1)
		e'[Uf](\omega)  \ - \ A_3 \\
		\le {} &
		2\int_{[0,\infty)} \mathrm d\omega
		\int_{[0,\omega]} \mathrm d\omega_1
	U(\omega-\omega_1)	f(\omega-\omega_1)\, e[Uf](\omega_1)\,
		\mathbf{1}_{\{\omega_1 \ge \omega-\omega_1\}} \\
		& -
	2 \int_{[0,\infty)} \mathrm d\omega
	\int_{[0,\omega]} \mathrm d\omega_1
U(\omega)	f(\omega)\,U(\omega_1) f(\omega_1)
	e'[Uf](\omega) \\
		& +
		  2\int_{[0,\infty)} \mathrm d\omega
		\int_{[0,\omega]} \mathrm d\omega_1
	U(\omega-\omega_1)	f(\omega-\omega_1)
		\Big(
		U(\omega)f(\omega)\, e'[Uf](\omega) - e[Uf](\omega)
		\Big)\,
		\mathbf{1}_{\{\omega_1 \ge \omega-\omega_1\}}   -  A_3.
	\end{aligned}
\end{equation*}
Performing the changes of variables
\[
(\omega-\omega_1,\omega_1)\mapsto(\omega',\omega),\qquad
(\omega,\omega_1)\mapsto(\omega,\omega'),\qquad
(\omega-\omega_1,\omega)\mapsto(\omega',\omega),
\]
respectively in the three integrals above, we obtain
\begin{equation}\label{Lemma:EntropyTest:E3}
	\begin{aligned}
		A_1 \le {} &
	2	\int_{[0,\infty)} \mathrm d\omega
		\int_{[0,\omega]} \mathrm d\omega'
	U(\omega')	f(\omega')\, e[Uf](\omega) \\
		& -
		2 \int_{[0,\infty)} \mathrm d\omega
		\int_{[0,\omega]} \mathrm d\omega'
		U(\omega)f(\omega)\,U(\omega') f(\omega')
		e'[Uf](\omega) \\
		& +
		2\int_{[0,\infty)} \mathrm d\omega
		\int_{[0,\omega]} \mathrm d\omega'
	U(\omega')	f(\omega')
		\Big(
	U(\omega)	f(\omega)\, e'[Uf](\omega) - e[Uf](\omega)
		\Big)\,
		\mathbf{1}_{\{\omega' \le \omega-\omega'\}}  \ - \ A_3\\
	\le {} &
-
	2\int_{[0,\infty)} \mathrm d\omega
	\int_{[0,\omega]} \mathrm d\omega'
U(\omega')f(\omega')
	\Big(
U(\omega)	f(\omega)\, e'[Uf](\omega) - e[Uf](\omega)
	\Big)\,
	\mathbf{1}_{\{\omega' \ge \omega-\omega'\}}	 \ - \ A_3.
	\end{aligned}
\end{equation}

We now estimate $A_2$. We expand the second term in $A_2$ as follows:
\begin{equation}\label{Lemma:EntropyTest:E4}
	\begin{aligned}
		&\;
		2\int_{[0,\infty)} \mathrm d\omega
		\int_{[0,\infty)} \mathrm d\omega_1\,
		U(\omega_1)U(\omega_1+\omega)
		f(\omega_1) f(\omega+\omega_1)
		\,e'[Uf](\omega) \\
		=\;&
		2\int_{[0,\infty)} \mathrm d\omega
		\int_{[0,\infty)} \mathrm d\omega_1\,
		U(\omega_1+\omega) f(\omega+\omega_1)
		\Big[
		U(\omega_1)f(\omega_1)
		- U(\omega)f(\omega)
		\Big]
		\,e'[Uf](\omega) \\
		&\;
		+ 2\int_{[0,\infty)} \mathrm d\omega
		\int_{[0,\infty)} \mathrm d\omega_1\,
		U(\omega_1+\omega) f(\omega+\omega_1)
		U(\omega)f(\omega)
		\,e'[Uf](\omega).
	\end{aligned}
\end{equation}
By the convexity of $e$, we have the estimate
\[
\Big[
U(\omega_1)f(\omega_1)
- U(\omega)f(\omega)
\Big]
\,e'[Uf](\omega)
\le e[Uf](\omega_1) - e[Uf](\omega),
\]
which, in combination with \eqref{Lemma:EntropyTest:E4}, yields
\begin{equation}\label{Lemma:EntropyTest:E5}
	\begin{aligned}
		&\;
		2\int_{[0,\infty)} \mathrm d\omega
		\int_{[0,\infty)} \mathrm d\omega_1\,
		U(\omega_1)U(\omega_1+\omega)
		f(\omega_1) f(\omega+\omega_1)
		\,e'[Uf](\omega) \\
		\le\;&
		2\int_{[0,\infty)} \mathrm d\omega
		\int_{[0,\infty)} \mathrm d\omega_1\,
		U(\omega_1+\omega) f(\omega+\omega_1)
		\Big(
		e[Uf](\omega_1) - e[Uf](\omega)
		\Big) \\
		&\;
		+ 2\int_{[0,\infty)} \mathrm d\omega
		\int_{[0,\infty)} \mathrm d\omega_1\,
		U(\omega_1+\omega) f(\omega+\omega_1)
		U(\omega)f(\omega)
		\,e'[Uf](\omega).
	\end{aligned}
\end{equation}
Observing that
\[
2\int_{[0,\infty)} \mathrm d\omega
\int_{[0,\infty)} \mathrm d\omega_1\,
U(\omega_1+\omega) f(\omega+\omega_1)
e[Uf](\omega_1)
=
2\int_{[0,\infty)} \mathrm d\omega
\int_{[0,\infty)} \mathrm d\omega_1\,
U(\omega_1+\omega) f(\omega+\omega_1)
e[Uf](\omega),
\]
we deduce from \eqref{Lemma:EntropyTest:E5} that
\begin{equation}\label{Lemma:EntropyTest:E6}
	\begin{aligned}
		&\;
		2\int_{[0,\infty)} \mathrm d\omega
		\int_{[0,\infty)} \mathrm d\omega_1\,
		U(\omega_1)U(\omega_1+\omega)
		f(\omega_1) f(\omega+\omega_1)
		\,e'[Uf](\omega) \\
		\le\;&
		2\int_{[0,\infty)} \mathrm d\omega
		\int_{[0,\infty)} \mathrm d\omega_1\,
		U(\omega_1+\omega) f(\omega+\omega_1)
		U(\omega)f(\omega)
		\,e'[Uf](\omega).
	\end{aligned}
\end{equation}

Plugging \eqref{Lemma:EntropyTest:E6} into the expression for $A_2$, we obtain the bound
\begin{equation}\label{Lemma:EntropyTest:E7}
	\begin{aligned}
		A_2 \le {} &
		-2\int_{[0,\infty)} \mathrm d\omega
		\int_{[0,\infty)} \mathrm d\omega_1\,
		\Big(
		U(\omega_1)U(\omega_1+\omega)
		f(\omega) f(\omega_1)\mathbf{1}_{\omega_1\ge \omega} \\
		&\hspace{6.3em}
		- f(\omega)U(\omega_1+\omega) f(\omega+\omega_1)
		\big(U(\omega_1)+U(\omega)\big)
		\Big)
		\,e'[Uf](\omega) \\
		\le {} &
		-2\int_{[0,\infty)} \mathrm d\omega
		\int_{[\omega,\infty)} \mathrm d\omega_1\,
		\Big(
		U(\omega_1)U(\omega_1+\omega)
		f(\omega) f(\omega_1) \\
		&\hspace{6.3em}
		- f(\omega)U(\omega_1) f(\omega_1)
		\big(U(\omega_1-\omega)+U(\omega)\big)
		\Big)
		\,e'[Uf](\omega).
	\end{aligned}
\end{equation}
where we have used the change of variables
\((\omega_1+\omega,\omega)\mapsto(\omega_1,\omega)\)
in the second term on the right-hand side.

Since $U(\omega_1+\omega)\ge U(\omega_1-\omega)+U(\omega)$, we deduce from
\eqref{Lemma:EntropyTest:E7} that $A_2\le 0$, which, in combination with
\eqref{Lemma:EntropyTest:E3}, yields
\begin{equation}\label{Lemma:EntropyTest:E8}
	\begin{aligned}
		&\;
		\int_{[0,\infty)} \mathrm d\omega\,
		\mathcal{Q}[f](\omega)\, e'[Uf](\omega) \\
		\le {} &
		-2\int_{[0,\infty)} \mathrm d\omega
		\int_{[0,\omega]} \mathrm d\omega'\,
		U(\omega')f(\omega')
		\Big(
		U(\omega) f(\omega)\, e'[Uf](\omega)
		- e[Uf](\omega)
		\Big)\,
		\mathbf{1}_{\{\omega' \ge \omega-\omega'\}}  \ - \ A_3.
	\end{aligned}
\end{equation}
The first conclusion of the lemma follows since, by  the convexity of $e$
\[
U(\omega) f(\omega)\, e'[Uf](\omega) - e[Uf](\omega)
\ge 0,
\]

and
$$\int_{[0,\infty)}\, \mathrm d\omega \partial_t\!\left(\frac{e[Uf](\omega)}{U(\omega)}\right) \ = \ 	\int_{[0,\infty)}\, \mathrm d\omega \mathcal{Q}[f](\omega)\, e'[Uf](\omega).$$

Inequality \eqref{eq:dissipation-assumption} follows from the fact that 

\begin{equation}\label{Lemma:EntropyTest:E8a}
	\begin{aligned}
		&\;
		\int_{[0,\infty)} \mathrm d\omega\,
		\mathcal{Q}[f](\omega)\, e'[Uf](\omega) \
		\le {} 
		- \ A_3.
	\end{aligned}
\end{equation}

\end{proof}
\begin{remark}\label{Remark:entrpy}
The particular form of the entropy functional merits some heuristic explanation,
since it differs in an essential way from the entropy functionals more commonly
associated with classical kinetic equations. For the \(3\)-wave kinetic equations
studied here, there is no evident detailed-balance mechanism of the type familiar
from Boltzmann-type models. In fact, the only entropy functional for wave kinetic
equations that appears to be commonly available in the literature is the logarithmic
entropy
\[
\int_{[0,\infty)} \ln f,
\]
see, for instance, \cite{rumpf2021wave,zakharov2012kolmogorov}. This quantity,
however, is not useful  in general.

The key point in our setting is that the relevant entropy structure is not generated
by detailed balance, but rather by a one-sided algebraic balance encoded in the
interaction weight \(U\). More precisely, the structural condition
\[
U(x+y)\ge U(y)+U(x-y),
\qquad x\ge y\ge0,
\]
exposes the hidden coercive structure needed to control the nonlinear interaction
terms. After symmetrization, this super-additivity condition plays the role of a
one-sided microscopic balance: it ensures that the loss contributions dominate
the gain contributions once the convexity inequality for \(e\) is applied to the
weighted density \(g=Uf\). This observation is the central mechanism underlying
the entropy estimates developed below.

The factor \(1/U(\omega)\) in the definition of \(\mathcal E_e\) is therefore not a
mere technical weight, but an intrinsic part of the entropy structure. First,
it converts an entropy written in the variable \(g=Uf\) into a functional of
the original density \(f\). More importantly, it is exactly the normalization
for which the entropy variation produces the test function \(e'(Uf)\) used in
Lemma~\ref{Lemma:EntropyTest}. With this choice, the symmetrized nonlinear
terms generate differences of the form
\[
U(\omega_1)f(\omega_1)-U(\omega)f(\omega),
\]
which are controlled by
\[
e\bigl(U(\omega_1)f(\omega_1)\bigr)
-
e\bigl(U(\omega)f(\omega)\bigr)
\]
through the elementary convexity inequality
\[
(a-b)e'(b)\le e(a)-e(b).
\]
Thus the entropy is designed so that the algebra of the collision operator and
the convexity of \(e\) act on the same quantity, namely \(g=Uf\).

In summary, the estimate relies on the alignment of three ingredients: the
super-additivity of \(U\), which provides the structural one-sided balance; the
convexity of \(e\), which provides the differential inequality; and the change
of variables \(f\mapsto g=Uf\), together with the weight \(1/U\), which brings
these two mechanisms into the same algebraic framework. This alignment is what
ultimately gives the entropy dissipation a definite sign.
\end{remark}
\begin{lemma}\label{lemmeentro}
Let
\[
U(\omega)=\omega^\rho(1+\omega)^{-\beta}, \qquad \mbox{ or } \qquad U(\omega)=\omega^{d-1}
\qquad  \mbox{ for }  \omega\ge 0,
\]
where
\[
\rho\ge 1,
\qquad
\beta\le \rho-1, \qquad d\ge 2.
\]
Then
\[
U(x+y)\ge U(y)+U(x-y),
\qquad
U(x)\ge U(y),
\]
for all \((x,y)\in [0,\infty)^2\) with \(x\ge y\).
\end{lemma}

\begin{proof}
First, we consider the case \(U(\omega)=\omega^\rho(1+\omega)^{-\beta}\). Since \(\rho\ge 1\), we have \(U(0)=0\). For \(\omega>0\), we compute
\[
U'(\omega)
=
\rho \omega^{\rho-1}(1+\omega)^{-\beta}
-\beta \omega^\rho (1+\omega)^{-\beta-1}.
\]
Equivalently, we have
\[
U'(\omega)
=
\omega^{\rho-1}(1+\omega)^{-\beta-1}
\bigl(\rho+(\rho-\beta)\omega\bigr).
\]
Since
\(
\beta\le \rho-1,
\)
we have
\(
\rho-\beta\ge 1.
\)
Therefore, we obtain
\[
\rho+(\rho-\beta)\omega>0
\qquad \text{for all } \omega>0.
\]
Hence
\(
U'(\omega)\ge 0
\) for all  \(\omega>0.\)
Thus \(U\) is nondecreasing on \([0,\infty)\). In particular, if \(x\ge y\), then
\(
U(x)\ge U(y).
\)

Next we define
\[
\mathfrak V(\omega):=\frac{U(\omega)}{\omega}
=
\omega^{\rho-1}(1+\omega)^{-\beta},
\qquad \omega>0.
\]
We show that \(\mathfrak V\) is nondecreasing. We observe that 
\[
\mathfrak V'(\omega)
=
\mathfrak V(\omega)\,
\frac{(\rho-1)+(\rho-1-\beta)\omega}{\omega(1+\omega)}.
\]
Since
\(
\rho\ge 1,
\beta\le \rho-1,
\)
we have
\(
\rho-1\ge 0,
\rho-1-\beta\ge 0.
\)
Therefore
\[
\mathfrak V'(\omega)\ge 0
\qquad \text{for all } \omega>0.
\]
Hence \(\mathfrak V\) is nondecreasing.

Now fix \(x,y\ge 0\) with \(x\ge y\). If \(x+y=0\), then \(x=y=0\), and the desired inequality is trivial. Assume \(x+y>0\). Since
\(
0\le y\le x+y,
0\le x-y\le x+y,
\)
and \(\mathfrak V\) is nondecreasing, we have
\[
\mathfrak V(y)\le \mathfrak V(x+y), \qquad
\mathfrak V(x-y)\le \mathfrak V(x+y).
\]

Therefore, we can write
\[
U(y)+U(x-y)
=
y\mathfrak V(y)+(x-y)\mathfrak V(x-y)
\le
y\mathfrak V(x+y)+(x-y)\mathfrak V(x+y).
\]
Hence, we find
\[
U(y)+U(x-y)
\le
x \mathfrak V(x+y).
\]
Since \(x\le x+y\), we obtain
\[
x\mathfrak V(x+y)\le (x+y)\mathfrak V(x+y)=U(x+y).
\]
Thus, we obtain
\[
U(y)+U(x-y)\le U(x+y).
\]

For \(U(\omega)=\omega^{d-1}\), \(U\) is nondecreasing on \([0,\infty)\) for
\(d\ge2\). Moreover,
\[
\frac{U(\omega)}{\omega}=\omega^{d-2}
\]
is nondecreasing on \((0,\infty)\). The same argument used above, with
\(\mathfrak V(\omega)=U(\omega)/\omega\), gives
\[
U(y)+U(x-y)\le U(x+y)
\]
for \(x\ge y\ge0\). 
\end{proof}

\section{Approximated equations}\label{Sec:Cutoff}
Let $N>0$ be a positive integer. We define the cut-off operator by
\begin{equation}\label{Cutoff1}
	\begin{aligned}
		\mathcal{Q}_N[f](t,\omega)
		:= {} &
	\iint_{[0,\infty)^2}
	\mathrm d\omega_1\, \mathrm d\omega_2\,
	\delta(\omega-\omega_1-\omega_2)\,
	U(\omega_1)U(\omega_2)
	\Big(
	f_1 f_2\chi_{\{\omega_1,\omega_2\le N\}}
	- f f_1\chi_{\{\omega,\omega_1\le N\}}\\
	&\;
	- f f_2\chi_{\{\omega,\omega_2\le N\}}
	\Big) \\
	&\;
	- 2
	\iint_{[0,\infty)^2}
	\mathrm d\omega_1\, \mathrm d\omega_2\,
	\delta(\omega_2-\omega-\omega_1)\,
	U(\omega_1)U(\omega_2)
	\Big(
	f_1 f\chi_{\{\omega,\omega_1\le N\}}
\\	&\;	- f_1 f_2\chi_{\{\omega_1,\omega_2\le N\}}
	- f f_2\chi_{\{\omega,\omega_2\le N\}}
	\Big),
	\end{aligned}
\end{equation}
where
$$\chi_{\{\omega\le N\}} = 1 \mbox{ when } \omega\le N \mbox{ and } 0 \mbox{ elsewhere}. $$
	Let \(f_0\in L^1([0,\infty))\) be nonnegative. We define
\[
f_{N,0}(\omega):=f_0(\omega)\chi_{\{\omega\le N\}}.
\]
We consider the cutoff problem
\begin{equation}\label{Cutoff2-new}
	\begin{aligned}
		\partial_t f_N(t,\omega)&=\mathcal Q_N[f_N](t,\omega),
		\qquad (t,\omega)\in(0,\infty)\times[0,\infty),\\
		f_N(0,\omega)&=f_{N,0}(\omega),
		\qquad \omega\in[0,\infty),
	\end{aligned}
\end{equation}

\subsection{Cut-off energy, entropy and Lipschitz estimates}

We  have the following energy conservation Lemma.

\begin{lemma}\label{Lemma:EnergyCutoff}
Let \(f_N\) be a solution to \eqref{Cutoff2-new}. We have the following energy conservation law:
\begin{equation}\label{Lemma:Energy:1}
	\frac{\mathrm d}{\mathrm dt}\left(\int_0^\infty \mathrm d\omega\, f_N(t,\omega)\,U(\omega)\,\omega\right)=0.
\end{equation}
\end{lemma} 
\begin{proof}
	A standard symmetry argument \cite{staffilani2025evolution,staffilani2025finite,staffilani2025formation} shows that
	\begin{align}
	&	\int_{[0,\infty)} \mathcal{Q}_N[f](\omega)\,U(\omega)\,\omega \,\mathrm d\omega
		=
		\iiint_{[0,\infty)^3}
		\mathrm d\omega\,\mathrm d\omega_1\,\mathrm d\omega_2\,
		\delta(\omega-\omega_1-\omega_2)
		\notag\\
		&\qquad\qquad \times
		U(\omega_1)U(\omega_2)U(\omega)\,
		\big(f_1f_2\chi_{\{\omega_1,\omega_2\le N\}}-ff_1\chi_{\{\omega_1,\omega\le N\}}-ff_2\chi_{\{\omega,\omega_2\le N\}}\big)\,
		(\omega-\omega_1-\omega_2)
		\label{Lemma:Energy:E1}\\
		&=0,
		\notag
	\end{align}
	which implies \eqref{Lemma:Energy:1}.
\end{proof}

In addition, we also have the following entropy estimate for the cut-off equation \eqref{Cutoff2-new}.

\begin{lemma}
	\label{Lemma:EntropyTestCutoff}
	Let \(f_N\) be a non-negative  solution to \eqref{Cutoff2-new}, and let
\(e:[0,\infty)\to[0,\infty)\) be a function satisfying
\[
e\in C^1([0,\infty)), \qquad e \text{ is convex}, \qquad e(0)=0, \qquad e\ge 0.
\]
Assume that \(U(x)=U(|x|)\ge 0\) and
\[
	U(x+y)\ge U(y) + U(x-y),
	\qquad
	U(x)\ge U(y),
	\qquad
	\text{for all } (x,y)\in[0,\infty)^2 \text{ with } x\ge y.
\]

Moreover, we assume that \(U\) only vanishes on a set of zero measure. 

Then we have the following inequalities:
\begin{equation}\label{Lemma:EntropyTest:1Cutoff}
\frac{\mathrm d}{\mathrm dt}\,\mathcal{E}_e[f_N\chi_{\{\omega\le N\}}]
=
\int_{[0,\infty)} \mathrm d\omega\,
\partial_t\!\left(
\frac{e[Uf_N\chi_{\{\omega\le N\}}](\omega)}{U(\omega)}
\right)
\le 0,
\end{equation}
and 	\begin{equation}\label{eq:dissipation-assumption:1}
	2	\int_0^t \mathrm{d}s
		\int_0^\infty \mathrm{d}\omega
		\int_0^\omega \mathrm{d}\omega_1\,
		U(\omega_1)U(\omega-\omega_1)\,f_N(s,\omega)\chi_{\{\omega\le N\}}\,f_N(s,\omega_1)\chi_{\{\omega_1\le N\}}\,
		e'[Uf_N\chi_{\{\omega\le N\}}](\omega)
		\le
		\mathcal{E}_e[f_0].
	\end{equation}
\end{lemma}

\begin{proof}
Using
\[
e'\!\left(Uf_N\chi_{\{\omega\le N\}}\right)\chi_{\{\omega\le N\}}
\]
as a test function, we obtain
\begin{equation}\label{Lemma:EntropyTest:E1t}
	\begin{aligned}
		&\int_{[0,\infty)} \mathrm d\omega\,
		\mathcal Q_N[f_N](\omega)\,
		e'\!\left[Uf_N\chi_{\{\omega\le N\}}\right](\omega)\,
		\chi_{\{\omega\le N\}} \\
		={}&
		\int_{[0,\infty)} \mathrm d\omega
		\int_{[0,\omega]} \mathrm d\omega_1\,
		U(\omega_1)U(\omega-\omega_1)
		\,\chi_{\{\omega\le N\}}
		\Big(
		f_N(\omega_1)f_N(\omega-\omega_1)\chi_{\{\omega_1,\omega-\omega_1\le N\}} \\
		&\qquad\qquad
		- f_N(\omega)f_N(\omega_1)\chi_{\{\omega,\omega_1\le N\}}
		- f_N(\omega)f_N(\omega-\omega_1)\chi_{\{\omega,\omega-\omega_1\le N\}}
		\Big)
		e'\!\left[Uf_N\chi_{\{\omega\le N\}}\right](\omega) \\
		&\quad
		-2\int_{[0,\infty)} \mathrm d\omega
		\int_{[0,\infty)} \mathrm d\omega_1\,
		U(\omega_1)U(\omega_1+\omega)\,\chi_{\{\omega\le N\}}
		\Big(
		f_N(\omega)f_N(\omega_1)\chi_{\{\omega,\omega_1\le N\}} \\
		&\qquad\qquad
		- f_N(\omega_1)f_N(\omega+\omega_1)\chi_{\{\omega+\omega_1,\omega_1\le N\}}
		- f_N(\omega)f_N(\omega+\omega_1)\chi_{\{\omega+\omega_1,\omega\le N\}}
		\Big)
		e'\!\left[Uf_N\chi_{\{\omega\le N\}}\right](\omega).
	\end{aligned}
\end{equation}
	
We then define
\begin{equation}\label{Lemma:EntropyTest:E1at}
	\begin{aligned}
		A_1 := {}&
		\int_{[0,\infty)} \mathrm d\omega
		\int_{[0,\omega]} \mathrm d\omega_1\,
		U(\omega_1)U(\omega-\omega_1)\,\chi_{\{\omega\le N\}}
		\Big(
		f_N(\omega_1) f_N(\omega-\omega_1)\chi_{\{\omega_1,\omega-\omega_1\le N\}} \\
		&\qquad
		- f_N(\omega) f_N(\omega_1)\chi_{\{\omega,\omega_1\le N\}}
		- f_N(\omega) f_N(\omega-\omega_1)\chi_{\{\omega,\omega-\omega_1\le N\}}
		\Big)
		\,e'\!\left[Uf_N\chi_{\{\omega\le N\}}\right](\omega) \\
		&\;-\;
		2\int_{[0,\infty)} \mathrm d\omega
		\int_{[0,\infty)} \mathrm d\omega_1\,
		U(\omega_1)U(\omega_1+\omega)\,
		f_N(\omega) f_N(\omega_1)\,
		\mathbf{1}_{\{\omega_1\le \omega\}}
		\chi_{\{\omega,\omega_1\le N\}}
		\,e'\!\left[Uf_N\chi_{\{\omega\le N\}}\right](\omega)\chi_{\{\omega\le N\}}, \\[0.5em]
		A_2 := {}&
		-2\int_{[0,\infty)} \mathrm d\omega
		\int_{[0,\infty)} \mathrm d\omega_1\,
		U(\omega_1)U(\omega_1+\omega)\,\chi_{\{\omega\le N\}}
		\Big(
		f_N(\omega) f_N(\omega_1)\mathbf{1}_{\{\omega_1\ge \omega\}}\chi_{\{\omega,\omega_1\le N\}} \\
		&\qquad\qquad
		- f_N(\omega_1) f_N(\omega+\omega_1)\chi_{\{\omega+\omega_1,\omega_1\le N\}}
		- f_N(\omega) f_N(\omega+\omega_1)\chi_{\{\omega,\omega+\omega_1\le N\}}
		\Big)
		\,e'\!\left[Uf_N\chi_{\{\omega\le N\}}\right](\omega).
	\end{aligned}
\end{equation}
	
We first estimate \(A_1\). On the support of the first integral in \(A_1\), we have
\(0\le \omega_1\le \omega\) and \(\omega\le N\). Hence
\(\omega_1\le N\) and \(\omega-\omega_1\le N\). By the same symmetry argument as in
\eqref{Lemma:EntropyTest:E3}, we obtain
\begin{equation*}
	\begin{aligned}
		A_1 \le {}&
		2\int_{[0,\infty)} \mathrm d\omega
		\int_{[0,\omega]} \mathrm d\omega_1\,
		U(\omega-\omega_1)f_N(\omega-\omega_1) \\
		&\qquad\qquad\times
		\Big(
		U(\omega_1)f_N(\omega_1)\chi_{\{\omega_1\le N\}}
		-
		U(\omega)f_N(\omega)\chi_{\{\omega\le N\}}
		\Big) \\
		&\qquad\qquad\times
		e'\!\left[Uf_N\chi_{\{\omega\le N\}}\right](\omega)\,
		\mathbf{1}_{\{\omega_1\ge \omega-\omega_1\}}\chi_{\{\omega\le N\}} \\
		&+
		2\int_{[0,\infty)} \mathrm d\omega
		\int_{[0,\omega]} \mathrm d\omega_1\,
		U(\omega-\omega_1)f_N(\omega-\omega_1)\,
		U(\omega)f_N(\omega)\chi_{\{\omega\le N\}} \\
		&\qquad\qquad\times
		e'\!\left[Uf_N\chi_{\{\omega\le N\}}\right](\omega)\,
		\mathbf{1}_{\{\omega_1\ge \omega-\omega_1\}}\\
		&-
		2\int_{[0,\infty)} \mathrm d\omega
		\int_{[0,\omega]} \mathrm d\omega_1\,
		U(\omega)f_N(\omega)\chi_{\{\omega\le N\}}\,
		U(\omega_1)f_N(\omega_1)\chi_{\{\omega_1\le N\}} \\
		&\qquad\qquad\times
		e'\!\left[Uf_N\chi_{\{\omega\le N\}}\right](\omega)
		\ - \ A_3,
	\end{aligned}
\end{equation*}
where
\begin{equation*}
	\begin{aligned}
		A_3:={}&
		2\int_{[0,\infty)} \mathrm d\omega
		\int_{[0,\omega]} \mathrm d\omega_1\,
		U(\omega_1)U(\omega-\omega_1)
		f_N(\omega)f_N(\omega_1)
		\chi_{\{\omega,\omega_1\le N\}}
		e'\!\left[Uf_N\chi_{\{\omega\le N\}}\right](\omega).
	\end{aligned}
\end{equation*}
Now we use explicitly the cutoff version of the convexity inequality. Since
\(e\) is convex, for
\[
a=
U(\omega_1)f_N(\omega_1)\chi_{\{\omega_1\le N\}},
\qquad
b=
U(\omega)f_N(\omega)\chi_{\{\omega\le N\}},
\]
we have
\[
(a-b)e'(b)\le e(a)-e(b).
\]
Therefore,
\begin{equation}\label{Lemma:EntropyTest:CutoffConvexityA1}
	\begin{aligned}
	&\Big(
	U(\omega_1)f_N(\omega_1)\chi_{\{\omega_1\le N\}}
	-
	U(\omega)f_N(\omega)\chi_{\{\omega\le N\}}
	\Big)
	e'\!\left[Uf_N\chi_{\{\omega\le N\}}\right](\omega) \\
	&\qquad\le
	e\!\left[Uf_N\chi_{\{\omega\le N\}}\right](\omega_1)
	-
	e\!\left[Uf_N\chi_{\{\omega\le N\}}\right](\omega).
	\end{aligned}
\end{equation}
Combining this inequality with the previous estimate for \(A_1\), we obtain
\begin{equation*}
	\begin{aligned}
		A_1 \le {}&
		2\int_{[0,\infty)} \mathrm d\omega
		\int_{[0,\omega]} \mathrm d\omega_1\,
		U(\omega-\omega_1)f_N(\omega-\omega_1)
		e\!\left[Uf_N\chi_{\{\omega\le N\}}\right](\omega_1) \\
		&\qquad\qquad\times
		\mathbf{1}_{\{\omega_1\ge \omega-\omega_1\}}\chi_{\{\omega\le N\}} \\
		&-
		2\int_{[0,\infty)} \mathrm d\omega
		\int_{[0,\omega]} \mathrm d\omega_1\,
		U(\omega)f_N(\omega)\chi_{\{\omega\le N\}}\,
		U(\omega_1)f_N(\omega_1)\chi_{\{\omega_1\le N\}} \\
		&\qquad\qquad\times
		e'\!\left[Uf_N\chi_{\{\omega\le N\}}\right](\omega) \\
		&+
		2\int_{[0,\infty)} \mathrm d\omega
		\int_{[0,\omega]} \mathrm d\omega_1\,
		U(\omega-\omega_1)f_N(\omega-\omega_1) \\
		&\qquad\qquad\times
		\Big(
		U(\omega)f_N(\omega)\chi_{\{\omega\le N\}}\,
		e'\!\left[Uf_N\chi_{\{\omega\le N\}}\right](\omega)
		-
		e\!\left[Uf_N\chi_{\{\omega\le N\}}\right](\omega)
		\Big) \\
		&\qquad\qquad\times
		\mathbf{1}_{\{\omega_1\ge \omega-\omega_1\}}\chi_{\{\omega\le N\}}
		\ - \ A_3.
	\end{aligned}
\end{equation*}
Performing the same changes of variables as in \eqref{Lemma:EntropyTest:E3}, namely
\[
(\omega-\omega_1,\omega_1)\mapsto(\omega',\omega),
\qquad
(\omega,\omega_1)\mapsto(\omega,\omega'),
\qquad
(\omega-\omega_1,\omega)\mapsto(\omega',\omega),
\]
respectively, and using the nonnegativity of the corresponding integrands together
with the fact that
\[
\chi_{\{\omega+\omega'\le N\}}
\le
\chi_{\{\omega'\le N\}},
\]
we obtain
\begin{equation}\label{Lemma:EntropyTest:E3t}
	\begin{aligned}
		A_1 \le {}&
		2\int_{[0,\infty)} \mathrm d\omega
		\int_{[0,\omega]} \mathrm d\omega'\,
		U(\omega')f_N(\omega')\chi_{\{\omega'\le N\}}\,
		e\!\left[Uf_N\chi_{\{\omega\le N\}}\right](\omega) \\
		&-
		2\int_{[0,\infty)} \mathrm d\omega
		\int_{[0,\omega]} \mathrm d\omega'\,
		U(\omega)f_N(\omega)\chi_{\{\omega\le N\}}\,
		U(\omega')f_N(\omega')\chi_{\{\omega'\le N\}} \\
		&\qquad\qquad\times
		e'\!\left[Uf_N\chi_{\{\omega\le N\}}\right](\omega) \\
		&+
		2\int_{[0,\infty)} \mathrm d\omega
		\int_{[0,\omega]} \mathrm d\omega'\,
		U(\omega')f_N(\omega')\chi_{\{\omega'\le N\}} \\
		&\qquad\qquad\times
		\Big(
		U(\omega)f_N(\omega)\chi_{\{\omega\le N\}}\,
		e'\!\left[Uf_N\chi_{\{\omega\le N\}}\right](\omega)
		-
		e\!\left[Uf_N\chi_{\{\omega\le N\}}\right](\omega)
		\Big) \\
		&\qquad\qquad\times
		\mathbf{1}_{\{\omega'\le \omega-\omega'\}}\chi_{\{\omega\le N\}}
		\ - \ A_3 \\
		\le {}&
		-2\int_{[0,\infty)} \mathrm d\omega
		\int_{[0,\omega]} \mathrm d\omega'\,
		U(\omega')f_N(\omega')\chi_{\{\omega'\le N\}}\chi_{\{\omega\le N\}} \\
		&\qquad\qquad\times
		\Big(
		U(\omega)f_N(\omega)\chi_{\{\omega\le N\}}\,
		e'\!\left[Uf_N\chi_{\{\omega\le N\}}\right](\omega)
		-
		e\!\left[Uf_N\chi_{\{\omega\le N\}}\right](\omega)
		\Big) \\
		&\qquad\qquad\times
		\mathbf{1}_{\{\omega'\ge \omega-\omega'\}}
		\ - \ A_3.
	\end{aligned}
\end{equation}	
We now estimate \(A_2\). We first consider the second term in \(A_2\).
Since the factor \(\chi_{\{\omega+\omega_1,\omega_1\le N\}}\) implies
\(\omega\le N\), we do not write the cutoff \(\chi_{\{\omega\le N\}}\) in this
term. We write
\begin{equation}\label{Lemma:EntropyTest:E4t}
	\begin{aligned}
		&\;
		2\int_{[0,\infty)} \mathrm d\omega
		\int_{[0,\infty)} \mathrm d\omega_1\,
		U(\omega_1)U(\omega_1+\omega)
		f_N(\omega_1)f_N(\omega+\omega_1)
		\chi_{\{\omega+\omega_1,\omega_1\le N\}} \\
		&\qquad\qquad\times
		e'\!\left[Uf_N\chi_{\{\omega\le N\}}\right](\omega) \\
		={}&
		2\int_{[0,\infty)} \mathrm d\omega
		\int_{[0,\infty)} \mathrm d\omega_1\,
		U(\omega_1+\omega)f_N(\omega+\omega_1)
		\chi_{\{\omega+\omega_1,\omega_1\le N\}} \\
		&\qquad\qquad\times
		\Big[
		U(\omega_1)f_N(\omega_1)\chi_{\{\omega_1\le N\}}
		-
		U(\omega)f_N(\omega)\chi_{\{\omega\le N\}}
		\Big]
		e'\!\left[Uf_N\chi_{\{\omega\le N\}}\right](\omega) \\
		&+
		2\int_{[0,\infty)} \mathrm d\omega
		\int_{[0,\infty)} \mathrm d\omega_1\,
		U(\omega_1+\omega)f_N(\omega+\omega_1)
		U(\omega)f_N(\omega)\chi_{\{\omega\le N\}} \\
		&\qquad\qquad\times
		\chi_{\{\omega+\omega_1,\omega_1\le N\}}
		e'\!\left[Uf_N\chi_{\{\omega\le N\}}\right](\omega).
	\end{aligned}
\end{equation}
By the cutoff convexity inequality, we have
\begin{equation}\label{Lemma:EntropyTest:CutoffConvexityA2}
	\begin{aligned}
	&\Big[
	U(\omega_1)f_N(\omega_1)\chi_{\{\omega_1\le N\}}
	-
	U(\omega)f_N(\omega)\chi_{\{\omega\le N\}}
	\Big]
	e'\!\left[Uf_N\chi_{\{\omega\le N\}}\right](\omega) \\
	&\qquad\le
	e\!\left[Uf_N\chi_{\{\omega\le N\}}\right](\omega_1)
	-
	e\!\left[Uf_N\chi_{\{\omega\le N\}}\right](\omega).
	\end{aligned}
\end{equation}
Thus, we can bound
\begin{equation}\label{Lemma:EntropyTest:E5t}
	\begin{aligned}
		&\;
		2\int_{[0,\infty)} \mathrm d\omega
		\int_{[0,\infty)} \mathrm d\omega_1\,
		U(\omega_1)U(\omega_1+\omega)
		f_N(\omega_1)f_N(\omega+\omega_1)
		\chi_{\{\omega+\omega_1,\omega_1\le N\}} \\
		&\qquad\qquad\times
		e'\!\left[Uf_N\chi_{\{\omega\le N\}}\right](\omega) \\
		\le{}&
		2\int_{[0,\infty)} \mathrm d\omega
		\int_{[0,\infty)} \mathrm d\omega_1\,
		U(\omega_1+\omega)f_N(\omega+\omega_1)
		\chi_{\{\omega+\omega_1,\omega_1\le N\}} \\
		&\qquad\qquad\times
		\Big(
		e\!\left[Uf_N\chi_{\{\omega\le N\}}\right](\omega_1)
		-
		e\!\left[Uf_N\chi_{\{\omega\le N\}}\right](\omega)
		\Big) \\
		&+
		2\int_{[0,\infty)} \mathrm d\omega
		\int_{[0,\infty)} \mathrm d\omega_1\,
		U(\omega_1+\omega)f_N(\omega+\omega_1)
		U(\omega)f_N(\omega)\chi_{\{\omega\le N\}} \\
		&\qquad\qquad\times
		\chi_{\{\omega+\omega_1,\omega_1\le N\}}
		e'\!\left[Uf_N\chi_{\{\omega\le N\}}\right](\omega).
	\end{aligned}
\end{equation}
Since the domain \(\{\omega+\omega_1\le N\}\) is symmetric in \(\omega\) and
\(\omega_1\), and since \(\chi_{\{\omega+\omega_1,\omega_1\le N\}}
=\chi_{\{\omega+\omega_1\le N\}}\), we have
\begin{equation*}
	\begin{aligned}
	&2\int_{[0,\infty)} \mathrm d\omega
	\int_{[0,\infty)} \mathrm d\omega_1\,
	U(\omega_1+\omega)f_N(\omega+\omega_1)
	\chi_{\{\omega+\omega_1,\omega_1\le N\}}
	e\!\left[Uf_N\chi_{\{\omega\le N\}}\right](\omega_1) \\
	={}&
	2\int_{[0,\infty)} \mathrm d\omega
	\int_{[0,\infty)} \mathrm d\omega_1\,
	U(\omega_1+\omega)f_N(\omega+\omega_1)
	\chi_{\{\omega+\omega_1,\omega\le N\}}
	e\!\left[Uf_N\chi_{\{\omega\le N\}}\right](\omega).
	\end{aligned}
\end{equation*}
Therefore the entropy terms in \eqref{Lemma:EntropyTest:E5t} cancel, and we get
\begin{equation}\label{Lemma:EntropyTest:E6t}
	\begin{aligned}
		&\;
		2\int_{[0,\infty)} \mathrm d\omega
		\int_{[0,\infty)} \mathrm d\omega_1\,
		U(\omega_1)U(\omega_1+\omega)
		f_N(\omega_1)f_N(\omega+\omega_1)
		\chi_{\{\omega+\omega_1,\omega_1\le N\}} \\
		&\qquad\qquad\times
		e'\!\left[Uf_N\chi_{\{\omega\le N\}}\right](\omega) \\
		\le{}&
		2\int_{[0,\infty)} \mathrm d\omega
		\int_{[0,\infty)} \mathrm d\omega_1\,
		U(\omega_1+\omega)f_N(\omega+\omega_1)
		U(\omega)f_N(\omega)\chi_{\{\omega\le N\}} \\
		&\qquad\qquad\times
		\chi_{\{\omega+\omega_1,\omega_1\le N\}}
		e'\!\left[Uf_N\chi_{\{\omega\le N\}}\right](\omega).
	\end{aligned}
\end{equation}

We next plug \eqref{Lemma:EntropyTest:E6t} into the expression for \(A_2\).
The third term in \(A_2\) is transformed by the change of variables
\[
(\omega+\omega_1,\omega)\mapsto(\omega_1,\omega),
\]
which gives
\[
\begin{aligned}
&2\int_{[0,\infty)} \mathrm d\omega
\int_{[0,\infty)} \mathrm d\omega_1\,
U(\omega_1)U(\omega_1+\omega)
f_N(\omega)f_N(\omega+\omega_1)
\chi_{\{\omega,\omega+\omega_1\le N\}}
e'\!\left[Uf_N\chi_{\{\omega\le N\}}\right](\omega) \\
={}&
2\int_{[0,\infty)} \mathrm d\omega
\int_{[\omega,\infty)} \mathrm d\omega_1\,
U(\omega_1-\omega)U(\omega_1)
f_N(\omega)f_N(\omega_1)
\chi_{\{\omega,\omega_1\le N\}}
e'\!\left[Uf_N\chi_{\{\omega\le N\}}\right](\omega).
\end{aligned}
\]
Similarly, the upper bound in \eqref{Lemma:EntropyTest:E6t} becomes
\[
\begin{aligned}
&2\int_{[0,\infty)} \mathrm d\omega
\int_{[0,\infty)} \mathrm d\omega_1\,
U(\omega_1+\omega)f_N(\omega+\omega_1)
U(\omega)f_N(\omega)\chi_{\{\omega\le N\}}
\chi_{\{\omega+\omega_1,\omega_1\le N\}}
e'\!\left[Uf_N\chi_{\{\omega\le N\}}\right](\omega) \\
={}&
2\int_{[0,\infty)} \mathrm d\omega
\int_{[\omega,\infty)} \mathrm d\omega_1\,
U(\omega_1)U(\omega)
f_N(\omega_1)f_N(\omega)
\chi_{\{\omega,\omega_1\le N\}}
e'\!\left[Uf_N\chi_{\{\omega\le N\}}\right](\omega).
\end{aligned}
\]
Hence, we can estimate
\begin{equation}\label{Lemma:EntropyTest:E7t}
	\begin{aligned}
		A_2
		\le {}&
		-2\int_{[0,\infty)} \mathrm d\omega
		\int_{[\omega,\infty)} \mathrm d\omega_1\,
		U(\omega_1)U(\omega_1+\omega)
		f_N(\omega)f_N(\omega_1)
		\chi_{\{\omega,\omega_1\le N\}}
		e'\!\left[Uf_N\chi_{\{\omega\le N\}}\right](\omega) \\
		&+
		2\int_{[0,\infty)} \mathrm d\omega
		\int_{[\omega,\infty)} \mathrm d\omega_1\,
		U(\omega_1)U(\omega)
		f_N(\omega)f_N(\omega_1)
		\chi_{\{\omega,\omega_1\le N\}}
		e'\!\left[Uf_N\chi_{\{\omega\le N\}}\right](\omega) \\
		&+
		2\int_{[0,\infty)} \mathrm d\omega
		\int_{[\omega,\infty)} \mathrm d\omega_1\,
		U(\omega_1-\omega)U(\omega_1)
		f_N(\omega)f_N(\omega_1)
		\chi_{\{\omega,\omega_1\le N\}}
		e'\!\left[Uf_N\chi_{\{\omega\le N\}}\right](\omega) \\
		={}&
		-2\int_{[0,\infty)} \mathrm d\omega
		\int_{[\omega,\infty)} \mathrm d\omega_1\,
		U(\omega_1)f_N(\omega)f_N(\omega_1)
		\chi_{\{\omega,\omega_1\le N\}} \\
		&\qquad\qquad\times
		\Big(
		U(\omega_1+\omega)
		-
		U(\omega)
		-
		U(\omega_1-\omega)
		\Big)
		e'\!\left[Uf_N\chi_{\{\omega\le N\}}\right](\omega).
	\end{aligned}
\end{equation}

Since \(U(\omega_1+\omega)\ge U(\omega_1-\omega)+U(\omega)\), it follows from
\eqref{Lemma:EntropyTest:E7t} that \(A_2\le 0\). Combining this with
\eqref{Lemma:EntropyTest:E3t}, we obtain
\begin{equation}\label{Lemma:EntropyTest:E8t}
	\begin{aligned}
		&\int_{[0,\infty)} \mathrm d\omega\,
		\mathcal Q_N[f_N](\omega)\,
		e'\!\left[Uf_N\chi_{\{\omega\le N\}}\right](\omega)
		\chi_{\{\omega\le N\}} \\
		\le {}&
		-2\int_{[0,\infty)} \mathrm d\omega
		\int_{[0,\omega]} \mathrm d\omega'\,
		U(\omega')f_N(\omega')\chi_{\{\omega'\le N\}}\chi_{\{\omega\le N\}} \\
		&\qquad\qquad\times
		\Big(
		U(\omega)f_N(\omega)\chi_{\{\omega\le N\}}\,
		e'\!\left[Uf_N\chi_{\{\omega\le N\}}\right](\omega)
		-
		e\!\left[Uf_N\chi_{\{\omega\le N\}}\right](\omega)
		\Big) \\
		&\qquad\qquad\times
		\mathbf{1}_{\{\omega'\ge \omega-\omega'\}}
		\ - \ A_3.
	\end{aligned}
\end{equation}
The conclusion of the lemma now follows from the inequality
\[
U(\omega)f_N(\omega)\chi_{\{\omega\le N\}}\,
e'\!\left[Uf_N\chi_{\{\omega\le N\}}\right](\omega)
-
e\!\left[Uf_N\chi_{\{\omega\le N\}}\right](\omega)
\ge 0,
\]
which follows from the convexity of \(e\) and \(e(0)=0\). 

Moreover, we have
\[
\int_{[0,\infty)} \mathrm d\omega\,
\partial_t\!\left(
\frac{e\!\left[Uf_N\chi_{\{\omega\le N\}}\right](\omega)}{U(\omega)}
\right)
=
\int_{[0,\infty)} \mathrm d\omega\,
\mathcal Q_N[f_N](\omega)\,
e'\!\left[Uf_N\chi_{\{\omega\le N\}}\right](\omega)\chi_{\{\omega\le N\}}.
\]
Therefore,
\[
\frac{\mathrm d}{\mathrm dt}\,
\mathcal E_e[f_N\chi_{\{\omega\le N\}}]\le0.
\]
Inequality \eqref{eq:dissipation-assumption:1} follows by a straightforward argument.
\end{proof}

\begin{lemma}\label{Lemma:QN:Lipschitz:weighted}
Assume that
\[
U(\omega)\ge 0 \quad \text{for all } \omega\ge0,
\qquad
\mathfrak M_N:=\sup_{\omega\in[0,2N]}(1+U(\omega))<\infty.
\]
Then the cutoff operator \(\mathcal Q_N\) is locally Lipschitz on
\(L^1_{1+U}([0,\infty))\), where
\[
\|f\|_{L^1_{1+U}}
:=
\|(1+U)f\|_{L^1([0,\infty))}
=
\int_{[0,\infty)}(1+U(\omega))|f(\omega)|\,\mathrm d\omega .
\]
More precisely, for every \(f,g\in L^1_{1+U}([0,\infty))\),
\[
\|(1+U)(\mathcal Q_N[f]-\mathcal Q_N[g])\|_{L^1([0,\infty))}
\le
C_N
\left(
\|(1+U)f\|_{L^1([0,\infty))}
+
\|(1+U)g\|_{L^1([0,\infty))}
\right)
\|(1+U)(f-g)\|_{L^1([0,\infty))},
\]
where one may take
\[
C_N=9\mathfrak M_N^3.
\]
Consequently, for every \(R>0\), if
\[
\|f\|_{L^1_{1+U}},\ \|g\|_{L^1_{1+U}}\le R,
\]
then
\[
\|\mathcal Q_N[f]-\mathcal Q_N[g]\|_{L^1_{1+U}}
\le
18R\mathfrak M_N^3
\|f-g\|_{L^1_{1+U}}.
\]
\end{lemma}

\begin{proof}
	Using the Dirac masses, we rewrite \(\mathcal Q_N[f](\omega)\) as
	\begin{equation}\label{QN:weighted:expanded}
		\begin{aligned}
			\mathcal Q_N[f](\omega)
			={}&
			\int_0^\omega \mathrm d\omega_1\,
			U(\omega_1)U(\omega-\omega_1)
			f(\omega_1)f(\omega-\omega_1)
			\chi_{\{\omega_1,\omega-\omega_1\le N\}} \\
			&-
			\int_0^\omega \mathrm d\omega_1\,
			U(\omega_1)U(\omega-\omega_1)
			f(\omega)f(\omega_1)
			\chi_{\{\omega,\omega_1\le N\}} \\
			&-
			\int_0^\omega \mathrm d\omega_1\,
			U(\omega_1)U(\omega-\omega_1)
			f(\omega)f(\omega-\omega_1)
			\chi_{\{\omega,\omega-\omega_1\le N\}} \\
			&-
			2\int_0^\infty \mathrm d\omega_1\,
			U(\omega_1)U(\omega+\omega_1)
			f(\omega_1)f(\omega)
			\chi_{\{\omega,\omega_1\le N\}} \\
			&+
			2\int_0^\infty \mathrm d\omega_1\,
			U(\omega_1)U(\omega+\omega_1)
			f(\omega_1)f(\omega+\omega_1)
			\chi_{\{\omega+\omega_1,\omega_1\le N\}} \\
			&+
			2\int_0^\infty \mathrm d\omega_1\,
			U(\omega_1)U(\omega+\omega_1)
			f(\omega)f(\omega+\omega_1)
			\chi_{\{\omega+\omega_1,\omega\le N\}} .
		\end{aligned}
	\end{equation}
	
	Accordingly, we define
	\[
	\mathcal Q_N[f]
	=
	Q_N^{(1)}[f]-Q_N^{(2)}[f]-Q_N^{(3)}[f]
	-2Q_N^{(4)}[f]+2Q_N^{(5)}[f]+2Q_N^{(6)}[f],
	\]
	where
	\begin{align*}
		Q_N^{(1)}[f](\omega)
		:={}&
		\int_0^\omega \mathrm d\omega_1\,
		U(\omega_1)U(\omega-\omega_1)
		f(\omega_1)f(\omega-\omega_1)
		\chi_{\{\omega_1,\omega-\omega_1\le N\}},\\
		Q_N^{(2)}[f](\omega)
		:={}&
		\int_0^\omega \mathrm d\omega_1\,
		U(\omega_1)U(\omega-\omega_1)
		f(\omega)f(\omega_1)
		\chi_{\{\omega,\omega_1\le N\}},\\
		Q_N^{(3)}[f](\omega)
		:={}&
		\int_0^\omega \mathrm d\omega_1\,
		U(\omega_1)U(\omega-\omega_1)
		f(\omega)f(\omega-\omega_1)
		\chi_{\{\omega,\omega-\omega_1\le N\}},\\
		Q_N^{(4)}[f](\omega)
		:={}&
		\int_0^\infty \mathrm d\omega_1\,
		U(\omega_1)U(\omega+\omega_1)
		f(\omega_1)f(\omega)
		\chi_{\{\omega,\omega_1\le N\}},\\
		Q_N^{(5)}[f](\omega)
		:={}&
		\int_0^\infty \mathrm d\omega_1\,
		U(\omega_1)U(\omega+\omega_1)
		f(\omega_1)f(\omega+\omega_1)
		\chi_{\{\omega+\omega_1,\omega_1\le N\}},\\
		Q_N^{(6)}[f](\omega)
		:={}&
		\int_0^\infty \mathrm d\omega_1\,
		U(\omega_1)U(\omega+\omega_1)
		f(\omega)f(\omega+\omega_1)
		\chi_{\{\omega+\omega_1,\omega\le N\}}.
	\end{align*}
	
	Since
	\[
\mathfrak	M_N=\sup_{\omega\in[0,2N]}(1+U(\omega))<\infty,
	\]
	we have
	\[
	0\le U(\omega)\le \mathfrak M_N,\qquad \omega\in[0,2N].
	\]

	We estimate each term separately.
	
	\medskip
	
	\noindent
	\textit{Estimate for \(Q_N^{(1)}\).}
	Using
	\[
	f(\omega_1)f(\omega-\omega_1)-g(\omega_1)g(\omega-\omega_1)
	=
	\bigl(f(\omega_1)-g(\omega_1)\bigr)f(\omega-\omega_1)
	+
	g(\omega_1)\bigl(f(\omega-\omega_1)-g(\omega-\omega_1)\bigr),
	\]
	we obtain
	\[
	\begin{aligned}
		&|Q_N^{(1)}[f](\omega)-Q_N^{(1)}[g](\omega)|\\
		\le{}&
		\int_0^\omega \mathrm d\omega_1\,
		U(\omega_1)U(\omega-\omega_1)
		\chi_{\{\omega_1,\omega-\omega_1\le N\}}
		|f(\omega_1)-g(\omega_1)|\,|f(\omega-\omega_1)|\\
		&+
		\int_0^\omega \mathrm d\omega_1\,
		U(\omega_1)U(\omega-\omega_1)
		\chi_{\{\omega_1,\omega-\omega_1\le N\}}
		|g(\omega_1)|\,|f(\omega-\omega_1)-g(\omega-\omega_1)|.
	\end{aligned}
	\]
	Multiplying by \(1+U(\omega)\), integrating in \(\omega\), and using that
	\(\omega\le 2N\) on the support of the integrand, we get
	\[
	\begin{aligned}
		&\|(1+U)\bigl(Q_N^{(1)}[f]-Q_N^{(1)}[g]\bigr)\|_{L^1([0,\infty))}\\
		\le{}&
	\mathfrak M_N^3
		\int_0^\infty \mathrm d\omega
		\int_0^\omega \mathrm d\omega_1\,
		\chi_{\{\omega_1,\omega-\omega_1\le N\}}
		|f(\omega_1)-g(\omega_1)|\,|f(\omega-\omega_1)|\\
		&+
		\mathfrak M_N^3
		\int_0^\infty \mathrm d\omega
		\int_0^\omega \mathrm d\omega_1\,
		\chi_{\{\omega_1,\omega-\omega_1\le N\}}
		|g(\omega_1)|\,|f(\omega-\omega_1)-g(\omega-\omega_1)|.
	\end{aligned}
	\]
	Applying Fubini's theorem and the change of variables
	\(\omega_2=\omega-\omega_1\), we find
	\[
	\begin{aligned}
		&\|(1+U)\bigl(Q_N^{(1)}[f]-Q_N^{(1)}[g]\bigr)\|_{L^1([0,\infty))}\\
		\le{}&
		\mathfrak M_N^3
		\int_0^N \mathrm d\omega_1
		\int_0^N \mathrm d\omega_2\,
		|f(\omega_1)-g(\omega_1)|\,|f(\omega_2)|\\
		&+
	\mathfrak	M_N^3
		\int_0^N \mathrm d\omega_1
		\int_0^N \mathrm d\omega_2\,
		|g(\omega_1)|\,|f(\omega_2)-g(\omega_2)|\\
		\le{}&
	\mathfrak	M_N^3
		\bigl(\|f\|_{L^1([0,\infty))}+\|g\|_{L^1([0,\infty))}\bigr)
		\|f-g\|_{L^1([0,\infty))}.
	\end{aligned}
	\]
	Since \(1+U\ge 1\), we have
	\[
	\|h\|_{L^1([0,\infty))}
	\le \|(1+U)h\|_{L^1([0,\infty))}.
	\]
	Therefore,
	\begin{equation}\label{QN:weighted:Q1}
		\begin{aligned}
			&\|(1+U)\bigl(Q_N^{(1)}[f]-Q_N^{(1)}[g]\bigr)\|_{L^1([0,\infty))}\\
			\le{}&
		\mathfrak	M_N^3
			\bigl(
			\|(1+U)f\|_{L^1([0,\infty))}
			+
			\|(1+U)g\|_{L^1([0,\infty))}
			\bigr)
			\|(1+U)(f-g)\|_{L^1([0,\infty))}.
		\end{aligned}
	\end{equation}
	
	\medskip
	
	\noindent
	\textit{Estimate for \(Q_N^{(2)}\).}
	Since
	\[
	f(\omega)f(\omega_1)-g(\omega)g(\omega_1)
	=
	\bigl(f(\omega)-g(\omega)\bigr)f(\omega_1)
	+
	g(\omega)\bigl(f(\omega_1)-g(\omega_1)\bigr),
	\]
	we have
	\[
	\begin{aligned}
		|Q_N^{(2)}[f](\omega)-Q_N^{(2)}[g](\omega)|
		\le{}&
		\int_0^\omega \mathrm d\omega_1\,
		U(\omega_1)U(\omega-\omega_1)
		\chi_{\{\omega,\omega_1\le N\}}
		|f(\omega)-g(\omega)|\,|f(\omega_1)|\\
		&+
		\int_0^\omega \mathrm d\omega_1\,
		U(\omega_1)U(\omega-\omega_1)
		\chi_{\{\omega,\omega_1\le N\}}
		|g(\omega)|\,|f(\omega_1)-g(\omega_1)|.
	\end{aligned}
	\]
	Hence
	\[
	\begin{aligned}
		&\|(1+U)\bigl(Q_N^{(2)}[f]-Q_N^{(2)}[g]\bigr)\|_{L^1([0,\infty))}\\
		\le{}&
		\mathfrak M_N^3
		\int_0^N \mathrm d\omega
		\int_0^\omega \mathrm d\omega_1\,
		|f(\omega)-g(\omega)|\,|f(\omega_1)|\\
		&+
		\mathfrak M_N^3
		\int_0^N \mathrm d\omega
		\int_0^\omega \mathrm d\omega_1\,
		|g(\omega)|\,|f(\omega_1)-g(\omega_1)|\\
		\le{}&
		\mathfrak M_N^3
		\bigl(
		\|(1+U)f\|_{L^1([0,\infty))}
		+
		\|(1+U)g\|_{L^1([0,\infty))}
		\bigr)
		\|(1+U)(f-g)\|_{L^1([0,\infty))}.
	\end{aligned}
	\]
	Thus
	\begin{equation}\label{QN:weighted:Q2}
		\begin{aligned}
			&\|(1+U)\bigl(Q_N^{(2)}[f]-Q_N^{(2)}[g]\bigr)\|_{L^1([0,\infty))}\\
			\le{}&
			\mathfrak M_N^3
			\bigl(
			\|(1+U)f\|_{L^1([0,\infty))}
			+
			\|(1+U)g\|_{L^1([0,\infty))}
			\bigr)
			\|(1+U)(f-g)\|_{L^1([0,\infty))}.
		\end{aligned}
	\end{equation}
	
	\medskip
	
	\noindent
	\textit{Estimate for \(Q_N^{(3)}\).}
Using	exactly the same argument as above, we obtain
	\begin{equation}\label{QN:weighted:Q3}
		\begin{aligned}
			&\|(1+U)\bigl(Q_N^{(3)}[f]-Q_N^{(3)}[g]\bigr)\|_{L^1([0,\infty))}\\
			\le{}&
			\mathfrak M_N^3
			\bigl(
			\|(1+U)f\|_{L^1([0,\infty))}
			+
			\|(1+U)g\|_{L^1([0,\infty))}
			\bigr)
			\|(1+U)(f-g)\|_{L^1([0,\infty))}.
		\end{aligned}
	\end{equation}
	
	\medskip
	
	\noindent
	\textit{Estimate for \(Q_N^{(4)}\).}
	Using
	\[
	f(\omega_1)f(\omega)-g(\omega_1)g(\omega)
	=
	\bigl(f(\omega_1)-g(\omega_1)\bigr)f(\omega)
	+
	g(\omega_1)\bigl(f(\omega)-g(\omega)\bigr),
	\]
	we get
	\[
	\begin{aligned}
		|Q_N^{(4)}[f](\omega)-Q_N^{(4)}[g](\omega)|
		\le{}&
		\int_0^\infty \mathrm d\omega_1\,
		U(\omega_1)U(\omega+\omega_1)
		\chi_{\{\omega,\omega_1\le N\}}
		|f(\omega_1)-g(\omega_1)|\,|f(\omega)|\\
		&+
		\int_0^\infty \mathrm d\omega_1\,
		U(\omega_1)U(\omega+\omega_1)
		\chi_{\{\omega,\omega_1\le N\}}
		|g(\omega_1)|\,|f(\omega)-g(\omega)|.
	\end{aligned}
	\]
	Therefore
	\[
	\begin{aligned}
		&\|(1+U)\bigl(Q_N^{(4)}[f]-Q_N^{(4)}[g]\bigr)\|_{L^1([0,\infty))}\\
		\le{}&
		\mathfrak M_N^3
		\int_0^N \mathrm d\omega
		\int_0^N \mathrm d\omega_1\,
		|f(\omega_1)-g(\omega_1)|\,|f(\omega)|\\
		&+
		\mathfrak M_N^3
		\int_0^N \mathrm d\omega
		\int_0^N \mathrm d\omega_1\,
		|g(\omega_1)|\,|f(\omega)-g(\omega)|\\
		\le{}&
		\mathfrak M_N^3
		\bigl(
		\|(1+U)f\|_{L^1([0,\infty))}
		+
		\|(1+U)g\|_{L^1([0,\infty))}
		\bigr)
		\|(1+U)(f-g)\|_{L^1([0,\infty))}.
	\end{aligned}
	\]
	Hence
	\begin{equation}\label{QN:weighted:Q4}
		\begin{aligned}
			&\|(1+U)\bigl(Q_N^{(4)}[f]-Q_N^{(4)}[g]\bigr)\|_{L^1([0,\infty))}\\
			\le{}&
			\mathfrak M_N^3
			\bigl(
			\|(1+U)f\|_{L^1([0,\infty))}
			+
			\|(1+U)g\|_{L^1([0,\infty))}
			\bigr)
			\|(1+U)(f-g)\|_{L^1([0,\infty))}.
		\end{aligned}
	\end{equation}
	
	\medskip
	
	\noindent
	\textit{Estimate for \(Q_N^{(5)}\).}
	Observing that
	\[
	f(\omega_1)f(\omega+\omega_1)-g(\omega_1)g(\omega+\omega_1)
	=
	\bigl(f(\omega_1)-g(\omega_1)\bigr)f(\omega+\omega_1)
	+
	g(\omega_1)\bigl(f(\omega+\omega_1)-g(\omega+\omega_1)\bigr),
	\]
	we find
	\[
	\begin{aligned}
		|Q_N^{(5)}[f](\omega)-Q_N^{(5)}[g](\omega)|
		\le{}&
		\int_0^\infty \mathrm d\omega_1\,
		U(\omega_1)U(\omega+\omega_1)
		\chi_{\{\omega+\omega_1,\omega_1\le N\}}
		|f(\omega_1)-g(\omega_1)|\,|f(\omega+\omega_1)|\\
		&+
		\int_0^\infty \mathrm d\omega_1\,
		U(\omega_1)U(\omega+\omega_1)
		\chi_{\{\omega+\omega_1,\omega_1\le N\}}
		|g(\omega_1)|\,|f(\omega+\omega_1)-g(\omega+\omega_1)|.
	\end{aligned}
	\]
	Integrating in \(\omega\) and using the change of variables
	\[
	\eta=\omega_1,\qquad \xi=\omega+\omega_1,
	\]
	whose Jacobian equals \(1\), we obtain
	\[
	\begin{aligned}
		&\|(1+U)\bigl(Q_N^{(5)}[f]-Q_N^{(5)}[g]\bigr)\|_{L^1([0,\infty))}\\
		\le{}&
		\mathfrak M_N^3
		\int_0^N \mathrm d\eta
		\int_\eta^N \mathrm d\xi\,
		|f(\eta)-g(\eta)|\,|f(\xi)|\\
		&+
		\mathfrak M_N^3
		\int_0^N \mathrm d\eta
		\int_\eta^N \mathrm d\xi\,
		|g(\eta)|\,|f(\xi)-g(\xi)|\\
		\le{}&
		\mathfrak M_N^3
		\bigl(
		\|(1+U)f\|_{L^1([0,\infty))}
		+
		\|(1+U)g\|_{L^1([0,\infty))}
		\bigr)
		\|(1+U)(f-g)\|_{L^1([0,\infty))}.
	\end{aligned}
	\]
	Hence
	\begin{equation}\label{QN:weighted:Q5}
		\begin{aligned}
			&\|(1+U)\bigl(Q_N^{(5)}[f]-Q_N^{(5)}[g]\bigr)\|_{L^1([0,\infty))}\\
			\le{}&
			\mathfrak M_N^3
			\bigl(
			\|(1+U)f\|_{L^1([0,\infty))}
			+
			\|(1+U)g\|_{L^1([0,\infty))}
			\bigr)
			\|(1+U)(f-g)\|_{L^1([0,\infty))}.
		\end{aligned}
	\end{equation}
	
	\medskip
	
	\noindent
	\textit{Estimate for \(Q_N^{(6)}\).}
	Finally, we write
	\[
	f(\omega)f(\omega+\omega_1)-g(\omega)g(\omega+\omega_1)
	=
	\bigl(f(\omega)-g(\omega)\bigr)f(\omega+\omega_1)
	+
	g(\omega)\bigl(f(\omega+\omega_1)-g(\omega+\omega_1)\bigr),
	\]
	and therefore
	\[
	\begin{aligned}
		|Q_N^{(6)}[f](\omega)-Q_N^{(6)}[g](\omega)|
		\le{}&
		\int_0^\infty \mathrm d\omega_1\,
		U(\omega_1)U(\omega+\omega_1)
		\chi_{\{\omega+\omega_1,\omega\le N\}}
		|f(\omega)-g(\omega)|\,|f(\omega+\omega_1)|\\
		&+
		\int_0^\infty \mathrm d\omega_1\,
		U(\omega_1)U(\omega+\omega_1)
		\chi_{\{\omega+\omega_1,\omega\le N\}}
		|g(\omega)|\,|f(\omega+\omega_1)-g(\omega+\omega_1)|.
	\end{aligned}
	\]
	Thus, we find
	\[
	\begin{aligned}
		&\|(1+U)\bigl(Q_N^{(6)}[f]-Q_N^{(6)}[g]\bigr)\|_{L^1([0,\infty))}\\
		\le{}&
		\mathfrak M_N^3
		\int_0^N \mathrm d\omega
		\int_0^{N-\omega} \mathrm d\omega_1\,
		|f(\omega)-g(\omega)|\,|f(\omega+\omega_1)|\\
		&+
		\mathfrak M_N^3
		\int_0^N \mathrm d\omega
		\int_0^{N-\omega} \mathrm d\omega_1\,
		|g(\omega)|\,|f(\omega+\omega_1)-g(\omega+\omega_1)|\\
		\le{}&
		\mathfrak M_N^3
		\bigl(
		\|(1+U)f\|_{L^1([0,\infty))}
		+
		\|(1+U)g\|_{L^1([0,\infty))}
		\bigr)
		\|(1+U)(f-g)\|_{L^1([0,\infty))}.
	\end{aligned}
	\]
	Hence
	\begin{equation}\label{QN:weighted:Q6}
		\begin{aligned}
			&\|(1+U)\bigl(Q_N^{(6)}[f]-Q_N^{(6)}[g]\bigr)\|_{L^1([0,\infty))}\\
			\le{}&
			\mathfrak M_N^3
			\bigl(
			\|(1+U)f\|_{L^1([0,\infty))}
			+
			\|(1+U)g\|_{L^1([0,\infty))}
			\bigr)
			\|(1+U)(f-g)\|_{L^1([0,\infty))}.
		\end{aligned}
	\end{equation}
	
	Combining \eqref{QN:weighted:Q1}--\eqref{QN:weighted:Q6}, we conclude that
	\[
	\begin{aligned}
		&\|(1+U)\bigl(\mathcal Q_N[f]-\mathcal Q_N[g]\bigr)\|_{L^1([0,\infty))}\\
		\le{}&
		\|(1+U)\bigl(Q_N^{(1)}[f]-Q_N^{(1)}[g]\bigr)\|_{L^1([0,\infty))}
		+\|(1+U)\bigl(Q_N^{(2)}[f]-Q_N^{(2)}[g]\bigr)\|_{L^1([0,\infty))}\\
		&+
		\|(1+U)\bigl(Q_N^{(3)}[f]-Q_N^{(3)}[g]\bigr)\|_{L^1([0,\infty))}
		+2\|(1+U)\bigl(Q_N^{(4)}[f]-Q_N^{(4)}[g]\bigr)\|_{L^1([0,\infty))}\\
		&+
		2\|(1+U)\bigl(Q_N^{(5)}[f]-Q_N^{(5)}[g]\bigr)\|_{L^1([0,\infty))}
		+2\|(1+U)\bigl(Q_N^{(6)}[f]-Q_N^{(6)}[g]\bigr)\|_{L^1([0,\infty))}\\
		\le{}&
		9\mathfrak M_N^3
		\bigl(
		\|(1+U)f\|_{L^1([0,\infty))}
		+
		\|(1+U)g\|_{L^1([0,\infty))}
		\bigr)
		\|(1+U)(f-g)\|_{L^1([0,\infty))}.
	\end{aligned}
	\]
	This proves the lemma.
\end{proof}

\subsection{Global existence of the approximated equations}
\begin{proposition}\label{Prop:cutoff:global}
Assume that
\[
f_0\ge 0,
\qquad
f_0\in L^1_{1+U}\bigl([0,\infty)\bigr),
\]
and define
\[
f_{N,0}(\omega):=f_0(\omega)\chi_{\{\omega\le N\}}.
\]
Assume moreover that \(U\) satisfies the condition of Lemma \ref{Lemma:EntropyTestCutoff} and
\[
\mathfrak m_N:=\inf_{\omega\in[N,2N]}U(\omega)>0,
\]
\[
\sup_{\omega\in[0,2N]} U(\omega)<\infty .
\]
Then \eqref{Cutoff2-new} admits a unique global solution
\[
f_N\in
C\Bigl([0,\infty);L^1_{1+U}\bigl([0,\infty)\bigr)\Bigr)
\cap
C^1\Bigl([0,\infty);L^1_{1+U}\bigl([0,\infty)\bigr)\Bigr),
\]
where the weighted norm is defined in Lemma \ref{Lemma:QN:Lipschitz:weighted}.

Moreover, \(f_N\) is nonnegative and satisfies
\[
\operatorname{supp} f_N(t,\cdot)\subset [0,2N]
\qquad\text{for all }t\ge 0.
\]
In addition, for all \(t\ge 0\),
\[
\int_{[0,\infty)} \mathrm d\omega\, U(\omega)\omega f_N(t,\omega)
=
\int_{[0,\infty)} \mathrm d\omega\, U(\omega)\omega f_{N,0}(\omega),
\]
and
\[
\int_{[0,N]} \mathrm d\omega\, f_N(t,\omega)
\le
\int_{[0,N]} \mathrm d\omega\, f_{N,0}(\omega).
\]
Finally,  let
\(e:[0,\infty)\to[0,\infty)\) be a  function satisfying \[
e\in C^1([0,\infty)), \qquad e \text{ is convex}, \qquad e(0)=0, \qquad e\ge 0.
\] 
If
\[
\int_{[0,\infty)} \mathrm d\omega\,
\frac{e\!\left(U(\omega)f_0(\omega)\right)}{U(\omega)}<\infty,
\]
then, for all \(t\ge 0\),
\[
\int_{[0,N]} \mathrm d\omega\,
\frac{e\!\left(U(\omega)f_N(t,\omega)\right)}{U(\omega)}
\le
\int_{[0,\infty)} \mathrm d\omega\,
\frac{e\!\left(U(\omega)f_0(\omega)\right)}{U(\omega)}.
\]
\end{proposition}

\begin{proof}
We divide the proof into several steps.

\medskip

\noindent
\textit{Step 1: Local existence and uniqueness.}
By Lemma \ref{Lemma:QN:Lipschitz:weighted}, the map
\[
f\mapsto \mathcal Q_N[f]
\]
is locally Lipschitz on
\(
L^1_{1+U}\bigl([0,\infty)\bigr).
\)
Therefore, the Cauchy--Lipschitz theorem in Banach spaces yields a maximal time
\(T_{\max}\in(0,\infty]\) and a unique solution to  \eqref{Cutoff2-new} 
\[
f_N\in
C\Bigl([0,T_{\max});L^1_{1+U}\bigl([0,\infty)\bigr)\Bigr)
\cap
C^1\Bigl([0,T_{\max});L^1_{1+U}\bigl([0,\infty)\bigr)\Bigr).
\]
Moreover, if \(T_{\max}<\infty\), then
\[
\limsup_{t\uparrow T_{\max}}
\|f_N(t)\|_{L^1_{1+U}}
=+\infty.
\]

\medskip

\noindent
\textit{Step 2: Positivity.}
We write
\[
\mathcal Q_N[f](\omega)
=
G_N[f](\omega)
-
L_N[f](\omega)f(\omega)
+
A_N[f](\omega)f(\omega),
\]
where
\[
\begin{aligned}
G_N[f](\omega)
:={}&
\int_0^\omega \mathrm d\omega_1\,
U(\omega_1)U(\omega-\omega_1)
f(\omega_1)f(\omega-\omega_1)
\chi_{\{\omega_1,\omega-\omega_1\le N\}} \\
&+
2\int_0^\infty \mathrm d\omega_1\,
U(\omega_1)U(\omega+\omega_1)
f(\omega_1)f(\omega+\omega_1)
\chi_{\{\omega+\omega_1,\omega_1\le N\}},
\end{aligned}
\]
\[
\begin{aligned}
L_N[f](\omega)
:={}&
\int_0^\omega \mathrm d\omega_1\,
U(\omega_1)U(\omega-\omega_1)
\chi_{\{\omega,\omega_1\le N\}}\,f(\omega_1) \\
&+
\int_0^\omega \mathrm d\omega_1\,
U(\omega_1)U(\omega-\omega_1)
\chi_{\{\omega,\omega-\omega_1\le N\}}\,f(\omega-\omega_1) \\
&+
2\int_0^\infty \mathrm d\omega_1\,
U(\omega_1)U(\omega+\omega_1)
\chi_{\{\omega,\omega_1\le N\}}\,f(\omega_1),
\end{aligned}
\]
and
\[
A_N[f](\omega)
:=
2\int_0^\infty \mathrm d\omega_1\,
U(\omega_1)U(\omega+\omega_1)
\chi_{\{\omega+\omega_1,\omega\le N\}}\,f(\omega+\omega_1).
\]
If \(f\ge0\), then
\[
G_N[f]\ge0,\qquad L_N[f]\ge0,\qquad A_N[f]\ge0.
\]

For a.e. \(\omega\ge0\), we have
\[
\partial_t f_N(t,\omega)
+
\bigl(L_N[f_N(t)](\omega)-A_N[f_N(t)](\omega)\bigr)
f_N(t,\omega)
=
G_N[f_N(t)](\omega).
\]
Hence, we can write
\[
\begin{aligned}
f_N(t,\omega)
={}&
f_{N,0}(\omega)
\exp\!\left(
-\int_0^t
\bigl(L_N[f_N(s)](\omega)-A_N[f_N(s)](\omega)\bigr)
\,\mathrm ds
\right) \\
&+
\int_0^t
\exp\!\left(
-\int_s^t
\bigl(L_N[f_N(\tau)](\omega)-A_N[f_N(\tau)](\omega)\bigr)
\,\mathrm d\tau
\right)
G_N[f_N(s)](\omega)\,\mathrm ds .
\end{aligned}
\]
Moreover,  if \(f\ge0\), then
\[
\mathcal Q_N[f](\omega)=G_N[f](\omega)\ge0.
\]
Since \(f_{N,0}\ge0\), it follows that
\[
f_N(t,\omega)\ge0
\qquad\text{for a.e. }\omega\ge0,\quad t\in[0,T_{\max}).
\]

\medskip

\noindent
\textit{Step 3: Support property.}
We prove that
\[
f_N(t,\omega)=0
\qquad\text{for a.e. }\omega>2N.
\]

We let \(\omega>2N\). If the $G_N$ term is nonzero, then
\(\omega_1\le N\) and \(\omega-\omega_1\le N\), which implies
\(\omega\le 2N\), a contradiction. Hence the $G_N$ term vanishes.

In the remaining terms, each cutoff enforces either \(\omega\le N\) or
\(\omega+\omega_1\le N\), which cannot hold when \(\omega>2N\).
Thus \(\mathcal Q_N[f](\omega)=0\).

Thus, for a.e. \(\omega>2N\),
\[
\mathcal Q_N[f](\omega)=0.
\]
Since \(f_{N,0}(\omega)=0\) for a.e. \(\omega>2N\) and the equation gives
\[
\partial_t f_N(t,\omega)=0
\qquad\text{for a.e. }\omega>2N,
\]
then,
\[
f_N(t,\omega)=0
\qquad\text{for a.e. }\omega>2N,
\qquad t\in[0,T_{\max}).
\]
\medskip

\noindent
\textit{Step 4: A priori estimates.}
By Lemma \ref{Lemma:EntropyTestCutoff}, we have
\[
\frac{\mathrm d}{\mathrm dt}\int_{[0,N]} \mathrm d\omega\, f_N
\le 0,
\qquad
\int_{[0,N]} \mathrm d\omega\, f_N(t)
\le
\int_{[0,N]} \mathrm d\omega\, f_{N,0}.
\]

By Lemma \ref{Lemma:EnergyCutoff}, we also have
\[
\int_{[0,\infty)} \mathrm d\omega\, U(\omega)\omega f_N(t,\omega)
=
\int_{[0,\infty)} \mathrm d\omega\, U(\omega)\omega f_{N,0}(\omega).
\]

Since the support is contained in \([0,2N]\), we split:
\[
\int_{[0,\infty)} \mathrm d\omega\, (1+U(\omega))f_N
=
\int_{[0,N]} \mathrm d\omega\, (1+U(\omega))f_N
+
\int_{[N,2N]} \mathrm d\omega\, (1+U(\omega))f_N.
\]

On \([0,N]\), we observe that
\[
\int_{[0,N]} \mathrm d\omega\, (1+U(\omega))f_N
\le
\left(1+\sup_{[0,N]}U\right)\int_{[0,N]} \mathrm d\omega\, f_N.
\]

On \([N,2N]\), using \(U(\omega)\ge \mathfrak m_N\) and \(\omega\ge N\), we find
\[
1+U(\omega)\le \mathfrak C\,U(\omega)\omega,
\]
for some constant \(\mathfrak C>0\) independent of the initial condition but possibly depending on \(N\), which implies
\[
\int_{[N,2N]} \mathrm d\omega\, (1+U(\omega))f_N
\le
\mathfrak C \int_{[N,2N]} \mathrm d\omega\, U(\omega)\omega f_N.
\]

Combining the above estimates gives us a uniform bound in time.

 Since the uniform bound prevents blow-up, we deduce \(T_{\max}=\infty\).

\medskip

\noindent
\textit{Step 5: Entropy estimate.}
This follows from Lemma \ref{Lemma:EntropyTestCutoff}.

\end{proof}

\subsection{A compactness in time lemma}

\begin{lemma}\label{lem:compactness-local-measures}
	Let \(T>0\). Let \(\{f_n\}_{n\ge1}\) be a sequence of nonnegative functions on
	\([0,T]\times(0,\infty)\) such that:
	
	\begin{enumerate}
		\item for every compact set \(K\Subset(0,\infty)\),
		\begin{equation}\label{eq:compactness-local-mass}
			\sup_{n\ge1}\sup_{t\in[0,T]}\int_K \mathrm d\omega\, f_n(t,\omega)<\infty;
		\end{equation}
		
		\item for every \(\varphi\in C_c^2((0,\infty))\), the functions
		\[
		F_n^\varphi(t):=\int_0^\infty \mathrm d\omega\, \varphi(\omega)f_n(t,\omega),
		\qquad t\in[0,T],
		\]
		are  bounded and equicontinuous on \([0,T]\), uniformly in \(n\).
	\end{enumerate}
	
	Then there exist a subsequence, still denoted by \(\{f_n\}\), and a map
	\[
	f\in C\bigl([0,T];w\text{-}\mathcal M_{\mathrm{loc}}((0,\infty))\bigr)
	\]
	such that, for every \(\varphi\in C_c((0,\infty))\),
	\[
	\int_0^\infty \mathrm d\omega\, \varphi(\omega)f_n(t,\omega)
	\longrightarrow
	\int_0^\infty f(t,\mathrm d\omega)\,\varphi(\omega)
	\qquad\text{uniformly for }t\in[0,T]  \mbox{ as } n\to\infty.
	\]
	Equivalently,
	\[
	f_n\longrightarrow f
	\qquad\text{in }C\bigl([0,T];w\text{-}\mathcal M_{\mathrm{loc}}((0,\infty))\bigr)
	\qquad\text{as } n\to\infty.
	\]
\end{lemma}

\begin{proof}
	We choose a countable family \(\{\varphi_k\}_{k\ge1}\subset C_c^2((0,\infty))\)
	which is dense in \(C_c((0,\infty))\) in the following sense: for every compact
	interval \(K=[a,b]\Subset(0,\infty)\) and every
	\(\psi\in C_c((0,\infty))\) with \(\operatorname{supp}\psi\Subset(a,b)\),
	there exists a sequence \(\varphi_{k_m}\) with
	\(\operatorname{supp}\varphi_{k_m}\Subset(a,b)\) such that
	\[
	\lim_{m\to\infty}\|\varphi_{k_m}-\psi\|_{L^\infty(K)}=0.
	\]
	
	For each \(k\), we define
	\[
	F_n^{\varphi_k}(t):=\int_0^\infty \mathrm d\omega\, \varphi_k(\omega)f_n(t,\omega).
	\]
	By assumption, for each fixed \(k\), the family
	\(\{F_n^{\varphi_k}\}_{n\ge1}\) is uniformly bounded and equicontinuous on
	\([0,T]\). Hence, by the Arzel\`a--Ascoli theorem, it is relatively compact in
	\(C([0,T])\). By a diagonal extraction, there exists a subsequence, still denoted by
	\(\{f_n\}\), such that for every \(k\ge1\),
	\[
	F_n^{\varphi_k}\longrightarrow F^{\varphi_k}
	\qquad\text{in }C([0,T])
	\qquad\text{as } n\to\infty
	\]
	for some \(F^{\varphi_k}\in C([0,T])\).
	
	We now define the limit against arbitrary compactly supported continuous test functions.
	We fix \(t\in[0,T]\), and let \(\psi\in C_c((0,\infty))\). We choose
	\(0<a<b<\infty\) such that
	\[
	\operatorname{supp}\psi\Subset(a,b),
	\]
	and set \(K=[a,b]\). Let \(\{\varphi_{k_m}\}_{m\ge1}\subset\{\varphi_k\}_{k\ge1}\)
	be such that
	\[
	\operatorname{supp}\varphi_{k_m}\Subset(a,b),
	\qquad
	\|\varphi_{k_m}-\psi\|_{L^\infty(K)}\to0
	\qquad\text{as }m\to\infty.
	\]
 
	It follows that, for each fixed \(t\), the sequence
	\(\{F^{\varphi_{k_m}}(t)\}_{m\ge1}\) is Cauchy. Indeed, for \(m,\ell\) large, we can estimate
	\[
	\begin{aligned}
		|F^{\varphi_{k_m}}(t)-F^{\varphi_{k_\ell}}(t)|
		&\le
		\limsup_{n\to\infty}
		\left|
		\int_0^\infty \mathrm d\omega\,
		(\varphi_{k_m}-\varphi_{k_\ell})(\omega)f_n(t,\omega)
		\right| \\
		&\le
		\|\varphi_{k_m}-\varphi_{k_\ell}\|_{L^\infty(K)}
		\sup_{n\ge1}\sup_{s\in[0,T]}\int_K \mathrm d\omega\, f_n(s,\omega),
	\end{aligned}
	\]
	and the right-hand side tends to \(0\) as \(m,\ell\to\infty\). Moreover, the limit
	does not depend on the chosen approximating sequence. We may therefore define
	\[
	F^\psi(t):=\lim_{m\to\infty}F^{\varphi_{k_m}}(t).
	\]
	
	By construction, for each fixed \(t\in[0,T]\), the map
	\[
	\psi\longmapsto F^\psi(t)
	\]
	is linear on \(C_c((0,\infty))\). Moreover, if
	\(K\Subset(0,\infty)\) and \(\operatorname{supp}\psi\subset K\), then
	\[
	|F^\psi(t)|
	\le
	\|\psi\|_{L^\infty(K)}
	\sup_{n\ge1}\sup_{s\in[0,T]}\int_K \mathrm d\omega\, f_n(s,\omega).
	\]
	Indeed, this estimate is first established for the approximating sequence
	\(\varphi_{k_m}\), and then one passes to the limit as \(m\to\infty\).
	
	We fix \(t\in[0,T]\) and \(K\Subset(0,\infty)\). The map \(F^\psi\), restricted to
	functions \(\psi\in C_c((0,\infty))\) with \(\operatorname{supp}\psi\subset K\), is
	bounded with respect to the \(L^\infty(K)\)-norm. Hence it extends uniquely to a
	bounded linear functional on \(C(K)\), which we still denote by \(F^t\). Since the
	original functionals are positive on nonnegative test functions and \(f_n\ge0\), the
	extension is positive. By the Riesz representation theorem, there exists a finite
	nonnegative Radon measure \(f_K(t)\) on \(K\) such that
	\[
	F^t(\psi)=\int_K \psi(\omega)\,f_K(t,\mathrm d\omega)
	\qquad\text{for all }\psi\in C(K).
	\]
	
	These measures are compatible as \(K\) varies. Indeed, if \(K_1\subset K_2\), then the
	representations on \(K_1\) and \(K_2\) agree on continuous functions supported in the
	interior of \(K_1\), and hence, by uniqueness of the Riesz representation, the
	restriction of \(f_{K_2}(t)\) to \(K_1\) agrees with \(f_{K_1}(t)\). Therefore the
	family \(\{f_K(t)\}_K\) defines a unique nonnegative Radon measure
	\[
	f(t)\in \mathcal M_{\mathrm{loc}}((0,\infty))
	\]
	such that
	\[
	F^\psi(t)=\int_0^\infty f(t,\mathrm d\omega)\,\psi(\omega)
	\qquad\text{for all }\psi\in C_c((0,\infty)).
	\]
	
	We next verify that the measures \(f(t)\) have uniformly bounded mass on compact
	intervals. Let \(K\Subset(0,\infty)\), and choose \(\eta\in C_c((0,\infty))\) such that
	\[
	0\le \eta\le 1,
	\qquad
	\eta\equiv1 \quad\text{on }K.
	\]
	Then for every \(t\in[0,T]\), we can bound
	\[
	f(t)(K)
	\le
	\int_0^\infty f(t,\mathrm d\omega)\,\eta(\omega)
	=
	F^\eta(t).
	\]
	Using the local boundedness estimate above, we obtain
	\[
	F^\eta(t)
	\le
	\|\eta\|_{L^\infty}
	\sup_{n\ge1}\sup_{s\in[0,T]}
	\int_{\operatorname{supp}\eta}\mathrm d\omega\, f_n(s,\omega)
	\le
	\sup_{n\ge1}\sup_{s\in[0,T]}
	\int_{\operatorname{supp}\eta}\mathrm d\omega\, f_n(s,\omega).
	\]
	Therefore, we get
	\[
	\sup_{t\in[0,T]} f(t)(K)<\infty.
	\]
	
	We now prove that the convergence is uniform in time. Let
	\(\psi\in C_c((0,\infty))\), and choose \(K\Subset(0,\infty)\) such that
	\(\operatorname{supp}\psi\subset K\). Given \(\varepsilon>0\), we choose
	\(\varphi_{k_m}\in C_c^2((0,\infty))\), supported in \(K\), such that
	\[
	\|\varphi_{k_m}-\psi\|_{L^\infty(K)}
	\left(
	\sup_{n\ge1}\sup_{t\in[0,T]}\int_K \mathrm d\omega\, f_n(t,\omega)
	+
	\sup_{t\in[0,T]} f(t)(K)
	\right)
	<\varepsilon.
	\]
	Then, for every \(t\in[0,T]\), we can estimate
	\[
	\begin{aligned}
		\left|
		\int_0^\infty \mathrm d\omega\,\psi(\omega)f_n(t,\omega)
		-
		\int_0^\infty f(t,\mathrm d\omega)\,\psi(\omega)
		\right|
		&\le
		\left|
		\int_0^\infty \mathrm d\omega\,(\psi-\varphi_{k_m})(\omega)f_n(t,\omega)
		\right|
		\\
		&\quad+
		\left|
		\int_0^\infty \mathrm d\omega\,\varphi_{k_m}(\omega)f_n(t,\omega)
		-
		\int_0^\infty f(t,\mathrm d\omega)\,\varphi_{k_m}(\omega)
		\right|
		\\
		&\quad+
		\left|
		\int_0^\infty f(t,\mathrm d\omega)\,(\varphi_{k_m}-\psi)(\omega)
		\right|.
	\end{aligned}
	\]
	The first and third terms are bounded by \(\varepsilon\), uniformly in \(n\) and \(t\),
	while the middle term tends to \(0\) uniformly in \(t\), since
	\[
	F_n^{\varphi_{k_m}}\to F^{\varphi_{k_m}}
	\qquad\text{in }C([0,T])
	\qquad\text{as } n\to\infty.
	\]
	Hence, we obtain the convergence
	\[
	\sup_{t\in[0,T]}
	\left|
	\int_0^\infty \mathrm d\omega\,\psi(\omega)f_n(t,\omega)
	-
	\int_0^\infty f(t,\mathrm d\omega)\,\psi(\omega)
	\right|
	\to0
	\qquad\text{as } n\to\infty.
	\]
	Therefore,
	\[
	f_n\to f
	\qquad\text{in }C\bigl([0,T];w\text{-}\mathcal M_{\mathrm{loc}}((0,\infty))\bigr)
	\qquad\text{as } n\to\infty.
	\]
	
	Finally, for every \(\psi\in C_c((0,\infty))\), the map
	\[
	t\longmapsto \int_0^\infty f(t,\mathrm d\omega)\,\psi(\omega)
	\]
	is continuous on \([0,T]\), because it is the uniform limit of the continuous maps
	\[
	t\longmapsto \int_0^\infty \mathrm d\omega\,\psi(\omega)f_n(t,\omega).
	\]
	Thus, we have
	\[
	f\in C\bigl([0,T];w\text{-}\mathcal M_{\mathrm{loc}}((0,\infty))\bigr).
	\]
	This completes the proof.
\end{proof}

\subsection{Passing to the limit - Case A: $U(\omega)=\omega$}\label{Sec:CaseA}

\begin{proposition}\label{Prop:Existence:U=x:Hphi}
	Assume that \(U(\omega)=\omega\), and let \(f_0\in L^1([0,\infty))\) be nonnegative and compactly supported in \((0,\infty)\). Assume moreover that there exists a convex function
	\[
	e:[0,\infty)\to[0,\infty)
	\]
	such that
	\[
	e\in C^\infty[0,\infty),
	\qquad
	e(0)=0,
	\qquad
	e'(0)=0,
	\qquad
	\lim_{r\to\infty}\frac{e(r)}{r}=+\infty,
	\qquad
	\int_{[0,\infty)} \mathrm d\omega\, \frac{e(\omega f_0(\omega))}{\omega}<\infty.
	\]

	For each \(N\in\mathbb N\), let \(f_N\) be the global solution of the cutoff problem \eqref{Cutoff2-new}, and define
	\[
	h_N(t,\omega):=f_N(t,\omega)\chi_{\{\omega\le N\}}.
	\]
	
	Then there exist a subsequence, still denoted by \(\{f_N\}_{N\ge N_0}\), and a nonnegative function
	\[
	f^{\mathrm{reg}}\in C\bigl([0,\infty);w\text{-}L^1_{\mathrm loc}((0,\infty))\bigr)
	\]
	such that the following hold.
	
	For every \(T>0\), every \(0<a<b<\infty\), and every \(\psi\in L^\infty((a,b))\),
	\[
	\int_{(a,b)} \mathrm d\omega\, \psi(\omega)\,h_N(t,\omega)
	\longrightarrow
	\int_{(a,b)} \mathrm d\omega\, \psi(\omega)\,f^{\mathrm{reg}}(t,\omega)
	\]
	uniformly in \(t\in[0,T]\) as \(N\to\infty\).
	
	Moreover, for every fixed \(t\ge0\), the family
	\[
	\bigl\{f_N(t,\omega)\,\mathrm d\omega\bigr\}_{N\ge N_0}
	\]
	is relatively compact for the narrow topology of \(\mathcal M([0,\infty))\). Every narrow cluster
	point \(\mu_t\) of this family is of the form
	\[
	\mu_t=\alpha(t)\delta_{\omega=0}+f^{\mathrm{reg}}(t,\omega)\,\mathrm d\omega
	\]
	for some \(\alpha(t)\ge0\), possibly depending on the chosen cluster point.
	
	If we define
	\[
	g(t,\omega):=\omega f^{\mathrm{reg}}(t,\omega),
	\qquad
	g_0(\omega):=\omega f_0(\omega),
	\]
	then, for every \(T>0\), every \(0<b<\infty\), and every \(\psi\in L^\infty((0,b))\),
	\[
	\int_{(0,b)} \mathrm d\omega\, \psi(\omega)\,g_N(t,\omega)
	\longrightarrow
	\int_{(0,b)} \mathrm d\omega\, \psi(\omega)\,g(t,\omega)
	\]
	uniformly in \(t\in[0,T]\) as \(N\to\infty\).
	
	Furthermore, for every \(\varphi\in C_c^2((0,\infty))\) and every \(t\ge0\),
	\begin{equation}\label{Prop:Existence:U=x:Hphi:weak}
		\int_{(0,\infty)} \mathrm d\omega\, g(t,\omega)\varphi(\omega)
		=
		\int_{(0,\infty)} \mathrm d\omega\, g_0(\omega)\varphi(\omega)
		+
		2\int_0^t \mathrm ds
		\int\!\!\!\int_{\{x>y>0\}} \mathrm dx\,\mathrm dy\,
		g(s,x)g(s,y)\,H_\varphi(x,y),
	\end{equation}
	where
	\[
	H_\varphi(x,y)
	:=
	(x+y)\varphi(x+y)
	+
	(x-y)\varphi(x-y)
	-
	2x\,\varphi(x)
	-
	2y\,\varphi(y).
	\]
	In particular,
	\[
	f(t,\omega)=f^{\mathrm{reg}}(t,\omega)
	\]
	is a weak solution of the non-cutoff equation in the  weak form \eqref{Prop:Existence:U=x:Hphi:weak}.
	
	Finally,
	\[
	\int_{[0,\infty)} \mathrm d\omega\,
	\frac{e(\omega f^{\mathrm{reg}}(t,\omega))}{\omega}
	\le
	\int_{[0,\infty)} \mathrm d\omega\,
	\frac{e(\omega f_0(\omega))}{\omega}
	\qquad\text{for all }t\ge0, 
	\]
and \begin{equation}\label{eq:dissipation-assumption:2}
	2	\int_0^\infty \mathrm{d}s
		\int_0^\infty \mathrm{d}\omega
		\int_0^\omega \mathrm{d}\omega_1\,
		U(\omega_1)U(\omega-\omega_1)\,f_N(s,\omega)\,f_N(s,\omega_1)\,
		\le
		\mathcal{E}_e[f_0].
	\end{equation}
\end{proposition}

\begin{proof}
	We divide the proof into eight steps.
	
	\medskip
	\noindent
	\textit{Step 1: Bounds for the approximating sequence.}
	
	Since \(f_0\) is compactly supported in \((0,\infty)\), there exists \(N_0\in\mathbb N\) such that
	\[
	f_{N,0}=f_0
	\qquad\text{for all }N\ge N_0.
	\]
	
	We define
	\[
	g_N(t,\omega):=\omega h_N(t,\omega)=\omega f_N(t,\omega)\chi_{\{\omega\le N\}},
	\]
	and fix \(N\ge N_0\). Since \(U(\omega)=\omega\), Lemma~\ref{Lemma:EnergyCutoff} yields
	\begin{equation}\label{Prop:Existence:U=x:Hphi:secondmoment}
		\int_{[0,\infty)} \mathrm d\omega\, \omega^2 f_N(t,\omega)
		=
		\int_{[0,\infty)} \mathrm d\omega\, \omega^2 f_0(\omega)
		=:E
		\qquad\text{for all }t\ge0.
	\end{equation}
	
	Next we apply Proposition~\ref{Prop:cutoff:global} and obtain
	\[
	\int_{[0,\infty)} \mathrm d\omega\, h_N(t,\omega)
	=
	\int_{[0,N]} \mathrm d\omega\, f_N(t,\omega)
	\le
	\int_{[0,\infty)} \mathrm d\omega\, f_0(\omega)
	=:M_0
	\qquad\text{for all }t\ge0.
	\]
	
	We now estimate the first moment of \(h_N\). Since
	\[
	\omega\le 1+\omega^2
	\qquad\text{for all }\omega\ge0,
	\]
	we have
	\[
	\int_{[0,\infty)} \mathrm d\omega\, \omega h_N(t,\omega)
	\le
	\int_{[0,\infty)} \mathrm d\omega\, h_N(t,\omega)
	+
	\int_{[0,\infty)} \mathrm d\omega\, \omega^2 h_N(t,\omega).
	\]
	Using the previous bound and \eqref{Prop:Existence:U=x:Hphi:secondmoment}, we conclude that
	\begin{equation}\label{Prop:Existence:U=x:Hphi:firstmoment}
		\int_{[0,\infty)} \mathrm d\omega\, \omega h_N(t,\omega)
		\le
		M_0+E
		=:M_1
		\qquad\text{for all }t\ge0.
	\end{equation}
	Equivalently,
	\begin{equation}\label{Prop:Existence:U=x:Hphi:g-mass}
		\int_{[0,\infty)} \mathrm d\omega\, g_N(t,\omega)
		\le M_1
		\qquad\text{for all }t\ge0.
	\end{equation}
	
	Finally, since \(g_N(\omega)=\omega h_N(\omega)\), we have
	\begin{equation}\label{Prop:Existence:U=x:Hphi:g-firstmoment}
		\int_{[0,\infty)} \mathrm d\omega\, \omega g_N(t,\omega)
		=
		\int_{[0,\infty)} \mathrm d\omega\, \omega^2 h_N(t,\omega)
		\le E
		\qquad\text{for all }t\ge0.
	\end{equation}
	
	\medskip
	\noindent
	\textit{Step 2: Local \(L^1\)-bound and local entropy bounds for \(h_N\) and \(g_N\).}
	
	We fix \(0<a<b<\infty\). Since \(h_N\ge0\), we conclude that
	\begin{equation}\label{Prop:Existence:U=x:Hphi:localL1}
		\sup_{N\ge N_0}\sup_{t\ge0}
		\int_{(a,b)} \mathrm d\omega\, h_N(t,\omega)
		\le M_0.
	\end{equation}
	
	We now use the entropy estimate. By Lemma~\ref{Lemma:EntropyTestCutoff}, we estimate
	\[
	\int_{[0,\infty)} \mathrm d\omega\,
	\frac{e(\omega h_N(t,\omega))}{\omega}
	=
	\int_{[0,N]} \mathrm d\omega\,
	\frac{e(\omega f_N(t,\omega))}{\omega}
	\le
	\int_{[0,\infty)} \mathrm d\omega\,
	\frac{e(\omega f_0(\omega))}{\omega}
	=:C_e
	\qquad\text{for all }t\ge0.
	\]
	
	We define
	\[
	\Phi_{a,b}(r):=\frac1b\,e(ar),\qquad r\ge0.
	\]
	Since \(e\) is convex and nonnegative, \(\Phi_{a,b}\) is convex and nonnegative.  Moreover,
	\[
	\lim_{r\to\infty}\frac{\Phi_{a,b}(r)}{r}
	=
	\frac{a}{b}\lim_{r\to\infty}\frac{e(ar)}{ar}
	=
	+\infty,
	\]
	so \(\Phi_{a,b}\) is superlinear at infinity.
	
Since \(e\) is convex and \(e'(0)=0\), we have \(e'\ge0\) on \([0,\infty)\);
therefore \(e\) is nondecreasing. Hence, for every \(\omega\in(a,b)\) and every \(r\ge0\), we have
	\[
	e(ar)\le e(\omega r),
	\qquad
	\frac1b\le \frac1\omega.
	\]
	
	Therefore, we can bound
	\[
	\Phi_{a,b}(h_N(t,\omega))
	=
	\frac1b\,e(a h_N(t,\omega))
	\le
	\frac1\omega\,e(\omega h_N(t,\omega)),
	\qquad \omega\in(a,b).
	\]
	Integrating over \((a,b)\), we obtain
	\[
	\int_{(a,b)} \mathrm d\omega\, \Phi_{a,b}(h_N(t,\omega))
	\le
	\int_{[0,\infty)} \mathrm d\omega\, \frac{e(\omega h_N(t,\omega))}{\omega}
	\le C_e.
	\]
	Thus
	\begin{equation}\label{Prop:Existence:U=x:Hphi:dVP}
		\sup_{N\ge N_0}\sup_{t\ge0}
		\int_{(a,b)} \mathrm d\omega\, \Phi_{a,b}(h_N(t,\omega))
		\le C_e.
	\end{equation}
	By the de la Vall\'ee--Poussin criterion \cite{poussin1915integrale} and \cite[Theorem 2, Section 1.2]{RaoRen1991}, the family
	\[
	\bigl\{h_N(t,\cdot):\, N\ge N_0,\ t\ge0\bigr\}
	\]
	is uniformly integrable on \((a,b)\).
	
	We next establish the corresponding bound for \(g_N\) on intervals of the form \((0,b)\). We define
	\[
	\Psi_b(r):=\frac1b\,e(r),\qquad r\ge0.
	\]
	Since \(e\) is convex, nonnegative, and superlinear, the same properties hold for \(\Psi_b\). Moreover, for every \(\omega\in(0,b)\) and every \(r\ge0\),
	\[
	\Psi_b(r)=\frac1b\,e(r)\le \frac1\omega\,e(r).
	\]
	Applying this with \(r=g_N(t,\omega)=\omega h_N(t,\omega)\), we obtain
	\[
	\Psi_b(g_N(t,\omega))
	\le
	\frac1\omega\,e(\omega h_N(t,\omega)),
	\qquad \omega\in(0,b).
	\]
	Hence we obtain
	\[
	\int_{(0,b)} \mathrm d\omega\, \Psi_b(g_N(t,\omega))
	\le
	\int_{[0,\infty)} \mathrm d\omega\, \frac{e(\omega h_N(t,\omega))}{\omega}
	\le C_e.
	\]
	Together with \eqref{Prop:Existence:U=x:Hphi:g-mass}, this implies that the family
	\begin{equation}\label{Prop:Existence:U=x:Hphi:g-dVP}
		\bigl\{g_N(t,\cdot):\, N\ge N_0,\ t\ge0\bigr\}
		\qquad\text{is uniformly integrable on }(0,b).
	\end{equation}
	
	\medskip
	\noindent
	\textit{Step 3: Weak formulation for \(g_N\).}
	
	Let \(\varphi\in C_c^2((0,\infty))\). Since \(\varphi\) has compact support in \((0,\infty)\), there exists \(b_\varphi>0\) such that
	\[
	\operatorname{supp}\varphi\subset (0,b_\varphi).
	\]
	We fix \(N\ge b_\varphi\) and define
	\[
	F_N^\varphi(t):=
	\int_{(0,\infty)} \mathrm d\omega\, \varphi(\omega)g_N(t,\omega).
	\]
	Since \(g_N(t,\omega)=\omega f_N(t,\omega)\) on \((0,N)\), we may write
	\[
	F_N^\varphi(t)
	=
	\int_{(0,\infty)} \mathrm d\omega\, \phi(\omega)f_N(t,\omega),
	\qquad
	\phi(\omega):=\omega\varphi(\omega).
	\]
	Differentiating in time and using \eqref{Cutoff2-new}, we obtain
	\[
	\frac{\mathrm d}{\mathrm dt}F_N^\varphi(t)
	=
	\int_{(0,\infty)} \mathrm d\omega\, \phi(\omega)\mathcal Q_N[f_N](t,\omega).
	\]
	
	We now compute the right-hand side using the full cutoff operator. We write
	\[
	\int_{(0,\infty)} \mathrm d\omega\,  \phi\,\mathcal Q_N[f_N]
	=
	A_1-A_2-A_3-2B_1+2B_2+2B_3,
	\]
	where
	\[
	\begin{aligned}
		A_1
		&:=
		\int_0^\infty \mathrm d\omega\, \phi(\omega)
		\int_0^\omega \mathrm d\omega_1\,
		\omega_1(\omega-\omega_1)
		f_N(t,\omega_1)f_N(t,\omega-\omega_1)
		\chi_{\{\omega_1,\omega-\omega_1\le N\}},
		\\[1mm]
		A_2
		&:=
		\int_0^\infty \mathrm d\omega\, \phi(\omega)
		\int_0^\omega \mathrm d\omega_1\,
		\omega_1(\omega-\omega_1)
		f_N(t,\omega)f_N(t,\omega_1)
		\chi_{\{\omega,\omega_1\le N\}},
		\\[1mm]
		A_3
		&:=
		\int_0^\infty \mathrm d\omega\, \phi(\omega)
		\int_0^\omega \mathrm d\omega_1\,
		\omega_1(\omega-\omega_1)
		f_N(t,\omega)f_N(t,\omega-\omega_1)
		\chi_{\{\omega,\omega-\omega_1\le N\}},
	\end{aligned}
	\]
	and
	\[
	\begin{aligned}
		B_1
		&:=
		\int_0^\infty \mathrm d\omega\, \phi(\omega)
		\int_0^\infty \mathrm d\omega_1\,
		\omega_1(\omega+\omega_1)
		f_N(t,\omega_1)f_N(t,\omega)
		\chi_{\{\omega,\omega_1\le N\}},
		\\[1mm]
		B_2
		&:=
		\int_0^\infty \mathrm d\omega\, \phi(\omega)
		\int_0^\infty \mathrm d\omega_1\,
		\omega_1(\omega+\omega_1)
		f_N(t,\omega_1)f_N(t,\omega+\omega_1)
		\chi_{\{\omega_1,\omega+\omega_1\le N\}},
		\\[1mm]
		B_3
		&:=
		\int_0^\infty \mathrm d\omega\, \phi(\omega)
		\int_0^\infty \mathrm d\omega_1\,
		\omega_1(\omega+\omega_1)
		f_N(t,\omega)f_N(t,\omega+\omega_1)
		\chi_{\{\omega,\omega+\omega_1\le N\}}.
	\end{aligned}
	\]
	
	We now rewrite each term.
	
	First, for \(A_1\), we set
	\(
	x:=\omega-\omega_1, y:=\omega_1.
	\)
	Then \(\omega=x+y\), \(x>0\), \(y>0\), and the Jacobian equals \(1\). Hence
	\[
	A_1
	=
	\int\!\!\!\int_{(0,\infty)^2} \mathrm dx\,\mathrm dy\,
	xy\,f_N(t,x)f_N(t,y)\chi_{\{x,y\le N\}}\phi(x+y).
	\]
	Since \(\phi(x+y)=(x+y)\varphi(x+y)\), we obtain
	\[
	A_1
	=
	\int\!\!\!\int_{(0,\infty)^2} \mathrm dx\,\mathrm dy\,
	g_N(t,x)g_N(t,y)\,\varphi(x+y)(x+y).
	\]
	
	Next, for \(A_2\), we set
	\(
	x:=\omega, y:=\omega_1.
	\)
	Then \(x>y>0\), \(\omega-\omega_1=x-y\), and therefore
	\[
	A_2
	=
	\int\!\!\!\int_{x>y>0} \mathrm dx\,\mathrm dy\,
	y(x-y)f_N(t,x)f_N(t,y)\chi_{\{x,y\le N\}}\phi(x).
	\]
	Since \(\phi(x)=x\varphi(x)\), this becomes
	\[
	A_2
	=
	\int\!\!\!\int_{x>y>0} \mathrm dx\,\mathrm dy\,
	g_N(t,x)g_N(t,y)\,(x-y)\varphi(x).
	\]
	
	Similarly, for \(A_3\), we set
	\(
	x:=\omega, y:=\omega-\omega_1.
	\)
	Then \(x>y>0\), \(\omega_1=x-y\), and
	\[
	A_3
	=
	\int\!\!\!\int_{x>y>0} \mathrm dx\,\mathrm dy\,
	g_N(t,x)g_N(t,y)\,(x-y)\varphi(x).
	\]
	
	We next turn to \(B_1\). We split the region \((0,\infty)^2\) into \(\{\omega>\omega_1\}\) and \(\{\omega_1>\omega\}\).
	
	On \(\{\omega>\omega_1\}\), we set
	\(
	x:=\omega, y:=\omega_1.
	\)
	Then \(x>y>0\), and on this domain, we find
	\[
	\int\!\!\!\int_{x>y>0} \mathrm dx\,\mathrm dy\,
	g_N(t,x)g_N(t,y)\,(x+y)\varphi(x).
	\]
	
	On \(\{\omega_1>\omega\}\), we set
	\(
	x:=\omega_1, y:=\omega.
	\)
	Then \(x>y>0\), and on this region, we find
	\[
	\int\!\!\!\int_{x>y>0} \mathrm dx\,\mathrm dy\,
	g_N(t,x)g_N(t,y)\,(x+y)\varphi(y).
	\]
	Therefore, we find
	\[
	B_1
	=
	\int\!\!\!\int_{x>y>0} \mathrm dx\,\mathrm dy\,
	g_N(t,x)g_N(t,y)\,(x+y)\varphi(x)
	+
	\int\!\!\!\int_{x>y>0} \mathrm dx\,\mathrm dy\,
	g_N(t,x)g_N(t,y)\,(x+y)\varphi(y).
	\]
	
	For \(B_2\), we set
	\(
	x:=\omega+\omega_1, y:=\omega_1.
	\)
	Then \(x>y>0\), \(\omega=x-y\), and the Jacobian equals \(1\). Hence, we obtain
	\[
	B_2
	=
	\int\!\!\!\int_{x>y>0} \mathrm dx\,\mathrm dy\,
	yx\,f_N(t,y)f_N(t,x)\chi_{\{x,y\le N\}}\phi(x-y)
	=
	\int\!\!\!\int_{x>y>0} \mathrm dx\,\mathrm dy\,
	g_N(t,x)g_N(t,y)\,(x-y)\varphi(x-y).
	\]
	
	Finally, for \(B_3\), we set
	\[
	x:=\omega+\omega_1,\qquad y:=\omega.
	\]
	Then \(x>y>0\), \(\omega_1=x-y\), and we can write
	\[
	B_3
	=
	\int\!\!\!\int_{x>y>0} \mathrm dx\,\mathrm dy\,
	(x-y)x\,f_N(t,y)f_N(t,x)\chi_{\{x,y\le N\}}\phi(y)
	=
	\int\!\!\!\int_{x>y>0} \mathrm dx\,\mathrm dy\,
	g_N(t,x)g_N(t,y)\,(x-y)\varphi(y).
	\]
	
	Substituting all rewritten terms into
	\[
	\frac{\mathrm d}{\mathrm dt}F_N^\varphi(t)=A_1-A_2-A_3-2B_1+2B_2+2B_3,
	\]
	we obtain
	\[
	\begin{aligned}
		\frac{\mathrm d}{\mathrm dt}F_N^\varphi(t)
		&=
		2\int\!\!\!\int_{x>y>0} \mathrm dx\,\mathrm dy\,
		g_N(t,x)g_N(t,y)
		\Bigl[
		(x+y)\varphi(x+y)
		-(x-y)\varphi(x)
		\\
		&\qquad\qquad\qquad\qquad
		-(x+y)\varphi(x)
		-(x+y)\varphi(y)
		+(x-y)\varphi(x-y)
		+(x-y)\varphi(y)
		\Bigr].
	\end{aligned}
	\]
	We now simplify the expression in brackets. The terms multiplying \(\varphi(x)\) combine as
	\[
	-(x-y)-(x+y)=-2x,
	\]
	whereas the terms multiplying \(\varphi(y)\) combine as
	\[
	-(x+y)+(x-y)=-2y.
	\]
	Therefore, we can write
	\[
	\frac{\mathrm d}{\mathrm dt}F_N^\varphi(t)
	=
	2\int\!\!\!\int_{x>y>0} \mathrm dx\,\mathrm dy\,
	g_N(t,x)g_N(t,y)\,H_\varphi(x,y),
	\]
	with
	\[
	H_\varphi(x,y)
	=
	(x+y)\varphi(x+y)
	+
	(x-y)\varphi(x-y)
	-
	2x\,\varphi(x)
	-
	2y\,\varphi(y).
	\]
	
	\medskip
	\noindent
	\textit{Step 4: Uniform boundedness and equicontinuity.}
	
	We let \(\psi\in C_c^2((0,\infty))\) and choose \(0<a<b<\infty\) such that
	\[
	\operatorname{supp}\psi\subset(a,b),
	\]
	and then we define
	\[
	\varphi(\omega):=\frac{\psi(\omega)}{\omega}.
	\]
	Then \(\varphi\in C_c^2((0,\infty))\), \(\operatorname{supp}\varphi\subset(a,b)\), and
	\[
	G_N^\psi(t):=
	\int_{(0,\infty)} \mathrm d\omega\, \psi(\omega)h_N(t,\omega)
	=
	\int_{(0,\infty)} \mathrm d\omega\, \varphi(\omega)g_N(t,\omega).
	\]
	Hence, by Step~3, we have
	\begin{equation}\label{Prop:Existence:U=x:Hphi:dG}
		\frac{\mathrm d}{\mathrm dt}G_N^\psi(t)
		=
		2\int\!\!\!\int_{x>y>0} \mathrm dx\,\mathrm dy\,
		g_N(t,x)g_N(t,y)\,H_\varphi(x,y).
	\end{equation}
	
	We now estimate the right-hand side directly. Since \(\operatorname{supp}\varphi\subset(a,b)\), we examine separately the four terms in
	\[
	H_\varphi(x,y)
	=
	(x+y)\varphi(x+y)
	+
	(x-y)\varphi(x-y)
	-
	2x\varphi(x)
	-
	2y\varphi(y).
	\]
	
	First, if \(\varphi(x+y)\neq0\), then \(x+y\in(a,b)\), hence \(0<x<b\) and \(0<y<b\). Therefore, we can bound
	\[
	\begin{aligned}
		&\int\!\!\!\int_{x>y>0} \mathrm dx\,\mathrm dy\,
		g_N(t,x)g_N(t,y)\,(x+y)|\varphi(x+y)|
		\\
		&\qquad\le
		b\|\varphi\|_{L^\infty}
		\int\!\!\!\int_{(0,b)^2} \mathrm dx\,\mathrm dy\,
		g_N(t,x)g_N(t,y)
		\le
		b\|\varphi\|_{L^\infty}M_1^2.
	\end{aligned}
	\]
	
	Next, if \(\varphi(x-y)\neq0\), then \(x-y\in(a,b)\), so \(x=y+z\) with \(z\in(a,b)\). Hence
	\[
	\begin{aligned}
		&\int\!\!\!\int_{x>y>0} \mathrm dx\,\mathrm dy\,
		g_N(t,x)g_N(t,y)\,(x-y)|\varphi(x-y)|
		\\
		&\qquad=
		\int_0^\infty \mathrm dy\, g_N(t,y)
		\int_a^b \mathrm dz\, g_N(t,y+z)\,z|\varphi(z)|.
	\end{aligned}
	\]
	Since \(y+z\ge a\), we have
	\[
	\int_a^b \mathrm dz\, g_N(t,y+z)\,z|\varphi(z)|
	\le
	b\|\varphi\|_{L^\infty}\int_{y+a}^{y+b} \mathrm dx\, g_N(t,x).
	\]
	Therefore
	\[
	\begin{aligned}
		&\int\!\!\!\int_{x>y>0} \mathrm dx\,\mathrm dy\,
		g_N(t,x)g_N(t,y)\,(x-y)|\varphi(x-y)|
		\\
		&\qquad\le
		b\|\varphi\|_{L^\infty}M_1
		\int_0^\infty \mathrm dy\, g_N(t,y)
		\le
		b\|\varphi\|_{L^\infty}M_1^2.
	\end{aligned}
	\]
	
	We now consider the term \(2x\varphi(x)\). If \(\varphi(x)\neq0\), then \(x\in(a,b)\) and \(0<y<x<b\). Hence, we can bound
	\[
	\begin{aligned}
		&\int\!\!\!\int_{x>y>0} \mathrm dx\,\mathrm dy\,
		g_N(t,x)g_N(t,y)\,2x|\varphi(x)|
		\\
		&\qquad\le
		2b\|\varphi\|_{L^\infty}
		\int_a^b \mathrm dx\, g_N(t,x)
		\int_0^x \mathrm dy\, g_N(t,y)
		\le
		2b\|\varphi\|_{L^\infty}M_1^2.
	\end{aligned}
	\]
	
	Finally, if \(\varphi(y)\neq0\), then \(y\in(a,b)\). Thus, we can estimate
	\[
	\begin{aligned}
		&\int\!\!\!\int_{x>y>0} \mathrm dx\,\mathrm dy\,
		g_N(t,x)g_N(t,y)\,2y|\varphi(y)|
		\\
		&\qquad\le
		2b\|\varphi\|_{L^\infty}
		\int_a^b \mathrm dy\, g_N(t,y)
		\int_y^\infty \mathrm dx\, g_N(t,x)
		\le
		2b\|\varphi\|_{L^\infty}M_1^2.
	\end{aligned}
	\]
	
	Combining the four bounds, we find a constant \(C_\varphi>0\), depending only on \(a\), \(b\), \(\varphi\), \(M_1\), such that
	\[
	\left|
	\int\!\!\!\int_{x>y>0} \mathrm dx\,\mathrm dy\,
	g_N(t,x)g_N(t,y)\,H_\varphi(x,y)
	\right|
	\le C_\varphi
	\qquad\text{for all }N\ge N_0,\ t\ge0.
	\]
	In view of \eqref{Prop:Existence:U=x:Hphi:dG}, this yields
	\[
	\left|\frac{\mathrm d}{\mathrm dt}G_N^\psi(t)\right|
	\le 2C_\varphi
	\qquad\text{for all }N\ge N_0,\ t\ge0.
	\]
	Hence \(\{G_N^\psi\}_{N\ge N_0}\) is uniformly Lipschitz on every bounded interval \([0,T]\).
	
	Also, by \eqref{Prop:Existence:U=x:Hphi:localL1}, we find
	\[
	|G_N^\psi(t)|
	\le
	\|\psi\|_{L^\infty}\int_{(a,b)} \mathrm d\omega\, h_N(t,\omega)
	\le
	\|\psi\|_{L^\infty}M_0.
	\]
	Therefore \(\{G_N^\psi\}_{N\ge N_0}\) is uniformly bounded and uniformly equicontinuous on every interval \([0,T]\).
	
\medskip
\noindent
\textit{Step 5: Weak convergence of \(h_N\) and \(g_N\).}

Fix \(T>0\). By \eqref{Prop:Existence:U=x:Hphi:localL1}, for every compact interval
\((a,b)\Subset(0,\infty)\), we have
\[
\sup_{N\ge N_0}\sup_{t\in[0,T]}
\int_{(a,b)} \mathrm d\omega\, h_N(t,\omega)<\infty.
\]
Moreover, by Step~4, for every \(\psi\in C_c^2((0,\infty))\), the functions
\[
t\longmapsto \int_0^\infty \mathrm d\omega\, \psi(\omega)h_N(t,\omega)
\]
are uniformly bounded and uniformly equicontinuous on \([0,T]\).

Therefore Lemma~\ref{lem:compactness-local-measures} applies. Hence there exist a subsequence,
still denoted by \(\{h_N\}\), and a nonnegative map
\[
f^{\mathrm{reg}}\in C\bigl([0,T];w\text{-}\mathcal M_{\mathrm loc}((0,\infty))\bigr)
\]
such that, for every \(\psi\in C_c((0,\infty))\), we have
\[
\int_0^\infty \mathrm d\omega\, \psi(\omega)h_N(t,\omega)
\longrightarrow
\int_0^\infty f^{\mathrm{reg}}(t,\mathrm d\omega)\,\psi(\omega)
\qquad\text{uniformly in }t\in[0,T]
\]
as \(N\to\infty\).

We now show that \(f^{\mathrm{reg}}(t,\cdot)\in L^1_{\mathrm loc}((0,\infty))\), and that the convergence is
actually weak convergence in \(L^1_{\mathrm loc}\).

To this end, we fix \(0<a<b<\infty\), and set
\[
\mathcal F_{a,b,T}:=\bigl\{h_N(t,\cdot):\,N\ge N_0,\ t\in[0,T]\bigr\}\subset L^1((a,b)).
\]
By \eqref{Prop:Existence:U=x:Hphi:localL1}, the family \(\mathcal F_{a,b,T}\) is bounded in
\(L^1((a,b))\).
Moreover, by \eqref{Prop:Existence:U=x:Hphi:dVP}, the family \(\mathcal F_{a,b,T}\) is
uniformly integrable on \((a,b)\). Therefore, by the Dunford--Pettis theorem,
\(\mathcal F_{a,b,T}\) is relatively weakly compact in \(L^1((a,b))\).

We now fix  \(t\in[0,T]\). Since \(h_N(t,\cdot)\,\mathrm d\omega\) converges to
\(f^{\mathrm{reg}}(t,\mathrm d\omega)\) as a Radon measure on \((a,b)\), every subsequence of
\(\{h_N(t,\cdot)\}\) admits a further subsequence converging weakly in \(L^1((a,b))\)
to some \(\widetilde f_t\in L^1((a,b))\). But weak \(L^1\)-convergence implies convergence
against continuous test functions. Hence the measure limit must be
\(\widetilde f_t\,\mathrm d\omega\). By uniqueness of the Radon measure limit, we must have
\[
f^{\mathrm{reg}}(t,\mathrm d\omega)=\widetilde f_t(\omega)\,\mathrm d\omega
\qquad\text{on }(a,b).
\]
Therefore \(f^{\mathrm{reg}}(t,\cdot)\in L^1((a,b))\), and
\[
h_N(t,\cdot)\rightharpoonup f^{\mathrm{reg}}(t,\cdot)
\qquad\text{weakly in }L^1((a,b))
\qquad\text{as }N\to\infty.
\]
Since \((a,b)\Subset(0,\infty)\) is arbitrary, we obtain
\[
f^{\mathrm{reg}}(t,\cdot)\in L^1_{\mathrm loc}((0,\infty))
\qquad\text{for every }t\in[0,T].
\]

We also recall the corresponding local entropy bound for \(f^{\mathrm{reg}}\). Since
\(\Phi_{a,b}\), defined in Step 2, is convex, nonnegative, and superlinear, the weak lower
semicontinuity of convex integral functionals gives
\[
\int_{(a,b)} \mathrm d\omega\, \Phi_{a,b}(f^{\mathrm{reg}}(t,\omega))
\le
\liminf_{N\to\infty}
\int_{(a,b)} \mathrm d\omega\, \Phi_{a,b}(h_N(t,\omega))
\le C_e
\qquad\text{for every }t\in[0,T].
\]
In particular, the family
\[
\bigl\{f^{\mathrm{reg}}(t,\cdot):\,t\in[0,T]\bigr\}
\]
is uniformly integrable on \((a,b)\). Therefore
\[
f^{\mathrm{reg}}\in C\bigl([0,T];w\text{-}L^1_{\mathrm loc}((0,\infty))\bigr).
\]

We now upgrade the uniform convergence in time of $\{h_N\}$ from continuous test functions to arbitrary
\(L^\infty((a,b))\)-test functions. We claim that for every \(\psi\in L^\infty((a,b))\),
\[
\sup_{t\in[0,T]}
\left|
\int_{(a,b)} \mathrm d\omega\, \psi(\omega)\,h_N(t,\omega)
-
\int_{(a,b)} \mathrm d\omega\, \psi(\omega)\,f^{\mathrm{reg}}(t,\omega)
\right|\to0
\qquad\text{as }N\to\infty.
\]
Assume by contradiction that this fails. Then there exist \(\psi\in L^\infty((a,b))\),
\(\varepsilon>0\), a subsequence \(N_j\to\infty\), and times \(t_j\in[0,T]\) such that
\[
\left|
\int_{(a,b)} \mathrm d\omega\, \psi(\omega)\,h_{N_j}(t_j,\omega)
-
\int_{(a,b)} \mathrm d\omega\, \psi(\omega)\,f^{\mathrm{reg}}(t_j,\omega)
\right|
\ge \varepsilon
\qquad\text{for all }j\ge1.
\]
By the relative weak compactness of \(\mathcal F_{a,b,T}\), after extracting a subsequence, we may
assume that
\[
h_{N_j}(t_j,\cdot)\rightharpoonup g
\qquad\text{weakly in }L^1((a,b)).
\]
Moreover, by the uniform integrability of
\(\{f^{\mathrm{reg}}(t,\cdot):t\in[0,T]\}\) on \((a,b)\), after extracting a further subsequence, we may assume that
\[
f^{\mathrm{reg}}(t_j,\cdot)\rightharpoonup \widetilde g
\qquad\text{weakly in }L^1((a,b)).
\]
On the other hand, the already established uniform convergence in time against continuous
test functions yields, for every \(\varphi\in C([a,b])\),
\[
\int_{(a,b)} \mathrm d\omega\, \varphi(\omega)\,h_{N_j}(t_j,\omega)
-
\int_{(a,b)} \mathrm d\omega\, \varphi(\omega)\,f^{\mathrm{reg}}(t_j,\omega)
\to0
\qquad\text{as }j\to\infty.
\]
Here we have used the fact that every \(\varphi\in C([a,b])\) can be extended, after multiplication
by a cutoff, to a function in \(C_c((0,\infty))\). Passing to the limit gives
\[
\int_{(a,b)} \mathrm d\omega\, \varphi(\omega)\,g(\omega)
=
\int_{(a,b)} \mathrm d\omega\, \varphi(\omega)\,\widetilde g(\omega)
\qquad\text{for all }\varphi\in C([a,b]).
\]
Hence \(g=\widetilde g\) a.e. on \((a,b)\). Therefore,
\[
\int_{(a,b)} \mathrm d\omega\, \psi(\omega)\,h_{N_j}(t_j,\omega)
\to
\int_{(a,b)} \mathrm d\omega\, \psi(\omega)\,g(\omega)
\qquad\text{as }j\to\infty,
\]
and
\[
\int_{(a,b)} \mathrm d\omega\, \psi(\omega)\,f^{\mathrm{reg}}(t_j,\omega)
\to
\int_{(a,b)} \mathrm d\omega\, \psi(\omega)\,\widetilde g(\omega)
\qquad\text{as }j\to\infty.
\]
Since \(g=\widetilde g\) a.e. on \((a,b)\), the two limits coincide, which contradicts the
choice of \(\varepsilon\). This proves the claim.

We next identify the convergence of \(g_N\). We define
\[
g(t,\omega):=\omega f^{\mathrm{reg}}(t,\omega).
\]
We fix \(0<b<\infty\) and \(\psi\in L^\infty((0,b))\), and let \(\varepsilon>0\). We choose
\(0<\delta<b\) sufficiently small such that
\[
2\delta M_0\|\psi\|_{L^\infty}<\varepsilon.
\]
Then
\[
\left|
\int_{(0,b)} \mathrm d\omega\, \psi(\omega)\bigl(g_N(t,\omega)-g(t,\omega)\bigr)
\right|
\le I_{N,\delta}(t)+J_{N,\delta}(t),
\]
where
\[
I_{N,\delta}(t)
:=
\left|
\int_{(0,\delta)} \mathrm d\omega\, \psi(\omega)\bigl(g_N(t,\omega)-g(t,\omega)\bigr)
\right|,
\]
and
\[
J_{N,\delta}(t)
:=
\left|
\int_{(\delta,b)} \mathrm d\omega\, \psi(\omega)\bigl(g_N(t,\omega)-g(t,\omega)\bigr)
\right|.
\]

For the first term, using \(g_N=\omega h_N\) and \eqref{Prop:Existence:U=x:Hphi:localL1}, we obtain
\[
\int_{(0,\delta)} \mathrm d\omega\, g_N(t,\omega)
=
\int_{(0,\delta)} \mathrm d\omega\, \omega h_N(t,\omega)
\le
\delta\int_{(0,\delta)} \mathrm d\omega\, h_N(t,\omega)
\le
\delta M_0.
\]
Moreover, for every \(0<\eta<\delta\), the already proved convergence of \(h_N\) on
\((\eta,\delta)\) gives
\[
\int_{(\eta,\delta)} \mathrm d\omega\, \omega f^{\mathrm{reg}}(t,\omega)
=
\lim_{N\to\infty}
\int_{(\eta,\delta)} \mathrm d\omega\, \omega h_N(t,\omega)
\le
\delta M_0.
\]
Letting \(\eta\to0\), we obtain
\[
\int_{(0,\delta)} \mathrm d\omega\, g(t,\omega)
\le
\delta M_0.
\]
Therefore,
\[
I_{N,\delta}(t)\le 2\delta M_0\|\psi\|_{L^\infty}<\varepsilon
\qquad\text{for all }N,\ t.
\]

For the second term, since \(\omega\psi(\omega)\in L^\infty((\delta,b))\), the already proved
uniform convergence of \(h_N\) on \((\delta,b)\) gives
\[
J_{N,\delta}(t)\to0
\qquad\text{uniformly in }t\in[0,T]
\]
as \(N\to\infty\). Hence
\[
\sup_{t\in[0,T]}
\left|
\int_{(0,b)} \mathrm d\omega\, \psi(\omega)\,g_N(t,\omega)
-
\int_{(0,b)} \mathrm d\omega\, \psi(\omega)\,g(t,\omega)
\right|\to0
\qquad\text{as }N\to\infty.
\]
That is,
\[
g_N\to g
\qquad\text{in }C\bigl([0,T];w\text{-}L^1_{\mathrm loc}((0,\infty))\bigr)
\qquad\text{as }N\to\infty.
\]

We now identify the possible full measure limits of \(f_N\). We fix \(t\ge0\). By
\eqref{Prop:Existence:U=x:Hphi:secondmoment} and the mass bound, we have
\[
\sup_{N\ge N_0}\int_{[0,\infty)} \mathrm d\omega\, h_N(t,\omega)\le M_0,
\qquad
\sup_{N\ge N_0}\int_{[0,\infty)} \mathrm d\omega\, \omega^2 f_N(t,\omega)\le E.
\]
We  write
\[
\int_{[0,\infty)} \mathrm d\omega\, f_N(t,\omega)
=
\int_{[0,N]} \mathrm d\omega\, f_N(t,\omega)
+
\int_{(N,\infty)} \mathrm d\omega\, f_N(t,\omega).
\]
The first term is bounded by
\(M_0\) .
For the second term, using the second-moment estimate
\eqref{Prop:Existence:U=x:Hphi:secondmoment}, we obtain
\[
\int_{(N,\infty)} \mathrm d\omega\, f_N(t,\omega)
\le
\frac1{N^2}
\int_{(N,\infty)} \mathrm d\omega\, \omega^2 f_N(t,\omega)
\le
\frac{E}{N^2}
\le
\frac{E}{N_0^2},
\qquad N\ge N_0.
\]
Therefore, we can bound
\[
\sup_{N\ge N_0}
\int_{[0,\infty)} \mathrm d\omega\, f_N(t,\omega)
\le
M_0+\frac{E}{N_0^2}.
\]

Moreover, the same second-moment estimate gives tightness. For every \(R>0\), we can estimate
\[
\int_R^\infty \mathrm d\omega\, f_N(t,\omega)
\le
\frac1{R^2}
\int_R^\infty \mathrm d\omega\, \omega^2 f_N(t,\omega)
\le
\frac{E}{R^2}.
\]
Hence the family
\[
\bigl\{f_N(t,\omega)\,\mathrm d\omega\bigr\}_{N\ge N_0}
\]
is bounded in \(\mathcal M_+([0,\infty))\) and tight. By the Prokhorov theorem, every subsequence
admits a further subsequence, not relabeled, which converges narrowly in
\(\mathcal M([0,\infty))\) to a nonnegative measure \(\mu_t\), namely
\[
\lim_{N\to\infty}
\int_{[0,\infty)} \mathrm d\omega\, f_N(t,\omega)\phi(\omega)
=
\int_{[0,\infty)} \mu_t(\mathrm d\omega)\,\phi(\omega)
\qquad
\text{for all }\phi\in C_b([0,\infty)).
\]

For every \(0<a<b<\infty\), the already established weak convergence in \(L^1((a,b))\)
implies that
\[
\int_{(a,b)} \psi(\omega)\,\mu_t(\mathrm d\omega)
=
\int_{(a,b)} \mathrm d\omega\, \psi(\omega)\,f^{\mathrm{reg}}(t,\omega)
\qquad
\text{for all }\psi\in C_c((a,b)).
\]
Therefore, we obtain
\[
\mu_t(A)
=
\int_A \mathrm d\omega\, f^{\mathrm{reg}}(t,\omega)
\qquad
\text{for every Borel set }A\subset(a,b).
\]
Since \((a,b)\Subset(0,\infty)\) is arbitrary, the measure
\[
\mu_t-f^{\mathrm{reg}}(t,\omega)\,\mathrm d\omega
\]
is supported on \(\{0\}\). Therefore there exists a number \(\alpha(t)\ge0\), possibly
depending on the chosen  limit point in the narrow topology of \(\mathcal M([0,\infty))\), such that
\[
\mu_t
=
\alpha(t)\delta_{\omega=0}
+
f^{\mathrm{reg}}(t,\omega)\,\mathrm d\omega.
\]
Thus, for every fixed \(t\ge0\), every limit point in the narrow topology of
\[
\bigl\{f_N(t,\omega)\,\mathrm d\omega\bigr\}_{N\ge N_0}
\]
in \(\mathcal M([0,\infty))\) has the form
\[
\alpha(t)\delta_{\omega=0}
+
f^{\mathrm{reg}}(t,\omega)\,\mathrm d\omega.
\]

\medskip
\noindent
\textit{Step 6: Convergence of the measures \(\nu_N=g_N\,\mathrm d\omega\).}

We define
\[
g(t,\omega):=\omega f^{\mathrm{reg}}(t,\omega),
\]
and, for each \(t\ge0\),
\[
\nu_N(t):=g_N(t,\omega)\,\mathrm d\omega,
\qquad
\nu(t):=g(t,\omega)\,\mathrm d\omega.
\]

We claim that, for every fixed \(t\ge0\),
\[
\int_{[0,\infty)} \psi(\omega)\,\nu_N(t,\mathrm d\omega)
\longrightarrow
\int_{[0,\infty)} \psi(\omega)\,\nu(t,\mathrm d\omega)
\qquad\text{for every }\psi\in C_b([0,\infty)).
\]

Indeed, let \(\psi\in C_b([0,\infty))\). We fix \(\varepsilon>0\). Since
\[
\int_{[0,\infty)} \mathrm d\omega\, \omega g_N(t,\omega)\le E
\]
for all \(N\), we have
\[
\int_R^\infty \mathrm d\omega\, g_N(t,\omega)
\le
\frac1R\int_R^\infty \mathrm d\omega\, \omega g_N(t,\omega)
\le
\frac{E}{R}.
\]
Moreover, by the weak lower semicontinuity applied on finite intervals and then letting the
upper endpoint tend to infinity, the same bound holds for \(g(t,\cdot)\), namely
\[
\int_R^\infty \mathrm d\omega\, g(t,\omega)\le \frac{E}{R}.
\]
Hence we may choose \(R>0\) so large that
\[
\|\psi\|_{L^\infty}\frac{2E}{R}<\varepsilon.
\]

Let \(\Psi_R\in C_c([0,\infty))\) be such that
\[
0\le \Psi_R\le 1,
\qquad
\Psi_R\equiv1\quad\text{on }[0,R],
\qquad
\Psi_R\equiv0\quad\text{on }[R+1,\infty).
\]
Then
\[
\begin{aligned}
	\left|
	\int_{[0,\infty)} \psi(\omega)\,\nu_N(t,\mathrm d\omega)
	-
	\int_{[0,\infty)} \Psi_R(\omega)\psi(\omega)\,\nu_N(t,\mathrm d\omega)
	\right|
	&\le
	\|\psi\|_{L^\infty}\int_R^\infty \mathrm d\omega\, g_N(t,\omega)
	\\
	&\le
	\|\psi\|_{L^\infty}\frac{E}{R},
\end{aligned}
\]
and similarly
\[
\left|
\int_{[0,\infty)} \psi(\omega)\,\nu(t,\mathrm d\omega)
-
\int_{[0,\infty)} \Psi_R(\omega)\psi(\omega)\,\nu(t,\mathrm d\omega)
\right|
\le
\|\psi\|_{L^\infty}\frac{E}{R}.
\]

On the other hand, since \(\Psi_R\psi\in L^\infty((0,R+1))\),  
\[
\int_{[0,\infty)} \Psi_R(\omega)\psi(\omega)\,\nu_N(t,\mathrm d\omega)
=
\int_{(0,R+1)} \mathrm d\omega\, \Psi_R(\omega)\psi(\omega)g_N(t,\omega)
\]
and
\[
\int_{(0,R+1)} \mathrm d\omega\, \Psi_R(\omega)\psi(\omega)g_N(t,\omega)
\longrightarrow
\int_{(0,R+1)} \mathrm d\omega\, \Psi_R(\omega)\psi(\omega)g(t,\omega)
=
\int_{[0,\infty)} \Psi_R(\omega)\psi(\omega)\,\nu(t,\mathrm d\omega).
\]
Combining the three estimates and letting \(N\to\infty\), we conclude that
\[
\int_{[0,\infty)} \psi(\omega)\,\nu_N(t,\mathrm d\omega)
\longrightarrow
\int_{[0,\infty)} \psi(\omega)\,\nu(t,\mathrm d\omega) \qquad\text{as }N\to\infty.
\]
This proves the claim.

Thus, for every fixed \(t\ge0\), we have the convergence
\[
\nu_N(t)\rightharpoonup \nu(t)
\qquad\text{narrowly in }\mathcal M([0,\infty))
\qquad\text{as }N\to\infty.
\]

Consequently, for every fixed \(t\ge0\), we obtain the convergence
\[
\nu_N(t)\otimes\nu_N(t)
\rightharpoonup
\nu(t)\otimes\nu(t)
\qquad\text{narrowly in }\mathcal M([0,\infty)^2) \qquad\text{as }N\to\infty.
\]
Equivalently, for every \(\Phi\in C_b([0,\infty)^2)\), we have
\[
\int_{[0,\infty)^2} \Phi(x,y)\,\nu_N(t,\mathrm dx)\nu_N(t,\mathrm dy)
\longrightarrow
\int_{[0,\infty)^2} \Phi(x,y)\,\nu(t,\mathrm dx)\nu(t,\mathrm dy)
\qquad\text{as }N\to\infty.
\]

\medskip
\noindent
\textit{Step 7: Passage to the limit in the  weak formulation.}

We fix \(\varphi\in C_c^2((0,\infty))\), and choose \(0<a<b<\infty\) such that
\[
\operatorname{supp}\varphi\subset(a,b).
\]
For every \(N\) and every \(t\in[0,T]\), Step~3 yields
\begin{equation}\label{Prop:Existence:U=x:Hphi:weak-N}
	\int_{(0,\infty)} \mathrm d\omega\, g_N(t,\omega)\varphi(\omega)
	=
	\int_{(0,\infty)} \mathrm d\omega\, g_N(0,\omega)\varphi(\omega)
	+
	2\int_0^t \mathcal I_N(s)\,\mathrm ds,
\end{equation}
where
\[
\mathcal I_N(s):=
\int\!\!\!\int_{x>y>0} \mathrm dx\,\mathrm dy\,
g_N(s,x)g_N(s,y)H_\varphi(x,y).
\]

We now define the symmetrized kernel
\[
\widetilde H_\varphi(x,y)
:=
(x+y)\varphi(x+y)
+
|x-y|\varphi(|x-y|)
-
2x\,\varphi(x)
-
2y\,\varphi(y),
\qquad x,y\ge0.
\]
Since \(\varphi\in C_c^2((0,\infty))\), the function
\(\widetilde H_\varphi\) is bounded and continuous on \([0,\infty)^2\).
Moreover,
\[
\widetilde H_\varphi(x,y)=\widetilde H_\varphi(y,x)
\qquad\text{for all }x,y\ge0.
\]
We also observe that, for \(x>y>0\),
\[
\widetilde H_\varphi(x,y)=H_\varphi(x,y).
\]
Therefore, we can write
\[
2\mathcal I_N(s)
=
\int_{[0,\infty)^2} \widetilde H_\varphi(x,y)\,
\nu_N(s,\mathrm d x)\nu_N(s,\mathrm d y),
\]
where
\[
\nu_N(s):=g_N(s,\omega)\,\mathrm d\omega.
\]
Similarly, if we define
\[
\mathcal I(s):=
\int\!\!\!\int_{x>y>0} \mathrm dx\,\mathrm dy\,
g(s,x)g(s,y)H_\varphi(x,y),
\]
and
\[
\nu(s):=g(s,\omega)\,\mathrm d\omega,
\]
then
\[
2\mathcal I(s)
=
\int_{[0,\infty)^2} \widetilde H_\varphi(x,y)\,
\nu(s,\mathrm d x)\nu(s,\mathrm d y).
\]

Since \(\widetilde H_\varphi\in C_b([0,\infty)^2)\), Step~6 yields, for every fixed
\(s\in[0,T]\),
\[
\int_{[0,\infty)^2} \widetilde H_\varphi(x,y)\,
\nu_N(s,\mathrm d x)\nu_N(s,\mathrm d y)
\longrightarrow
\int_{[0,\infty)^2} \widetilde H_\varphi(x,y)\,
\nu(s,\mathrm d x)\nu(s,\mathrm d y)
\qquad\text{as }N\to\infty.
\]
Hence, we have the convergence
\[
\mathcal I_N(s)\longrightarrow \mathcal I(s)
\qquad\text{for every }s\in[0,T]
\qquad\text{as }N\to\infty.
\]

Moreover, we know that
\[
|\mathcal I_N(s)|\le C_\varphi
\qquad\text{for all }N\ge N_0,\ s\in[0,T].
\]
Therefore, by the dominated convergence theorem, we have
\[
\int_0^t \mathcal I_N(s)\,\mathrm ds
\longrightarrow
\int_0^t \mathcal I(s)\,\mathrm ds
\qquad\text{as }N\to\infty
\]
for every \(t\in[0,T]\).

Next we consider the left-hand side of
\eqref{Prop:Existence:U=x:Hphi:weak-N}. Since
\(\varphi\in C_c^2((0,\infty))\), we have
\(
\varphi\in C_b([0,\infty)).
\)
By the convergence of the measures
\[
\nu_N(t)=g_N(t,\omega)\,\mathrm d\omega
\rightharpoonup
\nu(t)=g(t,\omega)\,\mathrm d\omega
\qquad\text{narrowly in }\mathcal M([0,\infty)) \qquad\text{as }N\to\infty,
\]
obtained in Step~6, we get
\[
\int_{(0,\infty)} \mathrm d\omega\, g_N(t,\omega)\varphi(\omega)
=
\int_{[0,\infty)} \varphi(\omega)\,\nu_N(t,\mathrm d\omega)
\longrightarrow
\int_{[0,\infty)} \varphi(\omega)\,\nu(t,\mathrm d\omega)
=
\int_{(0,\infty)} \mathrm d\omega\, g(t,\omega)\varphi(\omega)
\]
as \(N\to\infty\).

We next consider the initial term. Since \(f_{N,0}=f_0\) for every \(N\ge N_0\), we have
\[
h_N(0,\omega)=f_0(\omega)\chi_{\{\omega\le N\}}=f_0(\omega),
\]
and therefore
\[
g_N(0,\omega)=\omega h_N(0,\omega)=\omega f_0(\omega)=g_0(\omega).
\]
Consequently,
\[
\int_{(0,\infty)} \mathrm d\omega\, g_N(0,\omega)\varphi(\omega)
=
\int_{(0,\infty)} \mathrm d\omega\, g_0(\omega)\varphi(\omega)
\qquad\text{for all }N\ge N_0.
\]

Passing to the limit in \eqref{Prop:Existence:U=x:Hphi:weak-N}, we obtain
\[
\int_{(0,\infty)} \mathrm d\omega\, g(t,\omega)\varphi(\omega)
=
\int_{(0,\infty)} \mathrm d\omega\, g_0(\omega)\varphi(\omega)
+
2\int_0^t \mathrm ds
\int\!\!\!\int_{x>y>0} \mathrm dx\,\mathrm dy\,
g(s,x)g(s,y)\,H_\varphi(x,y),
\]
which is exactly \eqref{Prop:Existence:U=x:Hphi:weak}.

Since
\[
g(t,\omega)=\omega f^{\mathrm{reg}}(t,\omega),
\]
this is precisely the symmetric weak formulation for the non-cutoff equation satisfied by
the regular part \(f^{\mathrm{reg}}\). In particular, the possible concentration of the full measure limit
of \(f_N(t,\omega)\,\mathrm d\omega\) at \(\omega=0\) does not contribute to this formulation,
as the equation is written in terms of \(g=\omega f^{\mathrm{reg}}\).

\medskip
\noindent
\textit{Step 8: Entropy bound for the regular part.}

For every \(N\) and every \(t\ge0\), we have, by Lemma \ref{Lemma:EntropyTestCutoff}
\begin{equation}\label{Prop:Existence:U=x:Hphi:entropy-N}
	\int_{[0,\infty)} \mathrm d\omega\, \frac{e(\omega h_N(t,\omega))}{\omega}
	\le
	\int_{[0,\infty)} \mathrm d\omega\, \frac{e(\omega f_0(\omega))}{\omega}.
\end{equation}

We fix \(0<a<b<\infty\). Since \(a\le\omega\le b\) on \((a,b)\), we define

\[
G_{a,b}(\omega,r):=\chi_{(a,b)}(\omega)\,\frac{e(\omega r)}{\omega},
\qquad \omega>0,\quad r\ge0.
\]
For each fixed \(r\ge0\), the map \(\omega\mapsto G_{a,b}(\omega,r)\) is measurable. For almost every fixed \(\omega\), the map \(r\mapsto G_{a,b}(\omega,r)\) is convex and lower semicontinuous on \([0,\infty)\), as it is either identically zero or a positive multiple of the convex map \(r\mapsto e(\omega r)\). Hence \(G_{a,b}\) is a nonnegative Carath\'eodory convex integrand. Since \(h_N(t,\cdot)\rightharpoonup f^{\mathrm{reg}}(t,\cdot)\) weakly in \(L^1((a,b))\), the lower semicontinuity theorem for integral functionals associated with nonnegative Carath\'eodory convex integrands yields
\[
\int_a^b \mathrm d\omega\, \frac{e(\omega f^{\mathrm{reg}}(t,\omega))}{\omega}
=
\int_{(0,\infty)} \mathrm d\omega\, G_{a,b}(\omega,f^{\mathrm{reg}}(t,\omega))
\le
\liminf_{N\to\infty}
\int_{(0,\infty)} \mathrm d\omega\, G_{a,b}(\omega,h_N(t,\omega)),
\]
see \cite[Theorem 3.20]{dacorogna2008direct} and \cite{ioffe1977lower}. That is,
\[
\int_a^b \mathrm d\omega\, \frac{e(\omega f^{\mathrm{reg}}(t,\omega))}{\omega}
\le
\liminf_{N\to\infty}
\int_a^b \mathrm d\omega\, \frac{e(\omega h_N(t,\omega))}{\omega}.
\]
Using \eqref{Prop:Existence:U=x:Hphi:entropy-N}, we infer
\[
\int_a^b \mathrm d\omega\, \frac{e(\omega f^{\mathrm{reg}}(t,\omega))}{\omega}
\le
\int_{[0,\infty)} \mathrm d\omega\, \frac{e(\omega f_0(\omega))}{\omega}.
\]

Finally, let \(a\downarrow0\) and \(b\uparrow\infty\). Since the integrand is nonnegative, the monotone convergence theorem gives
\[
\int_{[0,\infty)} \mathrm d\omega\, \frac{e(\omega f^{\mathrm{reg}}(t,\omega))}{\omega}
=
\lim_{a\downarrow0,\;b\uparrow\infty}
\int_a^b \mathrm d\omega\, \frac{e(\omega f^{\mathrm{reg}}(t,\omega))}{\omega}
\le
\int_{[0,\infty)} \mathrm d\omega\, \frac{e(\omega f_0(\omega))}{\omega}.
\]

\medskip
\noindent
\textit{Step 9: Dissipation bound.}

We finally prove the dissipation estimate. We apply
Lemma~\ref{Lemma:EntropyTestCutoff} with the linear entropy
\[
\ell(r):=r,\qquad r\ge 0.
\]
Then \(\ell\in C^1([0,\infty))\), \(\ell\) is convex, \(\ell(0)=0\), and
\(\ell\ge 0\). Moreover,
\[
\ell'(r)=1
\qquad\text{for all }r\ge 0.
\]
Therefore Lemma~\ref{Lemma:EntropyTestCutoff} gives, for every \(t>0\),
\[
2\int_0^t \mathrm{d}s
\int_0^\infty \mathrm{d}\omega
\int_0^\omega \mathrm{d}\omega_1\,
U(\omega_1)U(\omega-\omega_1)
h_N(s,\omega)h_N(s,\omega_1)
\le
\mathcal E_\ell[f_{N,0}].
\]
Since \(N\ge N_0\) implies \(f_{N,0}=f_0\), and since
\[
\mathcal E_\ell[f_0]
=
\int_{[0,\infty)} \mathrm d\omega\,
\frac{\ell(U(\omega)f_0(\omega))}{U(\omega)}
=
\int_{[0,\infty)} f_0(\omega)\,\mathrm d\omega
=:M_0,
\]
we obtain
\begin{equation}\label{eq:linear-dissipation-fN}
2\int_0^t \mathrm{d}s
\int_0^\infty \mathrm{d}\omega
\int_0^\omega \mathrm{d}\omega_1\,
U(\omega_1)U(\omega-\omega_1)
h_N(s,\omega)h_N(s,\omega_1)
\le M_0 .
\end{equation}

	When \(U(\omega)=\omega\),
we rewrite this estimate in terms of
\[
g_N(s,\omega):=\omega h_N(s,\omega).
\]
We set \(x:=\omega\) and \(y:=\omega_1\), then
\[
\omega_1(\omega-\omega_1)h_N(s,\omega)h_N(s,\omega_1)
=
\frac{x-y}{x}g_N(s,x)g_N(s,y),
\qquad x>y>0.
\]
Define
\[
\kappa(x,y)
:=
\begin{cases}
\dfrac{x-y}{x}, & x>y\ge 0,\ x>0,\\[2mm]
0, & \text{otherwise}.
\end{cases}
\]
Then \(0\le \kappa\le 1\), and \(\kappa\) is lower semicontinuous on
\([0,\infty)^2\). Since
\[
\nu_N(s):=g_N(s,\omega)\,\mathrm d\omega,
\]
then 
\begin{equation}\label{eq:kappa-bound-N}
2\int_0^t
\int_{[0,\infty)^2}
\kappa(x,y)\,
\nu_N(s,\mathrm dx)\nu_N(s,\mathrm dy)
\,\mathrm ds
\le M_0 .
\end{equation}

Since \(\kappa\) is nonnegative, bounded, and lower semicontinuous, the Portmanteau theorem gives
\[
\int_{[0,\infty)^2}
\kappa(x,y)\,\nu(s,\mathrm dx)\nu(s,\mathrm dy)
\le
\liminf_{N\to\infty}
\int_{[0,\infty)^2}
\kappa(x,y)\,\nu_N(s,\mathrm dx)\nu_N(s,\mathrm dy)
\]
for every fixed \(s\ge0\).

Using Fatou's lemma in time and then \eqref{eq:kappa-bound-N}, we obtain, for every \(t>0\),
\[
\begin{aligned}
&2\int_0^t
\int_{[0,\infty)^2}
\kappa(x,y)\,\nu(s,\mathrm dx)\nu(s,\mathrm dy)
\,\mathrm ds
\\
&\qquad\le
2\int_0^t
\liminf_{N\to\infty}
\int_{[0,\infty)^2}
\kappa(x,y)\,\nu_N(s,\mathrm dx)\nu_N(s,\mathrm dy)
\,\mathrm ds
\\
&\qquad\le
2\liminf_{N\to\infty}
\int_0^t
\int_{[0,\infty)^2}
\kappa(x,y)\,\nu_N(s,\mathrm dx)\nu_N(s,\mathrm dy)
\,\mathrm ds
\\
&\qquad\le M_0 .
\end{aligned}
\]
Since \(g(s,\omega)=\omega f^{\mathrm{reg}}(s,\omega)\), we have
\[
\int_{[0,\infty)^2}
\kappa(x,y)\,\nu(s,\mathrm dx)\nu(s,\mathrm dy)
=
\int_0^\infty \mathrm dx
\int_0^x \mathrm dy\,
y(x-y)
f^{\mathrm{reg}}(s,x)f^{\mathrm{reg}}(s,y).
\]
Therefore, for every \(t>0\),
\[
2\int_0^t \mathrm ds
\int_0^\infty \mathrm d\omega
\int_0^\omega \mathrm d\omega_1\,
\omega_1(\omega-\omega_1)
f^{\mathrm{reg}}(s,\omega)f^{\mathrm{reg}}(s,\omega_1)
\le M_0 .
\]
Finally, since the integrand is nonnegative, the monotone convergence theorem gives, by
letting \(t\to\infty\),
\begin{equation}\label{eq:dissipation-assumption-limit}
2\int_0^\infty \mathrm ds
\int_0^\infty \mathrm d\omega
\int_0^\omega \mathrm d\omega_1\,
\omega_1(\omega-\omega_1)
f^{\mathrm{reg}}(s,\omega)f^{\mathrm{reg}}(s,\omega_1)
\le
\int_0^\infty f_0(\omega)\,\mathrm d\omega .
\end{equation}

This completes the proof.
\end{proof}

\subsection{Passing to the limit - Case B: $U(\omega)=\omega^\rho(1+\omega)^{-\beta}$}\label{Sec:CaseB}
\begin{proposition}
	\label{prop:global-Ux2-Hphi}
	Let
	\[
	U(\omega)=\omega^\rho(1+\omega)^{-\beta},
	\qquad
	\rho-\beta<2,
	\qquad
	\rho\ge1,
	\qquad
	\beta\le \rho-1.
	\]
	Let \(f_0\in L^1((0,\infty))\) be nonnegative and compactly supported in \((0,\infty)\). Assume moreover that there exists a convex function
	\[
	e:[0,\infty)\to[0,\infty),
	\qquad
	e(0)=0,
	\qquad
	e'(0)=0,
	\qquad
	e\in C^\infty[0,\infty),
	\qquad
	\lim_{r\to\infty}\frac{e(r)}{r}=+\infty,
	\]
	 such that
	\[
	\int_{[0,\infty)} \mathrm d\omega\,
	\frac{e\!\bigl(\omega^\rho(1+\omega)^{-\beta} f_0(\omega)\bigr)}
	{\omega^\rho(1+\omega)^{-\beta}}
	<\infty.
	\]

	For each \(N\in\mathbb N\), let \(f_N\) be the global nonnegative solution of \eqref{Cutoff2-new}, and define
	\[
	h_N(t,\omega):=f_N(t,\omega)\chi_{\{\omega\le N\}}.
	\]
	
Then there exist a subsequence, still denoted by \(\{f_N\}_{N\ge1}\), and a nonnegative function
\[
f^{\mathrm{reg}}\in C\bigl([0,\infty);w\text{-}L^1_{\mathrm loc}((0,\infty))\bigr)
\]
such that the following hold.
	
	For every \(T>0\), every \(0<a<b<\infty\), and every \(\psi\in L^\infty((a,b))\),
	\[
	\int_a^b \mathrm d\omega\, \psi(\omega)\,h_N(t,\omega)
	\longrightarrow
	\int_a^b \mathrm d\omega\, \psi(\omega)\,f^{\mathrm{reg}}(t,\omega)
	\]
	uniformly in \(t\in[0,T]\) as \(N\to\infty\).
	
Moreover, for every fixed \(t\ge0\), the family
\[
\bigl\{f_N(t,\omega)\,\mathrm d\omega\bigr\}_{N\ge1}
\]
is relatively compact for the narrow topology of \(\mathcal M([0,\infty))\). Every narrow cluster
point \(\mu_t\) of this family has the form
\[
\mu_t=\alpha(t)\delta_{\omega=0}+f^{\mathrm{reg}}(t,\omega)\,\mathrm d\omega
\]
for some \(\alpha(t)\ge0\), possibly depending on the chosen cluster point.
	
	Defining
	\[
	g(t,\omega):=\omega^\rho(1+\omega)^{-\beta} f^{\mathrm{reg}}(t,\omega),
	\qquad
	g_0(\omega):=\omega^\rho(1+\omega)^{-\beta} f_0(\omega),
	\qquad
	g_N(t,\omega):=\omega^\rho(1+\omega)^{-\beta} h_N(t,\omega),
	\]
	we have, for every \(T>0\), every \(0<b<\infty\), and every \(\psi\in L^\infty((0,b))\),
	\[
	\int_0^b \mathrm d\omega\, \psi(\omega)\,g_N(t,\omega)
	\longrightarrow
	\int_0^b \mathrm d\omega\, \psi(\omega)\,g(t,\omega)
	\]
	uniformly in \(t\in[0,T]\) as \(N\to\infty\).
	
	Furthermore, for every \(\varphi\in C_c^2((0,\infty))\) and every \(t\ge0\),
	\begin{equation}\label{eq:weak-Hphi-final-Ux2}
		\int_0^\infty \mathrm d\omega\, g(t,\omega)\varphi(\omega)
		=
		\int_0^\infty \mathrm d\omega\, g_0(\omega)\varphi(\omega)
		+
		2\int_0^t \mathrm ds
		\int\!\!\!\int_{x>y>0} \mathrm dx\,\mathrm dy\,
		g(s,x)g(s,y)\,\mathcal H_\varphi(x,y),
	\end{equation}
	where
	\[
	\mathcal H_\varphi(x,y)
	=
	(x+y)^\rho(1+x+y)^{-\beta}\varphi(x+y)
	+
	(x-y)^\rho(1+x-y)^{-\beta}\varphi(x-y)
	\]
	\[
	\qquad
	-
	\Bigl((x-y)^\rho(1+x-y)^{-\beta}+(x+y)^\rho(1+x+y)^{-\beta}\Bigr)\varphi(x)
	\]
	\[
	\qquad
	+
	\Bigl((x-y)^\rho(1+x-y)^{-\beta}-(x+y)^\rho(1+x+y)^{-\beta}\Bigr)\varphi(y).
	\]
	
	In particular, \eqref{eq:weak-Hphi-final-Ux2} is a weak form of \eqref{3wave} for
	\[
	f(t,\omega)=f^{\mathrm{reg}}(t,\omega).
	\]
	
	Finally, 
	\[
	\int_{[0,\infty)} \mathrm d\omega\,
	\frac{e\!\bigl(\omega^\rho(1+\omega)^{-\beta} f^{\mathrm{reg}}(t,\omega)\bigr)}
	{\omega^\rho(1+\omega)^{-\beta}}
	\le
	\int_{[0,\infty)} \mathrm d\omega\,
	\frac{e\!\bigl(\omega^\rho(1+\omega)^{-\beta} f_0(\omega)\bigr)}
	{\omega^\rho(1+\omega)^{-\beta}}
	\qquad\text{for all }t\ge0,
	\]
	and  \begin{equation}\label{eq:dissipation-assumption:3}
	2	\int_0^\infty \mathrm{d}s
		\int_0^\infty \mathrm{d}\omega
		\int_0^\omega \mathrm{d}\omega_1\,
		U(\omega_1)U(\omega-\omega_1)\,f_N(s,\omega)\,f_N(s,\omega_1)\,
		\le
		\mathcal{E}_e[f_0].
	\end{equation}

\end{proposition}

\begin{proof}
	We divide the proof into nine steps.
	
	\medskip
	\noindent
\textit{Step 1: Bounds for the approximating sequence.}

By Lemma~\ref{Lemma:EnergyCutoff}, we have
\begin{equation}\label{eq:Ux2-third-moment}
	\int_0^\infty \mathrm d\omega\, \omega U(\omega)f_N(t,\omega)
	=
	\int_0^\infty \mathrm d\omega\, \omega U(\omega) f_0(\omega)
	=:E_3
	\qquad\text{for all }t\ge0.
\end{equation}

Next, by Lemma~\ref{Lemma:EntropyTestCutoff}, we find
\begin{equation}\label{eq:Ux2-localL1}
	\int_{[0,\infty)} \mathrm d\omega\, h_N(t,\omega)
	=
	\int_0^N \mathrm d\omega\, f_N(t,\omega)
	\le
	\int_0^\infty \mathrm d\omega\, f_0(\omega)
	=:M_0
	\qquad\text{for all }t\ge0.
\end{equation}

Since
\[
U(\omega)=\omega^\rho(1+\omega)^{-\beta}\le C(1+\omega U(\omega))
\qquad\text{for all }\omega\ge0,\text{ where   }C>0 \text{ is a constant},
\]
we infer
\[
\int_{[0,\infty)} \mathrm d\omega\, U(\omega) h_N(t,\omega)
\le
C\int_{[0,\infty)} \mathrm d\omega\, h_N(t,\omega)
+
C\int_{[0,\infty)} \mathrm d\omega\, \omega U(\omega) h_N(t,\omega)
\le
C(M_0+E_3).
\]

Therefore
\begin{equation}\label{eq:Ux2-g-mass}
	\int_{[0,\infty)} \mathrm d\omega\, g_N(t,\omega)
	\le
	M_2:=C(M_0+E_3)
	\qquad\text{for all }t\ge0,
\end{equation}
and
\begin{equation}\label{eq:Ux2-g-firstmoment}
	\int_{[0,\infty)} \mathrm d\omega\, \omega g_N(t,\omega)
	=
	\int_{[0,\infty)} \mathrm d\omega\, \omega U(\omega) h_N(t,\omega)
	\le
	E_3
	\qquad\text{for all }t\ge0.
\end{equation}

\medskip
\noindent
\textit{Step 2: Local entropy bound for \(h_N\).}

We now use the entropy estimate. By Lemma~\ref{Lemma:EntropyTestCutoff}, we find
\[
\begin{aligned}
\int_{[0,\infty)} \mathrm d\omega\,
\frac{e\!\bigl(\omega^\rho(1+\omega)^{-\beta} h_N(t,\omega)\bigr)}
{\omega^\rho(1+\omega)^{-\beta}}
&=
\int_0^N \mathrm d\omega\,
\frac{e\!\bigl(\omega^\rho(1+\omega)^{-\beta} f_N(t,\omega)\bigr)}
{\omega^\rho(1+\omega)^{-\beta}}
\\
&\le
\int_0^\infty \mathrm d\omega\,
\frac{e\!\bigl(\omega^\rho(1+\omega)^{-\beta} f_0(\omega)\bigr)}
{\omega^\rho(1+\omega)^{-\beta}}
=:C_e.
\end{aligned}
\]
Repeating the arguments used in Step~2 of the proof of
Proposition~\ref{Prop:Existence:U=x:Hphi}, we obtain that the family
\[
\bigl\{h_N(t,\cdot):\,N\ge1,\ t\ge0\bigr\}
\]
is uniformly integrable on \((a,b)\). In addition, the family
\begin{equation}\label{eq:Ux2-g-dVP}
	\bigl\{g_N(t,\cdot):\,N\ge1,\ t\ge0\bigr\}
	\qquad\text{is uniformly integrable on }(0,b).
\end{equation}
	
	\medskip
	\noindent
	\textit{Step 3: Weak formulation for \(g_N\).}
	
We let \(\varphi\in C_c^2((0,\infty))\), and define
\[
\phi(\omega):=\omega^\rho(1+\omega)^{-\beta}\varphi(\omega).
\]
Since \(\varphi\in C_c^2((0,\infty))\), there exists \(N_\varphi\ge1\) such that
\[
\operatorname{supp}\varphi\subset(0,N_\varphi).
\]
In what follows we take \(N\ge N_\varphi\). Then, on the support of \(\varphi\),
\[
g_N(t,\omega)=U(\omega)f_N(t,\omega),
\]
and hence
\[
\int_0^\infty \mathrm d\omega\, \phi(\omega)f_N(t,\omega)
=
\int_0^\infty \mathrm d\omega\, g_N(t,\omega)\varphi(\omega).
\]
Since \(f_N\) solves \eqref{Cutoff2-new}, we have
\[
\frac{\mathrm d}{\mathrm dt}
\int_0^\infty \mathrm d\omega\, g_N(t,\omega)\varphi(\omega)
=
\int_0^\infty \mathrm d\omega\, \phi(\omega)\mathcal Q_N[f_N](t,\omega).
\]

We decompose the right-hand side into six terms:
\[
\int \phi\,\mathcal Q_N[f_N]
=
A_1-A_2-A_3-2B_1+2B_2+2B_3,
\]
where
\[
\begin{aligned}
	A_1
	&:=
	\int_0^\infty \mathrm d\omega\, \phi(\omega)
	\int_0^\omega \mathrm d\omega_1\,
	\omega_1^\rho(1+\omega_1)^{-\beta}(\omega-\omega_1)^\rho(1+\omega-\omega_1)^{-\beta}
	f_N(t,\omega_1)f_N(t,\omega-\omega_1)
	\chi_{\{\omega_1,\omega-\omega_1\le N\}},
	\\[1mm]
	A_2
	&:=
	\int_0^\infty \mathrm d\omega\, \phi(\omega)
	\int_0^\omega \mathrm d\omega_1\,
	\omega_1^\rho(1+\omega_1)^{-\beta}(\omega-\omega_1)^\rho(1+\omega-\omega_1)^{-\beta}
	f_N(t,\omega)f_N(t,\omega_1)
	\chi_{\{\omega,\omega_1\le N\}},
	\\[1mm]
	A_3
	&:=
	\int_0^\infty \mathrm d\omega\, \phi(\omega)
	\int_0^\omega \mathrm d\omega_1\,
	\omega_1^\rho(1+\omega_1)^{-\beta}(\omega-\omega_1)^\rho(1+\omega-\omega_1)^{-\beta}
	f_N(t,\omega)f_N(t,\omega-\omega_1)
	\chi_{\{\omega,\omega-\omega_1\le N\}},
\end{aligned}
\]
and
\[
\begin{aligned}
	B_1
	&:=
	\int_0^\infty \mathrm d\omega\, \phi(\omega)
	\int_0^\infty \mathrm d\omega_1\,
	\omega_1^\rho(1+\omega_1)^{-\beta}(\omega+\omega_1)^\rho(1+\omega+\omega_1)^{-\beta}
	f_N(t,\omega_1)f_N(t,\omega)
	\chi_{\{\omega,\omega_1\le N\}},
	\\[1mm]
	B_2
	&:=
	\int_0^\infty \mathrm d\omega\, \phi(\omega)
	\int_0^\infty \mathrm d\omega_1\,
	\omega_1^\rho(1+\omega_1)^{-\beta}(\omega+\omega_1)^\rho(1+\omega+\omega_1)^{-\beta}
	f_N(t,\omega_1)f_N(t,\omega+\omega_1)
	\chi_{\{\omega_1,\omega+\omega_1\le N\}},
	\\[1mm]
	B_3
	&:=
	\int_0^\infty \mathrm d\omega\, \phi(\omega)
	\int_0^\infty \mathrm d\omega_1\,
	\omega_1^\rho(1+\omega_1)^{-\beta}(\omega+\omega_1)^\rho(1+\omega+\omega_1)^{-\beta}
	f_N(t,\omega)f_N(t,\omega+\omega_1)
	\chi_{\{\omega,\omega+\omega_1\le N\}}.
\end{aligned}
\]

Repeating the arguments used in Step~3 of the proof of Proposition~\ref{Prop:Existence:U=x:Hphi}, we may rewrite
\[
A_1-A_2-A_3-2B_1+2B_2+2B_3
\]
and obtain
\[
\begin{aligned}
	\frac{\mathrm d}{\mathrm dt}\int_0^\infty \mathrm d\omega\, g_N(t,\omega)\varphi(\omega)
	&=
	2\int\!\!\!\int_{x>y>0} \mathrm dx\,\mathrm dy\,
	g_N(t,x)g_N(t,y)
	\Bigl[
		(x+y)^\rho(1+x+y)^{-\beta}\varphi(x+y)
		\\
		&\qquad\qquad
		-(x-y)^\rho(1+x-y)^{-\beta}\varphi(x)
		-(x+y)^\rho(1+x+y)^{-\beta}\varphi(x)
		\\
		&\qquad\qquad
		-(x+y)^\rho(1+x+y)^{-\beta}\varphi(y)
		+(x-y)^\rho(1+x-y)^{-\beta}\varphi(x-y)
		\\
		&\qquad\qquad
		+(x-y)^\rho(1+x-y)^{-\beta}\varphi(y)
	\Bigr].
\end{aligned}
\]

Therefore
\[
\frac{\mathrm d}{\mathrm dt}\int_0^\infty \mathrm d\omega\, g_N(t,\omega)\varphi(\omega)
=
2\int\!\!\!\int_{x>y>0} \mathrm dx\,\mathrm dy\,
g_N(t,x)g_N(t,y)\,\mathcal H_\varphi(x,y),
\]
where
\[
\mathcal H_\varphi(x,y)
=
(x+y)^\rho(1+x+y)^{-\beta}\varphi(x+y)
+
(x-y)^\rho(1+x-y)^{-\beta}\varphi(x-y)
\]
\[
\qquad
-
\Bigl((x-y)^\rho(1+x-y)^{-\beta}+(x+y)^\rho(1+x+y)^{-\beta}\Bigr)\varphi(x)
\]
\[
\qquad
+
\Bigl((x-y)^\rho(1+x-y)^{-\beta}-(x+y)^\rho(1+x+y)^{-\beta}\Bigr)\varphi(y).
\]
Equivalently, we have
\begin{equation}\label{eq:weak-Hphi-N-Ux2}
	\int_0^\infty \mathrm d\omega\, g_N(t,\omega)\varphi(\omega)
	=
	\int_0^\infty \mathrm d\omega\, g_N(0,\omega)\varphi(\omega)
	+
	2\int_0^t \mathrm ds
	\int\!\!\!\int_{x>y>0} \mathrm dx\,\mathrm dy\,
	g_N(s,x)g_N(s,y)\,\mathcal H_\varphi(x,y).
\end{equation}	
	\medskip
	\noindent
\textit{Step 4: Uniform boundedness and equicontinuity.}

Let \(\psi\in C_c^2((0,\infty))\), and choose \(0<a<b<\infty\) such that
\[
\operatorname{supp}\psi\subset(a,b).
\]
We define
\[
\varphi(\omega):=\frac{\psi(\omega)}{\omega^\rho(1+\omega)^{-\beta}}
=
\psi(\omega)\,\omega^{-\rho}(1+\omega)^\beta.
\]
It then follows that \(\varphi\in C_c^2((0,\infty))\), \(\operatorname{supp}\varphi\subset(a,b)\), and
\[
G_N^\psi(t):=
\int_0^\infty \mathrm d\omega\, \psi(\omega)h_N(t,\omega)
=
\int_0^\infty \mathrm d\omega\, \varphi(\omega)g_N(t,\omega).
\]
Therefore
\[
\frac{\mathrm d}{\mathrm dt}G_N^\psi(t)
=
2\int\!\!\!\int_{x>y>0} \mathrm dx\,\mathrm dy\,
g_N(t,x)g_N(t,y)\,\mathcal H_\varphi(x,y).
\]

We estimate the four terms in \(\mathcal H_\varphi\) separately.

First, if \(\varphi(x+y)\neq0\), then \(x+y\in(a,b)\), hence \(0<x<b\) and \(0<y<b\). We now estimate 
\[
(x+y)^\rho(1+x+y)^{-\beta}|\varphi(x+y)|
=
|\psi(x+y)|
\le
\|\psi\|_{L^\infty}.
\]
Thus
\[
\begin{aligned}
	&\int\!\!\!\int_{x>y>0} \mathrm dx\,\mathrm dy\,
	g_N(t,x)g_N(t,y)(x+y)^\rho(1+x+y)^{-\beta}|\varphi(x+y)|
	\\
	&\qquad\le
	\|\psi\|_{L^\infty}
	\int\!\!\!\int_{(0,b)^2} \mathrm dx\,\mathrm dy\,
	g_N(t,x)g_N(t,y)
	\le
	\|\psi\|_{L^\infty}M_2^2.
\end{aligned}
\]

Next, if \(\varphi(x-y)\neq0\), then \(x-y\in(a,b)\), so \(x=y+z\) with \(z\in(a,b)\). Since
\[
(x-y)^\rho(1+x-y)^{-\beta}|\varphi(x-y)|
=
|\psi(x-y)|,
\]
we can bound
\[
\begin{aligned}
	&\int\!\!\!\int_{x>y>0} \mathrm dx\,\mathrm dy\,
	g_N(t,x)g_N(t,y)(x-y)^\rho(1+x-y)^{-\beta}|\varphi(x-y)|
	\\
	&\qquad=
	\int_0^\infty \mathrm dy\, g_N(t,y)
	\int_a^b \mathrm dz\, g_N(t,y+z)|\psi(z)|.
\end{aligned}
\]
Since \(z\in(a,b)\), we estimate
\[
\int_a^b \mathrm dz\, g_N(t,y+z)|\psi(z)|
\le
\|\psi\|_{L^\infty}\int_{y+a}^{y+b} \mathrm dx\, g_N(t,x).
\]
For \(x\ge y+a\), we observe that
\[
g_N(t,x)\le \frac{x}{y+a}g_N(t,x),
\]
hence
\[
\int_{y+a}^{y+b} \mathrm dx\, g_N(t,x)
\le
\frac1{y+a}\int_{y+a}^{y+b} \mathrm dx\, xg_N(t,x)
\le
\frac{E_3}{y+a}.
\]
Therefore
\[
\begin{aligned}
	&\int\!\!\!\int_{x>y>0} \mathrm dx\,\mathrm dy\,
	g_N(t,x)g_N(t,y)(x-y)^\rho(1+x-y)^{-\beta}|\varphi(x-y)|
	\\
	&\qquad\le
	\|\psi\|_{L^\infty}E_3
	\int_0^\infty \mathrm dy\, \frac{g_N(t,y)}{y+a}
	\le
	\frac1a\|\psi\|_{L^\infty}E_3M_2.
\end{aligned}
\]

We now estimate the term containing \(\varphi(x)\). If \(\varphi(x)\neq0\), then \(x\in(a,b)\). Since \(x>y>0\), we have \(0<y<x<b\), and therefore
\[
0<x-y<x<b,
\qquad
0<x+y<2b.
\]
Hence there exists a constant \(C_{a,b}>0\), depending only on \(a,b,\rho,\beta\), such that
\[
\Bigl((x-y)^\rho(1+x-y)^{-\beta}+(x+y)^\rho(1+x+y)^{-\beta}\Bigr)|\varphi(x)|
\le
C_{a,b}\|\psi\|_{L^\infty}.
\]
Therefore,   we can bound
\[
\begin{aligned}
	&\int\!\!\!\int_{x>y>0} \mathrm dx\,\mathrm dy\,
	g_N(t,x)g_N(t,y)
	\Bigl((x-y)^\rho(1+x-y)^{-\beta}+(x+y)^\rho(1+x+y)^{-\beta}\Bigr)|\varphi(x)|
	\\
	&\qquad\le
	C_{a,b}\|\psi\|_{L^\infty}
	\int_a^b \mathrm dx\, g_N(t,x)\int_0^x \mathrm dy\, g_N(t,y)
	\le
	C_{a,b}\|\psi\|_{L^\infty}M_2^2.
\end{aligned}
\]

Finally, if \(\varphi(y)\neq0\), then \(y\in(a,b)\), and  thus \(y\le b\). We claim that there exists a constant \(C_b>0\), depending only on \(b,\rho,\beta\), such that
\[
\left|
(x-y)^\rho(1+x-y)^{-\beta}
-
(x+y)^\rho(1+x+y)^{-\beta}
\right|
\le
C_b(1+x)
\qquad\text{for all }x>y,\ y\in(a,b).
\]
Indeed, if we define
\begin{equation}\label{F}
F(z):=z^\rho(1+z)^{-\beta},
\end{equation}
then
\[
F'(z)
=
z^{\rho-1}(1+z)^{-\beta-1}\bigl(\rho+(\rho-\beta)z\bigr),
\]
and therefore, since \(\rho-\beta<2\), we estimate
\[
|F'(z)|\le C(1+z)^{\rho-\beta-1}\le C(1+z)
\qquad\text{for all }z\ge0.
\]
By the mean value theorem, for some \(\xi\in(x-y,x+y)\), we have
\[
|F(x-y)-F(x+y)|
\le
2y\,|F'(\xi)|
\le
C_b(1+x),
\]
because \(0<y\le b\) and \(0\le \xi\le x+y\le x+b\).

Since moreover
\[
|\varphi(y)|
=
|\psi(y)|\,y^{-\rho}(1+y)^\beta
\le
C_{a,b}\|\psi\|_{L^\infty},
\]
we find
\[
\left|
\Bigl((x-y)^\rho(1+x-y)^{-\beta}
-
(x+y)^\rho(1+x+y)^{-\beta}\Bigr)\varphi(y)
\right|
\le
C_{a,b}\|\psi\|_{L^\infty}(1+x).
\]
Hence, we can bound
\[
\begin{aligned}
	&\int\!\!\!\int_{x>y>0} \mathrm dx\,\mathrm dy\,
	g_N(t,x)g_N(t,y)
	\left|
	\Bigl((x-y)^\rho(1+x-y)^{-\beta}
	-
	(x+y)^\rho(1+x+y)^{-\beta}\Bigr)\varphi(y)
	\right|
	\\
	&\qquad\le
	C_{a,b}\|\psi\|_{L^\infty}
	\int_a^b \mathrm dy\, g_N(t,y)
	\int_y^\infty \mathrm dx\, (1+x)g_N(t,x).
\end{aligned}
\]
Since \(x\ge y\ge a\), we have \(1+x\le (1+\tfrac1a)x\). Therefore, we find
\[
\int_y^\infty \mathrm dx\, (1+x)g_N(t,x)
\le
\left(1+\frac1a\right)\int_y^\infty \mathrm dx\, xg_N(t,x)
\le
\left(1+\frac1a\right)E_3,
\]
and thus we bound
\[
\begin{aligned}
	&\int\!\!\!\int_{x>y>0} \mathrm dx\,\mathrm dy\,
	g_N(t,x)g_N(t,y)
	\left|
	\Bigl((x-y)^\rho(1+x-y)^{-\beta}
	-
	(x+y)^\rho(1+x+y)^{-\beta}\Bigr)\varphi(y)
	\right|
	\\
	&\qquad\le
	C_{a,b}\|\psi\|_{L^\infty}M_2E_3.
\end{aligned}
\]

Combining the four estimates, we find a constant \(C_{\varphi}>0\), independent of \(N\) and \(t\), such that
\[
\left|
\int\!\!\!\int_{x>y>0} \mathrm dx\,\mathrm dy\,
g_N(t,x)g_N(t,y)\,\mathcal H_\varphi(x,y)
\right|
\le
C_{\varphi}
\qquad\text{for all }N\ge1,\ t\ge0.
\]
Hence
\[
\left|\frac{\mathrm d}{\mathrm dt}G_N^\psi(t)\right|
\le
2C_{\varphi}
\qquad\text{for all }N\ge1,\ t\ge0.
\]
Thus \(\{G_N^\psi\}\) is uniformly Lipschitz on every interval \([0,T]\). In addition, using \eqref{eq:Ux2-localL1}, we get
\[
|G_N^\psi(t)|
\le
\|\psi\|_{L^\infty}\int_a^b \mathrm d\omega\, h_N(t,\omega)
\le
\|\psi\|_{L^\infty}M_0.
\]
Therefore \(\{G_N^\psi\}\) is uniformly bounded and uniformly equicontinuous on every interval \([0,T]\).

\medskip
	\noindent
	\textit{Step 5: Weak convergence of \(g_N\).}

We fix \(T>0\). By \eqref{eq:Ux2-localL1} and Step~4, Lemma~\ref{lem:compactness-local-measures} can be applied to the family \(\{h_N\}\) on \([0,T]\). As a consequence, after extracting a subsequence, there exists
\[
f^{\mathrm{reg}}
\in C\bigl([0,T];w\text{-}\mathcal M_{\mathrm loc}((0,\infty))\bigr)
\]
such that, for every \(\psi\in C_c((0,\infty))\),
\[
\int_0^\infty \mathrm d\omega\, \psi(\omega)h_N(t,\omega)
\longrightarrow
\int_0^\infty \mathrm d\omega\, \psi(\omega)f^{\mathrm{reg}}(t,\omega)
\qquad\text{as }N\to\infty,
\]
uniformly in \(t\in[0,T]\).

We now upgrade this convergence to local \(L^1\)-convergence. We fix \(0<a<b<\infty\). By Step 2, the family
\[
\bigl\{h_N(t,\cdot):\,N\ge1,\ t\in[0,T]\bigr\}
\]
is uniformly integrable on \((a,b)\). Together with \eqref{eq:Ux2-localL1}, this implies relative weak compactness in \(L^1((a,b))\) by the Dunford--Pettis theorem \cite{arkeryd1972boltzmann,komornik2016lectures}. Arguing exactly as in the  Proposition \ref{Prop:Existence:U=x:Hphi}, we conclude that for every \(t\in[0,T]\),
\[
f^{\mathrm{reg}}(t,\cdot)\in L^1((a,b)),
\qquad
h_N(t,\cdot)\rightharpoonup f^{\mathrm{reg}}(t,\cdot)\quad\text{weakly in }L^1((a,b)) \qquad\text{as }N\to\infty.
\]
Since \((a,b)\Subset(0,\infty)\) was arbitrarily chosen, we find
\[
f^{\mathrm{reg}}\in C\bigl([0,T];w\text{-}L^1_{\mathrm loc}((0,\infty))\bigr).
\]

We next upgrade the uniform convergence in time from continuous test functions to arbitrary
\(L^\infty((a,b))\)-test functions. That means for every \(\psi\in L^\infty((a,b))\),
\[
\sup_{t\in[0,T]}
\left|
\int_a^b \mathrm d\omega\, \psi(\omega)h_N(t,\omega)
-
\int_a^b \mathrm d\omega\, \psi(\omega)f^{\mathrm{reg}}(t,\omega)
\right|\to0
\qquad\text{as }N\to\infty.
\]
The proof is exactly the same as in the \(U(\omega)=\omega\) case and is omitted.

We now identify the possible full measure limits of \(f_N\). We fix \(t\ge0\).
Although the mass estimate above is stated only for \(h_N=f_N\chi_{\{\omega\le N\}}\),
it also yields a uniform full-mass bound for \(f_N(t,\cdot)\). Indeed,
\[
\int_{[0,\infty)} \mathrm d\omega\, f_N(t,\omega)
=
\int_0^N \mathrm d\omega\, f_N(t,\omega)
+
\int_N^\infty \mathrm d\omega\, f_N(t,\omega).
\]
The first term is bounded by
\[
\int_0^N \mathrm d\omega\, f_N(t,\omega)
=
\int_{[0,\infty)} \mathrm d\omega\, h_N(t,\omega)
\le M_0.
\]
We now bound the second term. Since
\[
\omega U(\omega)=\omega^{\rho+1}(1+\omega)^{-\beta},
\]
there exists \(C>0\) such that, for every \(\omega\ge1\), we have
\[
1\le \frac{C}{\omega^{\rho+1-\beta}}\,\omega U(\omega).
\]
Therefore, for \(N\ge1\), we estimate by Step 1
\[
\int_N^\infty \mathrm d\omega\, f_N(t,\omega)
\le
\frac{C}{N^{\rho+1-\beta}}
\int_N^\infty \mathrm d\omega\, \omega U(\omega)f_N(t,\omega)
\le
\frac{C E_3}{N^{\rho+1-\beta}}
\le
C E_3.
\]
Hence
\[
\sup_{N\ge1}
\int_{[0,\infty)} \mathrm d\omega\, f_N(t,\omega)
\le M_0+C E_3.
\]

Moreover, it is tight, because for every \(R\ge 1\),
\[
\int_R^\infty \mathrm d\omega\, f_N(t,\omega)
\le
\frac{C}{R^{\rho+1-\beta}}
\int_R^\infty \mathrm d\omega\, \omega U(\omega) f_N(t,\omega)
\le
\frac{CE_3}{R^{\rho+1-\beta}},
\]
for some universal constant \(C>0\).
Hence, by the Prokhorov theorem, after extraction there exists a nonnegative Radon measure \(\mu_t\in\mathcal M_+([0,\infty))\) such that
\[
f_N(t,\omega)\,\mathrm d\omega \rightharpoonup \mu_t
\qquad\text{narrowly in }\mathcal M([0,\infty)) \qquad\text{as }N\to\infty,
\]
which means, after extraction,
\[
\lim_{N\to\infty}
\int_{[0,\infty)} \mathrm d\omega\, f_N(t,\omega)\phi(\omega)
=
\int_{[0,\infty)} \phi(\omega)\,\mu_t(\mathrm d\omega),
\]
for all \(\phi\in C_b([0,\infty))\).

Similar with the previous case, for every fixed \(t\ge0\), every limit point in the narrow topology  of
\[
\bigl\{f_N(t,\omega)\,\mathrm d\omega\bigr\}_{N\ge1}
\]
in \(\mathcal M([0,\infty))\) has the form
\[
\alpha(t)\delta_{\omega=0}+f^{\mathrm{reg}}(t,\omega)\,\mathrm d\omega
\]
for some \(\alpha(t)\ge0\), possibly depending on the chosen cluster point.

We now identify the convergence of \(g_N\). We define
\[
g(t,\omega):=\omega^\rho(1+\omega)^{-\beta} f^{\mathrm{reg}}(t,\omega).
\]
We fix \(0<b<\infty\) and \(\psi\in L^\infty((0,b))\),  let \(\varepsilon>0\) and choose \(0<\delta<b\) so small that
\[
2\delta^\rho(1+\delta)^{-\beta} M_0\|\psi\|_{L^\infty}<\varepsilon.
\]
Then, we have
\[
\left|
\int_0^b \mathrm d\omega\, \psi(\omega)\bigl(g_N(t,\omega)-g(t,\omega)\bigr)
\right|
\le
I_{N,\delta}(t)+J_{N,\delta}(t),
\]
where
\[
I_{N,\delta}(t)
:=
\left|
\int_0^\delta \mathrm d\omega\, \psi(\omega)\bigl(g_N(t,\omega)-g(t,\omega)\bigr)
\right|,
\]
and
\[
J_{N,\delta}(t)
:=
\left|
\int_\delta^b \mathrm d\omega\, \psi(\omega)\bigl(g_N(t,\omega)-g(t,\omega)\bigr)
\right|.
\]

For the first term, using \(g_N=\omega^\rho(1+\omega)^{-\beta} h_N\), we obtain
\[
\int_0^\delta \mathrm d\omega\, g_N(t,\omega)
=
\int_0^\delta \mathrm d\omega\, \omega^\rho(1+\omega)^{-\beta} h_N(t,\omega)
\le
\delta^\rho(1+\delta)^{-\beta}\int_0^\delta \mathrm d\omega\, h_N(t,\omega)
\le
\delta^\rho(1+\delta)^{-\beta} M_0.
\]
Moreover,  we also have
\[
\int_0^\delta \mathrm d\omega\, g(t,\omega)\le \delta^\rho(1+\delta)^{-\beta} M_0.
\]
Therefore
\[
I_{N,\delta}(t)\le 2\delta^\rho(1+\delta)^{-\beta} M_0\|\psi\|_{L^\infty}<\varepsilon
\qquad\text{for all }N,\ t\in[0,T].
\]

For the second term, since \(\omega^\rho(1+\omega)^{-\beta}\psi(\omega)\in L^\infty((\delta,b))\), the already proved uniform convergence of \(h_N\) on \((\delta,b)\) gives
\[
J_{N,\delta}(t)\to0
\qquad\text{uniformly in }t\in[0,T] \qquad\text{as }N\to\infty.
\]
Hence
\[
\sup_{t\in[0,T]}
\left|
\int_0^b \mathrm d\omega\, \psi(\omega)\,g_N(t,\omega)
-
\int_0^b \mathrm d\omega\, \psi(\omega)\,g(t,\omega)
\right|\to0
\qquad\text{as }N\to\infty.
\]

Since the time \(T>0\) can be chosen arbitrarily, a diagonal extraction argument yields
\[
f^{\mathrm{reg}}\in C\bigl([0,\infty);w\text{-}L^1_{\mathrm loc}((0,\infty))\bigr),
\]
and all the preceding convergence properties hold on every interval \([0,T]\).

\medskip
\noindent
\textit{Step 6: Passage to the limit in the first three terms.}

We fix \(\varphi\in C_c^2((0,\infty))\), and choose \(0<a<b<\infty\) such that
\[
\operatorname{supp}\varphi\subset(a,b).
\]
We write
\[
\mathcal H_\varphi
=
\mathcal H_\varphi^{(1)}
+
\mathcal H_\varphi^{(2)}
-
\mathcal H_\varphi^{(3)}
+
\mathcal H_\varphi^{(4)},
\]
where
\[
\mathcal H_\varphi^{(1)}(x,y)
:=
(x+y)^\rho(1+x+y)^{-\beta}\varphi(x+y),
\]
\[
\mathcal H_\varphi^{(2)}(x,y)
:=
(x-y)^\rho(1+x-y)^{-\beta}\varphi(x-y),
\]
\[
\mathcal H_\varphi^{(3)}(x,y)
:=
\Bigl((x-y)^\rho(1+x-y)^{-\beta}
+
(x+y)^\rho(1+x+y)^{-\beta}\Bigr)\varphi(x),
\]
and
\[
\mathcal H_\varphi^{(4)}(x,y)
:=
\Bigl((x-y)^\rho(1+x-y)^{-\beta}
-
(x+y)^\rho(1+x+y)^{-\beta}\Bigr)\varphi(y).
\]

We first treat the terms \(\mathcal H_\varphi^{(1)}\), \(\mathcal H_\varphi^{(2)}\), and
\(\mathcal H_\varphi^{(3)}\). Throughout this step, we fix \(t\ge0\).

\smallskip
\noindent
\textit{The term \(\mathcal H_\varphi^{(1)}\).}

We define
\[
I_N^{(1)}(t)
:=
\int\!\!\!\int_{x>y>0} \mathrm dx\,\mathrm dy\,
g_N(t,x)g_N(t,y)(x+y)^\rho(1+x+y)^{-\beta}\varphi(x+y).
\]
Since \(\varphi(x+y)\neq0\) implies \(x+y\in(a,b)\), we have \(0<x<b\), \(0<y<b\), and for fixed
\(x\in(0,b)\) the conditions
\[
x>y>0,
\qquad
x+y<b
\]
imply
\[
0<y<\min\{x,b-x\}.
\]
Thus, we can write
\[
I_N^{(1)}(t)
=
\int_0^b \mathrm dx\, g_N(t,x)K_N^{(1)}(t,x),
\]
where
\[
K_N^{(1)}(t,x)
:=
\int_0^{\min\{x,b-x\}} \mathrm dy\,
g_N(t,y)(x+y)^\rho(1+x+y)^{-\beta}\varphi(x+y).
\]

For each \(x\in[0,b]\), we define
\[
\Psi_x(y)
:=
\mathbf 1_{(0,\min\{x,b-x\})}(y)
(x+y)^\rho(1+x+y)^{-\beta}\varphi(x+y),
\qquad y\in(0,b).
\]
Then \(\Psi_x\in L^\infty((0,b))\), and we have
\[
K_N^{(1)}(t,x)
=
\int_0^b \mathrm dy\, g_N(t,y)\Psi_x(y).
\]
Since
\[
g_N(t,\cdot)\rightharpoonup g(t,\cdot)
\qquad\text{weakly in }L^1((0,b)) \qquad\text{as }N\to\infty,
\]
we obtain, for every fixed \(x\in[0,b]\),
\[
K_N^{(1)}(t,x)\longrightarrow
K^{(1)}(t,x)
:=
\int_0^b \mathrm dy\, g(t,y)\Psi_x(y)
\qquad\text{as }N\to\infty.
\]

We next prove that the convergence is uniform in \(x\in[0,b]\). First, for every \(x\in[0,b]\), we bound
\[
|K_N^{(1)}(t,x)|
\le
\sup_{z\in[a,b]} z^\rho(1+z)^{-\beta}|\varphi(z)|
\int_0^b \mathrm dy\, g_N(t,y)
\le
\sup_{z\in[a,b]} z^\rho(1+z)^{-\beta}|\varphi(z)|\,M_2.
\]
Next, we let \(x,x'\in[0,b]\) and write
\[
\begin{aligned}
	\Psi_x(y)-\Psi_{x'}(y)
	&=
	\mathbf 1_{(0,\min\{x,b-x\})}(y)
	\Bigl[
	(x+y)^\rho(1+x+y)^{-\beta}\varphi(x+y)
	\\
	&\qquad\qquad\qquad
	-
	(x'+y)^\rho(1+x'+y)^{-\beta}\varphi(x'+y)
	\Bigr]
	\\
	&\quad
	+
	\Bigl(
	\mathbf 1_{(0,\min\{x,b-x\})}(y)
	-
	\mathbf 1_{(0,\min\{x',b-x'\})}(y)
	\Bigr)
	\\
	&\qquad\qquad\qquad
	\times
	(x'+y)^\rho(1+x'+y)^{-\beta}\varphi(x'+y).
\end{aligned}
\]
Thus, we can estimate
\[
\begin{aligned}
	|K_N^{(1)}(t,x)-K_N^{(1)}(t,x')|
	&\le
	\int_0^b \mathrm dy\, g_N(t,y)
	\bigl|
	(x+y)^\rho(1+x+y)^{-\beta}\varphi(x+y)
	\\
	&\qquad\qquad
	-
	(x'+y)^\rho(1+x'+y)^{-\beta}\varphi(x'+y)
	\bigr|
	\\
	&\quad
	+
	\int_0^b \mathrm dy\, g_N(t,y)
	\left|
	\mathbf 1_{(0,\min\{x,b-x\})}(y)
	-
	\mathbf 1_{(0,\min\{x',b-x'\})}(y)
	\right|
	\\
	&\qquad\qquad
	\times
	|(x'+y)^\rho(1+x'+y)^{-\beta}\varphi(x'+y)|.
\end{aligned}
\]

The map
\[
(x,y)\longmapsto (x+y)^\rho(1+x+y)^{-\beta}\varphi(x+y)
\]
is uniformly continuous on \([0,b]\times[0,b]\). Hence, given \(\varepsilon>0\), there exists
\(\delta_1>0\) such that, whenever \(|x-x'|<\delta_1\), we can bound
\[
\sup_{y\in[0,b]}
\bigl|
(x+y)^\rho(1+x+y)^{-\beta}\varphi(x+y)
-
(x'+y)^\rho(1+x'+y)^{-\beta}\varphi(x'+y)
\bigr|
<
\frac{\varepsilon}{2M_2}.
\]
Therefore, we obtain
\[
\int_0^b \mathrm dy\, g_N(t,y)
\bigl|
(x+y)^\rho(1+x+y)^{-\beta}\varphi(x+y)
-
(x'+y)^\rho(1+x'+y)^{-\beta}\varphi(x'+y)
\bigr|
<
\frac{\varepsilon}{2}.
\]

We now estimate the second term. To this end, we set
\[
m(x):=\min\{x,b-x\}.
\]
Then, we can estimate
\[
\left|
\mathbf 1_{(0,m(x))}(y)-\mathbf 1_{(0,m(x'))}(y)
\right|
\le
\mathbf 1_{(\min\{m(x),m(x')\},\,\max\{m(x),m(x')\})}(y),
\]
and
\[
|m(x)-m(x')|\le |x-x'|.
\]
Since
\[
|(x'+y)^\rho(1+x'+y)^{-\beta}\varphi(x'+y)|
\le
\sup_{z\in[a,b]} z^\rho(1+z)^{-\beta}|\varphi(z)|,
\]
and since the family
\[
\{g_N(t,\cdot):N\ge1,\ t\ge0\}
\]
is uniformly integrable on \((0,b)\), there exists \(\delta_2>0\) such that if
\(E\subset(0,b)\) is measurable and \(|E|<\delta_2\), then we find
\[
\sup_{N\ge1}\sup_{t\ge0}\int_E \mathrm dy\, g_N(t,y)
<
\frac{\varepsilon}{
	2\sup_{z\in[a,b]} z^\rho(1+z)^{-\beta}|\varphi(z)|}.
\]
Therefore, whenever \(|x-x'|<\delta_2\), we obtain
\[
\begin{aligned}
	&\int_0^b \mathrm dy\, g_N(t,y)
	\left|
	\mathbf 1_{(0,m(x))}(y)-\mathbf 1_{(0,m(x'))}(y)
	\right|
	|(x'+y)^\rho(1+x'+y)^{-\beta}\varphi(x'+y)|
	\\
	&\qquad <\frac{\varepsilon}{2}.
\end{aligned}
\]
Combining the two estimates, we conclude that
\[
|K_N^{(1)}(t,x)-K_N^{(1)}(t,x')|<\varepsilon
\]
whenever
\(
|x-x'|<\min\{\delta_1,\delta_2\},
\)
uniformly in \(N\) and \(t\). Thus the family
\[
\{K_N^{(1)}(t,\cdot)\}_{N\ge1}
\]
is uniformly bounded and uniformly equicontinuous on \([0,b]\). Since it converges pointwise to
\(K^{(1)}(t,\cdot)\), the Arzel\`a--Ascoli theorem implies that
\[
K_N^{(1)}(t,\cdot)\longrightarrow K^{(1)}(t,\cdot)
\qquad\text{uniformly on }[0,b].
\]

Therefore, we can estimate
\[
\begin{aligned}
	|I_N^{(1)}(t)-I^{(1)}(t)|
	&\le
	\left|
	\int_0^b \mathrm dx\, g_N(t,x)
	\bigl(K_N^{(1)}(t,x)-K^{(1)}(t,x)\bigr)
	\right|
	\\
	&\quad
	+
	\left|
	\int_0^b \mathrm dx\, \bigl(g_N(t,x)-g(t,x)\bigr)K^{(1)}(t,x)
	\right|.
\end{aligned}
\]
The first term tends to \(0\) by the uniform convergence of \(K_N^{(1)}\) and the bound
\[
\int_0^b \mathrm dx\,g_N(t,x)\le M_2.
\]
The second term  also tends to \(0\) as
\[
g_N(t,\cdot)\rightharpoonup g(t,\cdot)
\qquad\text{weakly in }L^1((0,b)),
\]
and \(K^{(1)}(t,\cdot)\in C([0,b])\subset L^\infty((0,b))\). Hence, we finally get
\[
I_N^{(1)}(t)\longrightarrow
I^{(1)}(t)
\qquad\text{for every }t\ge0 \qquad\text{as }N\to\infty.
\]

\smallskip
\noindent
\textit{The term \(\mathcal H_\varphi^{(2)}\).}

We now define
\[
I_N^{(2)}(t)
:=
\int\!\!\!\int_{x>y>0} \mathrm dx\,\mathrm dy\,
g_N(t,x)g_N(t,y)(x-y)^\rho(1+x-y)^{-\beta}\varphi(x-y).
\]
Writing \(x=y+z\), with \(z\in(a,b)\), we obtain
\[
I_N^{(2)}(t)
=
\int_0^\infty \mathrm dy\, g_N(t,y)K_N^{(2)}(t,y),
\]
where
\[
K_N^{(2)}(t,y)
:=
\int_a^b \mathrm dz\, g_N(t,y+z)z^\rho(1+z)^{-\beta}\varphi(z).
\]

We fix \(R>0\) and split
\[
I_N^{(2)}(t)
=
\int_0^R \mathrm dy\, g_N(t,y)K_N^{(2)}(t,y)
+
\int_R^\infty \mathrm dy\, g_N(t,y)K_N^{(2)}(t,y)
=:J_{N,R}^{(2)}(t)+T_{N,R}^{(2)}(t).
\]

For \(y\ge0\), we have
\[
\begin{aligned}
	|K_N^{(2)}(t,y)|
	&\le
	\sup_{z\in[a,b]} z^\rho(1+z)^{-\beta}|\varphi(z)|
	\int_a^b \mathrm dz\, g_N(t,y+z)
	\\
	&=
	\sup_{z\in[a,b]} z^\rho(1+z)^{-\beta}|\varphi(z)|
	\int_{y+a}^{y+b} \mathrm dx\, g_N(t,x).
\end{aligned}
\]
Using \eqref{eq:Ux2-g-firstmoment}, we obtain
\[
\int_{y+a}^{y+b} \mathrm dx\, g_N(t,x)
\le
\frac1{y+a}
\int_{y+a}^{y+b} \mathrm dx\, xg_N(t,x)
\le
\frac{E_3}{y+a}.
\]
Hence
\[
|K_N^{(2)}(t,y)|
\le
\frac{
	E_3\sup_{z\in[a,b]} z^\rho(1+z)^{-\beta}|\varphi(z)|
}{y+a}.
\]
It follows that
\[
|T_{N,R}^{(2)}(t)|
\le
\frac{
	E_3\sup_{z\in[a,b]} z^\rho(1+z)^{-\beta}|\varphi(z)|
}{R+a}
\int_R^\infty \mathrm dy\, g_N(t,y)
\le
\frac{
	E_3M_2\sup_{z\in[a,b]} z^\rho(1+z)^{-\beta}|\varphi(z)|
}{R+a}.
\]
This bound is uniform in \(N\) and \(t\), and tends to \(0\) as \(R\to\infty\).

For \(y\in[0,R]\), the variable \(y+z\) ranges in \([a,R+b]\). For each fixed
\(y\in[0,R]\), the function
\[
x\longmapsto (x-y)^\rho(1+x-y)^{-\beta}\varphi(x-y)
\]
belongs to \(L^\infty((0,R+b+1))\). Therefore, since
\[
g_N(t,\cdot)\rightharpoonup g(t,\cdot)
\qquad\text{weakly in }L^1((0,R+b+1)),
\]
we obtain
\[
K_N^{(2)}(t,y)\longrightarrow
K^{(2)}(t,y)
:=
\int_a^b \mathrm dz\, g(t,y+z)z^\rho(1+z)^{-\beta}\varphi(z)
\]
for every fixed \(y\in[0,R]\).

We next prove that this convergence is uniform in \(y\in[0,R]\). First, we bound
\[
|K_N^{(2)}(t,y)|
\le
\sup_{z\in[a,b]} z^\rho(1+z)^{-\beta}|\varphi(z)|
\int_0^{R+b+1} \mathrm dx\, g_N(t,x),
\]
so the family \(\{K_N^{(2)}(t,\cdot)\}_{N\ge1}\) is uniformly bounded on \([0,R]\). Moreover, for
\(y,y'\in[0,R]\),
\[
\begin{aligned}
	|K_N^{(2)}(t,y)-K_N^{(2)}(t,y')|
	&\le
	\int_0^{R+b+1} \mathrm dx\, g_N(t,x)
	\bigl|
	(x-y)^\rho(1+x-y)^{-\beta}\varphi(x-y)
	\\
	&\qquad\qquad
	-
	(x-y')^\rho(1+x-y')^{-\beta}\varphi(x-y')
	\bigr|.
\end{aligned}
\]
Here the expression
\[
(x-y)^\rho(1+x-y)^{-\beta}\varphi(x-y)
\]
is understood to be zero when \(x-y\notin(a,b)\). The corresponding function of
\((x,y)\) is continuous on the compact set \([0,R+b+1]\times[0,R]\), since
\(\varphi\in C_c^2((a,b))\). Hence it is uniformly continuous. Since
\[
\int_0^{R+b+1} \mathrm dx\, g_N(t,x)\le M_2,
\]
the family \(\{K_N^{(2)}(t,\cdot)\}_{N\ge1}\) is uniformly equicontinuous on \([0,R]\).
By the Arzel\`a--Ascoli theorem and uniqueness of the pointwise limit, we find
\[
K_N^{(2)}(t,\cdot)\longrightarrow K^{(2)}(t,\cdot)
\qquad\text{uniformly on }[0,R] \qquad\text{as }N\to\infty.
\]

Consequently, we can bound
\[
\begin{aligned}
	|J_{N,R}^{(2)}(t)-J_R^{(2)}(t)|
	&\le
	\left|
	\int_0^R \mathrm dy\, g_N(t,y)
	\bigl(K_N^{(2)}(t,y)-K^{(2)}(t,y)\bigr)
	\right|
	\\
	&\quad
	+
	\left|
	\int_0^R \mathrm dy\, \bigl(g_N(t,y)-g(t,y)\bigr)K^{(2)}(t,y)
	\right|,
\end{aligned}
\]
where
\[
J_R^{(2)}(t)
:=
\int_0^R \mathrm dy\, g(t,y)K^{(2)}(t,y).
\]
The first term tends to \(0\) by the uniform convergence of \(K_N^{(2)}\) on \([0,R]\) and the
mass bound for \(g_N\). The second term tends to \(0\) by the weak convergence of \(g_N\) in
\(L^1((0,R))\), since \(K^{(2)}(t,\cdot)\in C([0,R])\). Hence, 
\[
J_{N,R}^{(2)}(t)\longrightarrow J_R^{(2)}(t)
\qquad\text{as }N\to\infty.
\]

Moreover, the same tail estimate as above, together with the bound
\[
\int_{[0,\infty)} \mathrm d\omega\, g(t,\omega)\le M_2,
\]
gives
\[
|T_R^{(2)}(t)|
:=
\left|
\int_R^\infty \mathrm dy\, g(t,y)K^{(2)}(t,y)
\right|
\le
\frac{
	E_3M_2\sup_{z\in[a,b]} z^\rho(1+z)^{-\beta}|\varphi(z)|
}{R+a}.
\]
Therefore, we get
\[
\limsup_{N\to\infty}|I_N^{(2)}(t)-I^{(2)}(t)|
\le
\frac{
	2E_3M_2\sup_{z\in[a,b]} z^\rho(1+z)^{-\beta}|\varphi(z)|
}{R+a}.
\]
Letting \(R\to\infty\), we obtain
\[
I_N^{(2)}(t)\longrightarrow I^{(2)}(t)
\qquad\text{for every }t\ge0.
\]

\smallskip
\noindent
\textit{The term \(\mathcal H_\varphi^{(3)}\).}

We now define
\[
I_N^{(3)}(t)
:=
\int\!\!\!\int_{x>y>0} \mathrm dx\,\mathrm dy\,
g_N(t,x)g_N(t,y)
\Bigl((x-y)^\rho(1+x-y)^{-\beta}
+
(x+y)^\rho(1+x+y)^{-\beta}\Bigr)\varphi(x).
\]
Since \(\varphi(x)\neq0\) implies \(x\in(a,b)\), we have \(0<y<x<b\). Hence, we have
\[
I_N^{(3)}(t)
=
\int_a^b \mathrm dx\, g_N(t,x)K_N^{(3)}(t,x),
\]
where
\[
K_N^{(3)}(t,x)
:=
\varphi(x)
\int_0^x \mathrm dy\,
\Bigl((x-y)^\rho(1+x-y)^{-\beta}
+
(x+y)^\rho(1+x+y)^{-\beta}\Bigr)g_N(t,y).
\]

For each fixed \(x\in[a,b]\), the function
\[
y\longmapsto
\varphi(x)
\Bigl((x-y)^\rho(1+x-y)^{-\beta}
+
(x+y)^\rho(1+x+y)^{-\beta}\Bigr)\mathbf 1_{(0,x)}(y)
\]
belongs to \(L^\infty((0,b))\). Therefore, we obtain the convergence
\[
K_N^{(3)}(t,x)\longrightarrow
K^{(3)}(t,x)
:=
\varphi(x)
\int_0^x \mathrm dy\,
\Bigl((x-y)^\rho(1+x-y)^{-\beta}
+
(x+y)^\rho(1+x+y)^{-\beta}\Bigr)g(t,y)
\]
for every fixed \(x\in[a,b]\), as \(N\to\infty\).

We now show that the convergence is uniform in \(x\in[a,b]\). Since
\(
0<y<x<b,
\)
there exists a constant \(C_{a,b}>0\) such that
\[
\left|
\varphi(x)
\Bigl((x-y)^\rho(1+x-y)^{-\beta}
+
(x+y)^\rho(1+x+y)^{-\beta}\Bigr)
\right|
\le C_{a,b}\|\varphi\|_{L^\infty}.
\]
Hence, we can bound
\[
|K_N^{(3)}(t,x)|
\le
C_{a,b}\|\varphi\|_{L^\infty}\int_0^b \mathrm dy\, g_N(t,y)
\le
C_{a,b}\|\varphi\|_{L^\infty}M_2.
\]

Let \(x,x'\in[a,b]\), and assume for simplicity that \(x<x'\). We write
\[
\begin{aligned}
	K_N^{(3)}(t,x')-K_N^{(3)}(t,x)
	&=
	\int_0^b \mathrm dy\, g_N(t,y)
	\Bigl[
	\varphi(x')
	\Bigl((x'-y)^\rho(1+x'-y)^{-\beta}
	\\
	&\qquad\qquad
	+
	(x'+y)^\rho(1+x'+y)^{-\beta}\Bigr)\mathbf 1_{(0,x')}(y)
	\\
	&\qquad
	-
	\varphi(x)
	\Bigl((x-y)^\rho(1+x-y)^{-\beta}
	+
	(x+y)^\rho(1+x+y)^{-\beta}\Bigr)\mathbf 1_{(0,x)}(y)
	\Bigr].
\end{aligned}
\]
The   map
\[
(x,y)\longmapsto
\varphi(x)
\Bigl((x-y)^\rho(1+x-y)^{-\beta}
+
(x+y)^\rho(1+x+y)^{-\beta}\Bigr)
\]
is uniformly continuous on
\[
\{(x,y):a\le x\le b,\ 0\le y\le x\}.
\]
 Hence  using
\[
\int_0^b \mathrm dy\, g_N(t,y)\le M_2.
\]
 and the fact that
 the family \(\{g_N(t,\cdot):N\ge1,\ t\ge0\}\) is uniformly integrable on \((0,b)\), we conclude that the family
\(
\{K_N^{(3)}(t,\cdot)\}_{N\ge1}
\)
is uniformly bounded and uniformly equicontinuous on \([a,b]\). Since it converges pointwise to
\(K^{(3)}(t,\cdot)\), using the Arzel\`a--Ascoli theorem, we obtain the convergence
\[
K_N^{(3)}(t,\cdot)\longrightarrow K^{(3)}(t,\cdot)
\qquad\text{uniformly on }[a,b] \qquad\text{as }N\to\infty.
\]

Therefore, we can bound
\[
\begin{aligned}
	|I_N^{(3)}(t)-I^{(3)}(t)|
	&\le
	\left|
	\int_a^b \mathrm dx\, g_N(t,x)
	\bigl(K_N^{(3)}(t,x)-K^{(3)}(t,x)\bigr)
	\right|
	\\
	&\quad
	+
	\left|
	\int_a^b \mathrm dx\, \bigl(g_N(t,x)-g(t,x)\bigr)K^{(3)}(t,x)
	\right|.
\end{aligned}
\]
The first term tends to \(0\) by the uniform convergence of \(K_N^{(3)}\) and the bound
\[
\int_a^b \mathrm dx\, g_N(t,x)\le M_2.
\]
The second term tends to \(0\) by the weak convergence of \(g_N(t,\cdot)\) to \(g(t,\cdot)\) in
\(L^1((a,b))\), since \(K^{(3)}(t,\cdot)\in C([a,b])\). Hence, we find the convergence
\[
I_N^{(3)}(t)\longrightarrow I^{(3)}(t)
\qquad\text{for every }t\ge0 \qquad\text{as }N\to\infty.
\]

Combining the preceding three convergences, we obtain
\[
I_N^{(j)}(t)\longrightarrow I^{(j)}(t),
\qquad j=1,2,3,
\qquad\text{for every }t\ge0.
\]

	\medskip
	\noindent
	\textit{Step 7: Passage to the limit in the term \(\mathcal H_\varphi^{(4)}\).}

We define
\[
I_N^{(4)}(t)
:=
\int\!\!\!\int_{x>y>0} \mathrm dx\,\mathrm dy\,
g_N(t,x)g_N(t,y)
\Bigl((x-y)^\rho(1+x-y)^{-\beta}-(x+y)^\rho(1+x+y)^{-\beta}\Bigr)\varphi(y).
\]
Since \(\varphi(y)\neq0\) implies \(y\in(a,b)\), we may write
\[
I_N^{(4)}(t)
=
\int_a^b \mathrm dy\, \varphi(y)g_N(t,y)
\int_y^\infty \mathrm dx\, g_N(t,x)
\Bigl((x-y)^\rho(1+x-y)^{-\beta}-(x+y)^\rho(1+x+y)^{-\beta}\Bigr).
\]

For \(y\in[a,b]\), we define
\[
L_N(t,y)
:=
\varphi(y)\int_y^\infty \mathrm dx\, g_N(t,x)
\Bigl((x-y)^\rho(1+x-y)^{-\beta}-(x+y)^\rho(1+x+y)^{-\beta}\Bigr).
\]
Then
\[
I_N^{(4)}(t)=\int_a^b \mathrm dy\, g_N(t,y)L_N(t,y).
\]

We first show that \(L_N(t,y)\) is well defined and uniformly bounded. To this end, we define
\[
F(z):=z^\rho(1+z)^{-\beta}, \qquad z\ge0,
\]
and proceed as in \eqref{F}, we conclude that for every \(x>y>0\),
\[
|F(x-y)-F(x+y)|
\le
2y\sup_{\xi\in(x-y,x+y)}|F'(\xi)|
\le
C_{a,b}(1+x)^{\rho-\beta-1},
\]
for some constant \(C_{a,b}>0\) as \(y\in(a,b)\). Therefore, we can estimate
\[
\begin{aligned}
	|L_N(t,y)|
	&\le
	|\varphi(y)|\int_y^\infty \mathrm dx\, g_N(t,x)\,|F(x-y)-F(x+y)|
	\\
	&\le
	C_{a,b}\|\varphi\|_{L^\infty}
	\int_y^\infty \mathrm dx\, (1+x)^{\rho-\beta-1}g_N(t,x).
\end{aligned}
\]
Now, since \(x\ge y\ge a\), we have \((1+x)^{\rho-\beta-1}\le (1+\tfrac1a)x\). Thus, we can bound
\[
\int_y^\infty \mathrm dx\, (1+x)^{\rho-\beta-1}g_N(t,x)
\le
\left(1+\frac1a\right)\int_0^\infty \mathrm dx\, xg_N(t,x).
\]

Hence, we obtain
\[
\int_0^\infty \mathrm dx\, (1+x)^{\rho-\beta-1}g_N(t,x)\le C_{a,b}(M_0+E_3).
\]
Therefore
\[
|L_N(t,y)|\le C_{a,b,\varphi}
\qquad\text{for all }N,\ t\ge0,\ y\in[a,b].
\]

We now fix \(t\in[0,\infty)\) and prove that
\[
L_N(t,y)\longrightarrow L(t,y) \qquad\text{ as } N\to\infty
\qquad\text{for every }y\in[a,b], 
\]
where
\[
L(t,y)
:=
\varphi(y)\int_y^\infty \mathrm dx\, g(t,x)
\Bigl((x-y)^\rho(1+x-y)^{-\beta}-(x+y)^\rho(1+x+y)^{-\beta}\Bigr).
\]

To this end, we fix \(y\in[a,b]\) and \(R>b+1\). We split
\[
L_N(t,y)=L_{N,R}^{(4)}(t,y)+T_{N,R}^{(4)}(t,y),
\]
where
\[
L_{N,R}^{(4)}(t,y)
:=
\varphi(y)\int_y^R \mathrm dx\, g_N(t,x)
\Bigl((x-y)^\rho(1+x-y)^{-\beta}-(x+y)^\rho(1+x+y)^{-\beta}\Bigr),
\]
and
\[
T_{N,R}^{(4)}(t,y)
:=
\varphi(y)\int_R^\infty \mathrm dx\, g_N(t,x)
\Bigl((x-y)^\rho(1+x-y)^{-\beta}-(x+y)^\rho(1+x+y)^{-\beta}\Bigr).
\]

We first estimate the second term.  Using the estimate above, we have
\[
|T_{N,R}^{(4)}(t,y)|
\le
C_{a,b,\varphi}\int_R^\infty \mathrm dx\, (1+x)^{\rho-\beta-1}g_N(t,x).
\]
Since \(x\ge R\ge1\), we have, for some constant \(C>0\)
\[
(1+x)^{\rho-\beta-1}\le Cx^{\rho-\beta-1}.
\]
Moreover, because \(\rho-\beta<2\), it follows that \(\rho-\beta-1<1\), and therefore
\[
x^{\rho-\beta-1}\le R^{\rho-\beta-2}x
\qquad\text{for all }x\ge R.
\]
Hence, we can bound
\[
(1+x)^{\rho-\beta-1}\le C R^{\rho-\beta-2}x
\qquad\text{for all }x\ge R.
\]
Substituting this into the previous estimate, we obtain
\[
|T_{N,R}^{(4)}(t,y)|
\le
C_{a,b,\varphi}R^{\rho-\beta-2}
\int_R^\infty \mathrm dx\, xg_N(t,x).
\]
Using \eqref{eq:Ux2-g-firstmoment}, we infer
\[
\int_R^\infty \mathrm dx\, xg_N(t,x)
\le
\int_0^\infty \mathrm dx\, xg_N(t,x)
\le
E_3.
\]
Therefore, we obtain the estimate
\[
|T_{N,R}^{(4)}(t,y)|
\le
C_{a,b,\varphi}R^{\rho-\beta-2}E_3,
\]
uniformly in \(N\), \(t\), and \(y\in[a,b]\). Since \(\rho-\beta<2\), this tends to \(0\) as \(R\to\infty\).

On the other hand, for fixed \(y\in[a,b]\), the function
\[
x\longmapsto
\varphi(y)\Bigl((x-y)^\rho(1+x-y)^{-\beta}-(x+y)^\rho(1+x+y)^{-\beta}\Bigr)\chi_{(y,R)}(x)
\]
belongs to \(L^\infty((0,R+1))\). Since
\[
g_N(t,\cdot)\rightharpoonup g(t,\cdot)
\qquad\text{weakly in }L^1((0,R+1))
\qquad\text{as }N\to\infty,
\]
we obtain
\[
L_{N,R}^{(4)}(t,y)\longrightarrow
L_R^{(4)}(t,y)
:=
\varphi(y)\int_y^R \mathrm dx\, g(t,x)
\Bigl((x-y)^\rho(1+x-y)^{-\beta}-(x+y)^\rho(1+x+y)^{-\beta}\Bigr)
\]
as \(N\to\infty\). Letting then \(R\to\infty\), we conclude that
\[
L_N(t,y)\longrightarrow L(t,y)
\qquad\text{for every }y\in[a,b].
\]

We next prove that the convergence is uniform in \(y\in[a,b]\). First, we already know that the family \(\{L_N(t,\cdot)\}_N\) is uniformly bounded on \([a,b]\). We now show that it is uniformly equicontinuous. We let \(a\le y<y'\le b\) and then
\[
\begin{aligned}
	|L_N(t,y')-L_N(t,y)|
	&\le
	|\varphi(y')-\varphi(y)|
	\int_{y'}^\infty \mathrm dx\, g_N(t,x)\,
	|F(x-y')-F(x+y')|
	\\
	&\quad
	+
	|\varphi(y)|
	\left|
	\int_{y'}^\infty \mathrm dx\, g_N(t,x)\bigl(F(x-y')-F(x+y')\bigr)
	\right.
	\\
	&\qquad\qquad
	\left.
	-
	\int_y^\infty \mathrm dx\, g_N(t,x)\bigl(F(x-y)-F(x+y)\bigr)
	\right|.
\end{aligned}
\]
For the first term on the right-hand side, using again
\[
|F(x-y)-F(x+y)|\le C_{a,b}(1+x),
\]
we obtain
\[
|\varphi(y')-\varphi(y)|
\int_{y'}^\infty \mathrm dx\, g_N(t,x)\,
|F(x-y')-F(x+y')|
\le
C_{a,b}|\varphi(y')-\varphi(y)|
\int_0^\infty \mathrm dx\, (1+x)g_N(t,x),
\]
which is uniformly small when \(|y'-y|\) is small, by the continuity of the function \(\varphi\).

Now, we consider the second term. We write it as the sum of
\[
\int_{y'}^\infty \mathrm dx\, g_N(t,x)
\left|
\bigl(F(x-y')-F(x+y')\bigr)-\bigl(F(x-y)-F(x+y)\bigr)
\right|
\]
and
\[
\int_y^{y'} \mathrm dx\, g_N(t,x)\,|F(x-y)-F(x+y)|.
\]
Since the function
\[
(y,x)\longmapsto F(x-y)-F(x+y)
\]
is continuous on the compact set
\[
\{(y,x)\in[a,b]\times[a,\infty):\,x\in[y,R]\}
\]
for every fixed \(R>b\), the first integral over \((y',R)\) is uniformly small when \(|y'-y|\) is small, while the tail over \((R,\infty)\) is uniformly small by the estimate above. For the second integral, since \(x\in(y,y')\subset(a,b)\), the factor \(|F(x-y)-F(x+y)|\) is uniformly bounded on the corresponding compact region, and therefore
\[
\int_y^{y'} \mathrm dx\, g_N(t,x)\,|F(x-y)-F(x+y)|
\le
C_{a,b}\int_y^{y'} \mathrm dx\, g_N(t,x).
\]
Since \(\{g_N(t,\cdot)\}_N\) is uniformly integrable on \((a,b)\), the right-hand side of the above inequality is uniformly small when \(|y'-y|\) is small. Hence the family \(\{L_N(t,\cdot)\}_N\) is uniformly equicontinuous on \([a,b]\).

Therefore, by the Arzel\`a--Ascoli theorem and uniqueness of the pointwise limit, we conclude
\[
L_N(t,\cdot)\longrightarrow L(t,\cdot)
\qquad\text{uniformly on }[a,b]
\qquad\text{as }N\to\infty.
\]
Consequently, we obtain
\[
I_N^{(4)}(t)=\int_a^b \mathrm dy\, g_N(t,y)L_N(t,y)
\longrightarrow
\int_a^b \mathrm dy\, g(t,y)L(t,y)
=:I^{(4)}(t)
\qquad\text{as }N\to\infty
\]
for every \(t\ge0\).
	
	\medskip
	\noindent
	\textit{Step 8: Passage to the limit.}

We fix \(T>0\). Combining Steps~6 and~7, we obtain for every \(t\in[0,T]\),
\[
\int\!\!\!\int_{x>y>0} \mathrm dx\,\mathrm dy\,
g_N(t,x)g_N(t,y)\mathcal H_\varphi(x,y)
\longrightarrow
\int\!\!\!\int_{x>y>0} \mathrm dx\,\mathrm dy\,
g(t,x)g(t,y)\mathcal H_\varphi(x,y)
\qquad\text{as }N\to\infty.
\]
Moreover, Step~4 yields the uniform bound
\[
\left|
\int\!\!\!\int_{x>y>0} \mathrm dx\,\mathrm dy\,
g_N(t,x)g_N(t,y)\mathcal H_\varphi(x,y)
\right|
\le
C_\varphi
\qquad\text{for all }N\ge 1,\ t\in[0,T].
\]
Hence, by the dominated convergence theorem in time, we obtain the convergence
\[
\int_0^t \mathrm ds
\int\!\!\!\int_{x>y>0} \mathrm dx\,\mathrm dy\,
g_N(s,x)g_N(s,y)\mathcal H_\varphi(x,y)
\longrightarrow
\int_0^t \mathrm ds
\int\!\!\!\int_{x>y>0} \mathrm dx\,\mathrm dy\,
g(s,x)g(s,y)\mathcal H_\varphi(x,y),
\]
as \(N\to\infty\).
Repeating the arguments used in the proof of Proposition \ref{Prop:Existence:U=x:Hphi},  
passing to the limit in \eqref{eq:weak-Hphi-N-Ux2}, we obtain
\[
\int_0^\infty \mathrm d\omega\, g(t,\omega)\varphi(\omega)
=
\int_0^\infty \mathrm d\omega\, g_0(\omega)\varphi(\omega)
+
2\int_0^t \mathrm ds
\int\!\!\!\int_{x>y>0} \mathrm dx\,\mathrm dy\,
g(s,x)g(s,y)\mathcal H_\varphi(x,y),
\]
for every \(t\in[0,T]\). Since \(T>0\) is arbitrary, this holds for every \(t\ge0\).  It is therefore a weak form of \eqref{3wave} for
\[
f(t,\omega)=f^{\mathrm{reg}}(t,\omega).
\]

	\medskip
	\noindent
	\textit{Step 9: Entropy bound for the regular part.}

Fix \(t\ge0\). For every \(N\), we always have
\begin{equation}\label{eq:Ux2-entropyN}
	\int_0^\infty \mathrm d\omega\,
	\frac{e\!\bigl(\omega^\rho(1+\omega)^{-\beta} h_N(t,\omega)\bigr)}{\omega^\rho(1+\omega)^{-\beta}}
	\le
	\int_0^\infty \mathrm d\omega\,
	\frac{e\!\bigl(\omega^\rho(1+\omega)^{-\beta} f_0(\omega)\bigr)}{\omega^\rho(1+\omega)^{-\beta}}.
\end{equation}
Let \(0<a<b<\infty\). We consider the functional
\[
\mathcal E_{a,b}(u)
:=
\int_a^b \mathrm d\omega\,
\frac{e\!\bigl(\omega^\rho(1+\omega)^{-\beta}u(\omega)\bigr)}{\omega^\rho(1+\omega)^{-\beta}},
\qquad
u\in L^1((a,b)),\ u\ge0.
\]
Repeating the arguments used in the proof of Proposition \ref{Prop:Existence:U=x:Hphi}, we obtain
\[
\int_a^b \mathrm d\omega\,
\frac{e\!\bigl(\omega^\rho(1+\omega)^{-\beta} f^{\mathrm{reg}}(t,\omega)\bigr)}{\omega^\rho(1+\omega)^{-\beta}}
\le
\int_0^\infty \mathrm d\omega\,
\frac{e\!\bigl(\omega^\rho(1+\omega)^{-\beta} f_0(\omega)\bigr)}{\omega^\rho(1+\omega)^{-\beta}}.
\]
Letting \(a\) converge to \(0\) and \(b\) converge to \(\infty\), and using the monotone convergence theorem, we conclude that
\[
\int_0^\infty \mathrm d\omega\,
\frac{e\!\bigl(\omega^\rho(1+\omega)^{-\beta} f^{\mathrm{reg}}(t,\omega)\bigr)}{\omega^\rho(1+\omega)^{-\beta}}
\le
\int_0^\infty \mathrm d\omega\,
\frac{e\!\bigl(\omega^\rho(1+\omega)^{-\beta} f_0(\omega)\bigr)}{\omega^\rho(1+\omega)^{-\beta}}.
\]

The proof of \eqref{eq:dissipation-assumption:3} is the same with the case $U(\omega)=\omega$. \end{proof}
\section{Convergence to equilibrium}
\begin{proposition}
	\label{prop:equilibrium}
	Assume that \(U(\omega)=\omega\) or \(U(\omega)=\omega^\rho(1+\omega)^{-\beta}\), (\(\rho-\beta<2,
	\rho\ge1,
	\beta\le \rho-1\)) and let
	\[
	f\in C\bigl([0,\infty);w\text{-}L^1_{\mathrm{loc}}((0,\infty))\bigr)
	\]
	be a nonnegative global weak solution constructed in Subsections \ref{Sec:CaseA} and \ref{Sec:CaseB}.
		Then   	\[
	\lim_{t\to\infty}f(t,\cdot)= 0
	\qquad\text{in }L^1_{\mathrm{loc}}((0,\infty)).
	\]
\end{proposition}

\begin{proof}
	We divide the proof into three steps.
	
	\medskip
	\noindent
\medskip
\noindent
\textit{Step 1: Time-translated solutions.}

Let \(\{t_n\}_{n\ge1}\) be any sequence such that \(t_n\to\infty\). We fix \(T>0\), and define
\[
f_n(s,\omega):=f(t_n+s,\omega),
\qquad (s,\omega)\in[0,T]\times(0,\infty).
\]
We will show that each \(f_n\) satisfies on \([0,T]\) the same weak formulation as \(f\), with initial datum
\(
f_n(0,\omega)=f(t_n,\omega).
\)

We first consider the case \(U(\omega)=\omega\) and define
\[
g(t,\omega):=U(\omega)f(t,\omega),
\qquad
g_n(s,\omega):=U(\omega)f_n(s,\omega)=g(t_n+s,\omega).
\]

By the weak formulation already proved for \(f\), for every
\(\varphi\in C_c^2((0,\infty))\) and every \(r\ge0\), we have
\[
\int_0^\infty \mathrm{d}\omega\, g(r,\omega)\varphi(\omega)
=
\int_0^\infty \mathrm{d}\omega\, g_0(\omega)\varphi(\omega)
+
2\int_0^r \mathrm{d}\tau
\iint_{x>y>0}
\mathrm{d}x\,\mathrm{d}y\,
g(\tau,x)g(\tau,y)\,H_\varphi(x,y),
\]
where
\[
H_\varphi(x,y)
=
(x+y)\varphi(x+y)
+
(x-y)\varphi(x-y)
-
2x\,\varphi(x)
-
2y\,\varphi(y).
\]
Applying this identity with \(r=t_n+s\) and with \(r=t_n\), and subtracting the two equalities, we obtain
\[
\int_0^\infty \mathrm{d}\omega\, g(t_n+s,\omega)\varphi(\omega)
=
\int_0^\infty \mathrm{d}\omega\, g(t_n,\omega)\varphi(\omega)
+
2\int_{t_n}^{t_n+s} \mathrm{d}\tau
\iint_{x>y>0}
\mathrm{d}x\,\mathrm{d}y\,
g(\tau,x)g(\tau,y)\,H_\varphi(x,y).
\]
With the change of variable \(\tau=t_n+\sigma\), this identity becomes
\[
\int_0^\infty \mathrm{d}\omega\, g_n(s,\omega)\varphi(\omega)
=
\int_0^\infty \mathrm{d}\omega\, g_n(0,\omega)\varphi(\omega)
+
2\int_0^s \mathrm{d}\sigma
\iint_{x>y>0}
\mathrm{d}x\,\mathrm{d}y\,
g_n(\sigma,x)g_n(\sigma,y)\,H_\varphi(x,y).
\]
Thus \(f_n\) satisfies on \([0,T]\) the same weak formulation as \(f\).

We now consider the case
\[
U(\omega)=\omega^\rho(1+\omega)^{-\beta},
\qquad
\rho-\beta<2,
\qquad
\rho\ge1,
\qquad
\beta\le \rho-1.
\]
By the same argument as above, \(f_n\) also satisfies on \([0,T]\) the same weak formulation as \(f\).

First, we observe that the energy estimate for \(f\) is uniformly bounded for all times. More precisely, we have
\[
\sup_{t\ge0}\int_0^\infty \mathrm{d}\omega\,\omega U(\omega)f(t,\omega)<\infty .
\]
In particular, for every compact interval
\([a,b]\subset(0,\infty)\), we can bound
\[
\sup_{n\ge1}\sup_{s\in[0,T]}\int_a^b \mathrm{d}\omega\, f_n(s,\omega)<\infty.
\]

Next, the entropy inequality for \(f\) is uniform in time. Recalling that, for
\[
\mathcal E_e(f(t))
:=
\int_0^\infty \mathrm{d}\omega\, \frac{e(U(\omega)f(t,\omega))}{U(\omega)},
\]
we have
\[
\mathcal E_e(f(t))\le \mathcal E_e(f_0)
\qquad\text{for all } t\ge0.
\]
Hence, for every \(n\) and every \(s\in[0,T]\), we bound
\[
\int_0^\infty \mathrm{d}\omega\, \frac{e(U(\omega)f_n(s,\omega))}{U(\omega)}
=
\int_0^\infty \mathrm{d}\omega\, \frac{e(U(\omega)f(t_n+s,\omega))}{U(\omega)}
\le \mathcal E_e(f_0),
\]
which yields the uniform integrability of
\(\{f_n(s)\}_{n\ge1,\;s\in[0,T]}\) on every compact interval \((a,b)\subset(0,\infty)\).

We will now verify  equicontinuity. We let
\(\psi\in C_c^2((0,\infty))\). Then
\[
s\mapsto \int_0^\infty \mathrm{d}\omega\, \psi(\omega)f_n(s,\omega)
=
\int_0^\infty \mathrm{d}\omega\, \psi(\omega)f(t_n+s,\omega).
\]
By the same arguments as the ones used in Subsections \ref{Sec:CaseA} and \ref{Sec:CaseB}, we can conclude that there exists a constant \(C_\psi>0\), independent of \(n\), such that
\[
\left|
\frac{\mathrm d}{\mathrm ds}
\int_0^\infty \mathrm{d}\omega\, \psi(\omega)f_n(s,\omega)
\right|
\le C_\psi
\qquad\text{for all } s\in[0,T],\ n\ge1.
\]
Thus the family
\[
s\mapsto \int_0^\infty \mathrm{d}\omega\, \psi(\omega)f_n(s,\omega)
\]
is uniformly bounded and uniformly equicontinuous on \([0,T]\).

We have therefore verified for \(\{f_n\}\) the same hypotheses used in the compactness argument of the existence proofs done in Subsections \ref{Sec:CaseA} and \ref{Sec:CaseB}. Consequently, by the same argument used in Subsections \ref{Sec:CaseA} and \ref{Sec:CaseB}, after extraction of a subsequence there exists
\[
F\in C\bigl([0,T];w\text{-}L^1_{\mathrm{loc}}((0,\infty))\bigr)
\]
such that
\[
\lim_{n\to\infty}f_n=F
\quad\text{in } C\bigl([0,T];w\text{-}L^1_{\mathrm{loc}}((0,\infty))\bigr).
\]
In particular, we have 
\[
\lim_{n\to\infty}f(t_n,\cdot)
=
\lim_{n\to\infty}f_n(0,\cdot)
=
F(0,\cdot)
\qquad\text{in } w\text{-}L^1_{\mathrm{loc}}((0,\infty)).
\]
We then define
\(
f_\infty:=F(0).
\)
Thus every sequence \(t_n\to\infty\) admits a subsequence such that
\[
f(t_n)\rightharpoonup f_\infty
\qquad\text{in } w\text{-}L^1_{\mathrm{loc}}((0,\infty)).
\]

\medskip
\noindent
\textit{Step 2: Vanishing of the dissipation for every time-translation limit.}

We define
\[
\mathfrak W(s):=
\int_0^\infty \mathrm{d}\omega
\int_0^\omega \mathrm{d}\omega_1\,
U(\omega_1)U(\omega-\omega_1)\,f(s,\omega)\,f(s,\omega_1).
\]
Since \(f\ge0\) and \(U>0\) on \((0,\infty)\), we have
\[
\mathfrak W(s)\ge0
\qquad\text{for all } s\ge0.
\]
By \eqref{eq:dissipation-assumption}, applied to \(e(r)=r\), we find
\[
\int_0^\infty \mathrm{d}s\,\mathfrak W(s)<\infty.
\]
Hence \(\mathfrak W\in L^1(0,\infty)\). Therefore, for every fixed \(T>0\),
\[
\int_0^T \mathrm{d}s\,\mathfrak W(t_n+s)
\le
\int_{t_n}^{\infty} \mathrm{d}\tau\,\mathfrak W(\tau)
\longrightarrow 0
\qquad\text{as } n\to\infty.
\]

We now fix \(0<a<b<\infty\). Define
\[
K_{a,b}(\omega,\omega_1)
:=
\begin{cases}
	U(\omega_1)U(\omega-\omega_1),
	& a<\omega_1<\omega<b,\\
	0,
	& \text{otherwise}.
\end{cases}
\]

Then \(K_{a,b}\in L^\infty((a,b)^2)\), \(K_{a,b}\ge0\), and
\[
\iint_{(a,b)^2}
\mathrm{d}\omega\,\mathrm{d}\omega_1\,
K_{a,b}(\omega,\omega_1) f_n(s,\omega)f_n(s,\omega_1)
\le
\mathfrak W(t_n+s)
\]
for every \(s\in[0,T]\).

We claim that
\begin{equation}\label{eq:time-averaged-liminf}
	\begin{aligned}
		&\int_0^T \mathrm{d}s
		\iint_{(a,b)^2}
		\mathrm{d}\omega\,\mathrm{d}\omega_1\,
		K_{a,b}(\omega,\omega_1) F(s,\omega)F(s,\omega_1)
		\\
		&\qquad\le
		\liminf_{n\to\infty}
		\int_0^T \mathrm{d}s
		\iint_{(a,b)^2}
		\mathrm{d}\omega\,\mathrm{d}\omega_1\,
		K_{a,b}(\omega,\omega_1) f_n(s,\omega)f_n(s,\omega_1).
	\end{aligned}
\end{equation}

To prove this claim, we choose a sequence
\(\{K_m\}_{m\ge1}\subset C([a,b]^2)\) such that
\[
0\le K_m\le K_{a,b}
\qquad\text{a.e. on }(a,b)^2,
\]
and
\[
	\lim_{m\to\infty}K_m= K_{a,b}
\qquad\text{a.e. on }(a,b)^2.
\]
We now fix \(m\). Since \(K_m\) is continuous and bounded on \([a,b]^2\), by the Stone-Weierstrass theorem, there exist functions
\[
\Phi_{m,\ell}(\omega,\omega_1)
=
\sum_{j=1}^{N_{m,\ell}}
\phi_j^{(m,\ell)}(\omega)\psi_j^{(m,\ell)}(\omega_1),
\]
with \(\phi_j^{(m,\ell)},\psi_j^{(m,\ell)}\in C([a,b])\), such that
\[
\|\Phi_{m,\ell}-K_m\|_{L^\infty((a,b)^2)}\to0
\qquad\text{as } \ell\to\infty.
\]

For each fixed \(\ell\), by the convergence
\[
f_n\to F
\quad\text{in } C\bigl([0,T];w\text{-}L^1_{\mathrm{loc}}((0,\infty))\bigr),
\]
we get
\[
	\lim_{n\to\infty}\int_a^b \mathrm{d}\omega\,
\phi_j^{(m,\ell)}(\omega)f_n(s,\omega)
=
\int_a^b \mathrm{d}\omega\,
\phi_j^{(m,\ell)}(\omega)F(s,\omega)
\]
uniformly with respect to \(s\in[0,T]\), and similarly
\[
	\lim_{n\to\infty}\int_a^b \mathrm{d}\omega\,
\psi_j^{(m,\ell)}(\omega)f_n(s,\omega)
=
\int_a^b \mathrm{d}\omega\,
\psi_j^{(m,\ell)}(\omega)F(s,\omega)
\]
uniformly with respect to \(s\in[0,T]\). Hence,
\[
\begin{aligned}
	&	\lim_{n\to\infty}\int_0^T \mathrm{d}s
	\iint_{(a,b)^2}
	\mathrm{d}\omega\,\mathrm{d}\omega_1\,
	\Phi_{m,\ell}(\omega,\omega_1)f_n(s,\omega)f_n(s,\omega_1)
	\\
	&\qquad=
	\int_0^T \mathrm{d}s
	\iint_{(a,b)^2}
	\mathrm{d}\omega\,\mathrm{d}\omega_1\,
	\Phi_{m,\ell}(\omega,\omega_1)F(s,\omega)F(s,\omega_1).
\end{aligned}
\]

Also, we observe that
\[
\sup_{n\ge1}\sup_{s\in[0,T]}
\int_a^b \mathrm{d}\omega\, f_n(s,\omega)<\infty,
\qquad
\sup_{s\in[0,T]}
\int_a^b \mathrm{d}\omega\, F(s,\omega)<\infty.
\]
Therefore,
\[
\begin{aligned}
	&\left|
	\int_0^T \mathrm{d}s
	\iint_{(a,b)^2}
	\mathrm{d}\omega\,\mathrm{d}\omega_1\,
	(\Phi_{m,\ell}-K_m)(\omega,\omega_1)f_n(s,\omega)f_n(s,\omega_1)
	\right|
	\\
	&\qquad\le
	T\|\Phi_{m,\ell}-K_m\|_{L^\infty}
	\left(
	\sup_{n\ge1}\sup_{s\in[0,T]}
	\int_a^b \mathrm{d}\omega\, f_n(s,\omega)
	\right)^2,
\end{aligned}
\]
and similarly
\[
\begin{aligned}
	&\left|
	\int_0^T \mathrm{d}s
	\iint_{(a,b)^2}
	\mathrm{d}\omega\,\mathrm{d}\omega_1\,
	(\Phi_{m,\ell}-K_m)(\omega,\omega_1)F(s,\omega)F(s,\omega_1)
	\right|
	\\
	&\qquad\le
	T\|\Phi_{m,\ell}-K_m\|_{L^\infty}
	\left(
	\sup_{s\in[0,T]}
	\int_a^b \mathrm{d}\omega\, F(s,\omega)
	\right)^2.
\end{aligned}
\]
Letting \(\ell\to\infty\), we obtain
\[
\begin{aligned}
	&\int_0^T \mathrm{d}s
	\iint_{(a,b)^2}
	\mathrm{d}\omega\,\mathrm{d}\omega_1\,
	K_m(\omega,\omega_1)f_n(s,\omega)f_n(s,\omega_1)
	\\
	&\qquad\longrightarrow
	\int_0^T \mathrm{d}s
	\iint_{(a,b)^2}
	\mathrm{d}\omega\,\mathrm{d}\omega_1\,
	K_m(\omega,\omega_1)F(s,\omega)F(s,\omega_1) \mbox{ as } n\to\infty.
\end{aligned}
\]

Since \(0\le K_m\le K_{a,b}\), we get
\[
\begin{aligned}
	&\int_0^T \mathrm{d}s
	\iint_{(a,b)^2}
	\mathrm{d}\omega\,\mathrm{d}\omega_1\,
	K_m(\omega,\omega_1)F(s,\omega)F(s,\omega_1)
	\\
	&\qquad\le
	\liminf_{n\to\infty}
	\int_0^T \mathrm{d}s
	\iint_{(a,b)^2}
	\mathrm{d}\omega\,\mathrm{d}\omega_1\,
	K_{a,b}(\omega,\omega_1)f_n(s,\omega)f_n(s,\omega_1).
\end{aligned}
\]
Letting \(m\to\infty\) and using the dominated convergence theorem on the left-hand side, we obtain \eqref{eq:time-averaged-liminf}.
Combining \eqref{eq:time-averaged-liminf} with the estimate by \(\mathfrak W(t_n+s)\), we find
\[
\begin{aligned}
	&\int_0^T \mathrm{d}s
	\int_a^b \mathrm{d}\omega
	\int_a^\omega \mathrm{d}\omega_1\,
	U(\omega_1)U(\omega-\omega_1)F(s,\omega)F(s,\omega_1)
	\\
	&\qquad\le
	\liminf_{n\to\infty}
	\int_0^T \mathrm{d}s\,\mathfrak W(t_n+s)
	=0.
\end{aligned}
\]
Therefore,
\[
\int_a^b \mathrm{d}\omega
\int_a^\omega \mathrm{d}\omega_1\,
U(\omega_1)U(\omega-\omega_1)F(s,\omega)F(s,\omega_1)
=0
\]
for a.e. \(s\in[0,T]\). Letting \(a\downarrow0\) and \(b\uparrow\infty\), the monotone convergence theorem gives
\begin{equation}\label{eq:F-zero-dissipation-ae}
	\int_0^\infty \mathrm{d}\omega
	\int_0^\omega \mathrm{d}\omega_1\,
	U(\omega_1)U(\omega-\omega_1)F(s,\omega)F(s,\omega_1)
	=0
\end{equation}
for a.e. \(s\in[0,T]\).

We claim that \eqref{eq:F-zero-dissipation-ae} implies
\(
F(s,\cdot)=0
\text{ a.e. on }(0,\infty)
\)
for a.e. \(s\in[0,T]\).
Indeed, fix such a time \(s\). Suppose by contradiction that \(F(s,\cdot)\not\equiv0\). Since
\(F(s,\cdot)\ge0\) and \(F(s,\cdot)\in L^1_{\mathrm{loc}}((0,\infty))\), there exists an interval
\((a,b)\subset(0,\infty)\) such that the set
\[
\mathcal S:=\{\omega\in(a,b):F(s,\omega)>0\}
\]
has positive measure. We choose \(c\in(a,b)\) such that both
\[
\mathcal S_1:=\mathcal S\cap(a,c),
\qquad
\mathcal S_2:=\mathcal S\cap(c,b)
\]
have positive measure. Then the set
\[
\mathcal S_3:=\{(\omega,\omega_1):\omega\in\mathcal S_2,\ \omega_1\in\mathcal S_1\}
\]
has positive measure and satisfies \(0<\omega_1<\omega\) for every
\((\omega,\omega_1)\in\mathcal S_3\). For every
\((\omega,\omega_1)\in\mathcal S_3\), we have
\[
\omega_1>0,\qquad \omega-\omega_1>0,
\qquad
F(s,\omega)>0,\qquad F(s,\omega_1)>0.
\]
Since \(U(\xi)>0\) for every \(\xi>0\), it follows that
\[
U(\omega_1)U(\omega-\omega_1)F(s,\omega)F(s,\omega_1)>0
\]
on a set of positive measure, contradicting \eqref{eq:F-zero-dissipation-ae}. Hence
\[
F(s,\cdot)=0
\qquad\text{a.e. on }(0,\infty)
\]
for a.e. \(s\in[0,T]\).

Finally, since
\(
F\in C\bigl([0,T];w\text{-}L^1_{\mathrm{loc}}((0,\infty))\bigr),
\)
for every \(\psi\in C_c((0,\infty))\), the map
\[
s\mapsto \int_0^\infty \mathrm{d}\omega\,\psi(\omega)F(s,\omega)
\]
is continuous on \([0,T]\). Since \(F(s,\cdot)=0\) for a.e. \(s\in[0,T]\), this continuous function vanishes for a.e. \(s\), and therefore vanishes for every \(s\in[0,T]\). Hence
\[
F(s,\cdot)=0
\qquad\text{in } w\text{-}L^1_{\mathrm{loc}}((0,\infty))
\qquad\text{for every } s\in[0,T].
\]
In particular,
\[
f_\infty=F(0)=0.
\]

We have shown that every sequence \(t_n\to\infty\) admits a subsequence such that
\[
f(t_n,\cdot)\rightharpoonup 0
\qquad\text{in } w\text{-}L^1_{\mathrm{loc}}((0,\infty)).
\]
Therefore, we obtain the convergence
\[
f(t,\cdot)\rightharpoonup 0
\qquad\text{in } w\text{-}L^1_{\mathrm{loc}}((0,\infty))
\qquad\text{as } t\to\infty.
\]

\medskip
\noindent
\textit{Step 3: Strong local convergence.}

We fix \(0<a<b<\infty\). Then,  we find
	\[
	f(t,\cdot)\rightharpoonup 0
	\qquad\text{in } w\text{-}L^1((a,b))
	\quad\text{as } t\to\infty.
	\]
	Since the constant function \(1\) belongs to \(L^\infty((a,b))\), it may be used as a test function in the weak convergence. Hence, we find
	\[
	\lim_{t\to\infty}\int_a^b \mathrm{d}\omega\, f(t,\omega)= 0.
	\]
	Since \(f(t,\omega)\ge0\), we have	\[
	\lim_{t\to\infty}\|f(t,\cdot)\|_{L^1((a,b))}= 0.
	\]
	Thus, we obtain the convergence
	\[
	\lim_{t\to\infty} f(t,\cdot)= 0
	\qquad\text{strongly in }L^1((a,b)).
	\]
	Since \(0<a<b<\infty\) was chosen arbitrarily, it follows that
	\[
	\lim_{t\to\infty}f(t,\cdot)= 0
	\qquad\text{in }L^1_{\mathrm{loc}}((0,\infty)).
	\]
	This completes the proof.
\end{proof}

\section{Proof of Theorem \ref{thm}}
The conclusions of the Theorem follow from Propositions \ref{Prop:Existence:U=x:Hphi}, \ref{prop:global-Ux2-Hphi},  \ref{prop:equilibrium}.

\section{Data Availability Statement}
No datasets were generated or analyzed during the current study.

\section{Conflict of Interest}

The authors declare that they have no conflicts of interest.
\bibliographystyle{plain}

\bibliography{WaveTurbulence}

 \end{document}